\documentclass[11pt,reqno]{amsart}
\usepackage[T1]{fontenc}
\usepackage{amssymb,amsmath}
\usepackage[english]{babel}
\usepackage{enumerate}
\usepackage{enumitem,kantlipsum}
\usepackage[colorlinks=true,linkcolor=blue]{hyperref}
\usepackage{dsfont}
\usepackage{cite}
\usepackage[foot]{amsaddr}
\usepackage{nicefrac}

\newtheorem{theorem}{Theorem}[section]
\newtheorem{lemma}[theorem]{Lemma}
\newtheorem{corollary}[theorem]{Corollary}
\theoremstyle{definition}
\newtheorem{definition}[theorem]{Definition}

\newtheorem{remark}[theorem]{Remark}
\newtheorem{assumption}[theorem]{Assumption}

\setlist[enumerate]{leftmargin=*, label=(\roman*)}
\numberwithin{equation}{section}

\def\Cb{{\rm C}_{\rm b}}
\def\Cbi{{\rm C}_{\rm b}^\infty}
\def\Cci{{\rm C}_{\rm c}^\infty}
\def\Ck{{\rm C}_\kappa}

\def\Fk{{\rm F}_\kappa}
\def\Lipb{{\rm Lip}_{\rm b}}
\def\UCK{{\rm UC}_{\kappa}}

\def\LSplus{\mathcal{L}_S^+}

\def\Rd{\R^d}

\DeclareMathOperator{\ca}{ca}
\DeclareMathOperator{\id}{id}
\DeclareMathOperator{\supp}{supp}
\DeclareMathOperator{\tr}{tr}

\def\A{\mathcal{A}}
\def\B{\mathcal{B}}

\def\d{{\rm d}}
\def\D{\mathcal{D}}
\def\e{\mathcal{E}}
\def\E{\mathbb{E}}
\def\F{\mathcal{F}}
\def\H{\mathcal{H}}

\def\N{\mathbb{N}}

\def\P{\mathbb{P}}
\def\R{\mathbb{R}}
\def\S{\mathbb{S}}
\def\T{\mathcal{T}}

\def\epsilon{\varepsilon}

\def\one{\mathds{1}}

\begin{document}

\vspace*{-0.2cm}
\title[Stability of convex monotone semigroups]{Convergence of infinitesimal generators and stability of convex monotone semigroups}

\author{Jonas Blessing$^1$}
\address{{\rm $^1$Department of Mathematics, ETH Zurich, 8092 Zurich, Switzerland}}
\email{$^1$jonas.blessing@math.ethz.ch}

\author{Michael Kupper$^2$}
\address{{\rm $^2$Department of Mathematics and Statistics,  University of Konstanz, 
    78457 Konstanz, Germany}}
\email{$^2$kupper@uni-konstanz.de}

\author{Max Nendel$^3$}
\address{{\rm $^3$Department of Statistics and Actuarial Science, University of Waterloo,
    ON N2L 3G1 Waterloo, Canada}}
\email{$^3$mnendel@uwaterloo.ca}

\date{\today}

\thanks{Financial support through the Deutsche Forschungsgemeinschaft 
(DFG, German Research Foundation) -- SFB 1283/2 2021 – 317210226 is 
gratefully acknowledged.}

\begin{abstract}
Based on the convergence of their infinitesimal generators in the mixed topology, we 
provide a stability result for strongly continuous convex monotone semigroups on spaces 
of continuous functions.\ In contrast to previous results, we do not rely on the theory 
of viscosity solutions but use a recently established comparison principle which uniquely 
determines semigroups via their upper $\Gamma$-generator defined on their upper Lipschitz 
sets and therefore resembles the classical analogue from the linear case.\ The framework 
also allows for discretizations both in time and space and covers a variety of applications.\
This includes Euler schemes and Yosida-type approximations for upper envelopes of families 
of linear semigroups, stability results and finite-difference schemes for convex HJB equations, 
Freidlin--Wentzell-type results and Markov chain approximations for a class of stochastic optimal
control problems. 

\smallskip
\textit{Key words:} convex monotone semigroup, infinitesimal generator, convergence of semigroups, 
Euler formula, optimal control, finite-difference scheme, Markov chain approximation, large deviations
\smallskip\\
\textit{MSC 2020:} Primary 35B35; 47H20; Secondary 47H07; 49M25; 60G65
\end{abstract}

\maketitle 

\vspace*{-1cm}

\thispagestyle{empty}

\section{Introduction}

In this article, we consider sequences $(S_n)_{n\in\N}$ of convex monotone semigroups 
on spaces of continuous functions and give explicit conditions under which convergence
of their infinitesimal generators $(A_n)_{n\in\N}$ guarantees convergence of the semigroups.\
In the linear case, such results are classical, cf.\ Kurtz~\cite{Kurtz70} and Trotter~\cite{Trotter}, 
and can be applied, for instance, to obtain convergence results for Markov processes, 
see, e.g.,~Ethier and Kurtz~\cite{EK}, Kertz~\cite{Kertz74,Kertz78} and Kurtz~\cite{Kurtz75}.\ 
While~\cite{Kurtz70} and~\cite{EK} work on locally convex spaces and Banach spaces assuming 
conditions on the Laplace transform of the semigroups or even explicitly the existence of a 
limit semigroup, respectively, we restrict our analysis to a particular class of semigroups 
but only assume that the generators converge for smooth functions.\ Based on the Crandall--Liggett 
theorem, cf.\ Crandall and Liggett~\cite{CL71}, the previously mentioned results can be 
extended to nonlinear semigroups which are generated by m-accretive or dissipative operators, 
see~Br\'{e}zis and Pazy~\cite{BP} and Kurtz~\cite{Kurtz73}.\ While this approach closely 
resembles the theory of linear semigroups, the definition of the nonlinear resolvent typically 
requires the existence of a unique classical solution of a corresponding fully nonlinear 
elliptic partial differential equation (PDE). As pointed out in Evans~\cite{Evans87} and 
Feng and Kurtz~\cite{FK}, the necessary regularity of classical solutions is, in general, 
delicate.\ This observation was, among others, one of the motivations for the introduction 
of viscosity solutions, cf.\ Crandall et al.~\cite{CIL} and Crandall and Lions~\cite{CL83}.

In contrast to classical solutions, the latter have the stability property that, under mild
conditions, limits of viscosity solutions are again viscosity solutions.\ A prominent example
is the vanishing viscosity method, where smooth solutions of semilinear second order equations
converge to the unique viscosity solution of a fully nonlinear first order equation.\
Moreover, the Barles--Perthame method~\cite{BP88} establishes the convergence of viscosity 
sub- or super solutions using semi-relaxed limits and a comparison principle for the limit.
Another remarkable stability property of viscosity solutions is that any stable and consistent
monotone approximation scheme converges to the exact solution if the latter is unique, 
cf.\ Barles and Souganidis~\cite{BS91}.\ Based on this result, it is possible to derive 
explicit convergence rates for several numerical schemes for Hamilton-Jacobi-Bellman (HJB) 
equations, see, e.g.,~Barles and Jakobsen~\cite{BJ05, BJ07}, Briani et al.~\cite{BCZ12}, 
Caffarelli and Souganidis~\cite{CS08}, Jakobsen et al.~\cite{Jakobsen2019} and 
Krylov~\cite{Krylov98, Krylov99, Krylov05}. Another prominent numerical method for 
stochastic optimal control problems is the Markov chain approximation by Dupuis and 
Kushner~\cite{DK01}, where controlled diffusion processes are approximated by 
discrete-time controlled Markov chains.

In the context of stochastic optimal control, the value function is typically the unique 
viscosity solution of an HJB equation which can be seen as nonlinear evolution equation 
related to the generator of a convex monotone semigroup.\ For sublinear Markov semigroups, 
Kühn~\cite{kuhn2018viscosity} and Hu and Peng~\cite{MR4251961} explicitly characterize 
the generator as supremum of linear operators.\ We also refer to Neufeld and Nutz~\cite{neufeld2017nonlinear}, 
Nendel and R\"ockner~\cite{NR21} and Criens and Niemann~\cite{Criens, Criens25a, Criens25b} 
for the link between sublinear semigroups and value functions of stochastic optimal control problems.

Returning to the convergence of nonlinear semigroups, we briefly discuss the results from
Feng and Kurtz~\cite{FK} which are motivated by the large deviations principle for sequences
of Markov processes.\ The approach in~\cite{FK} combines results on dissipative operators
with key arguments from viscosity theory and has been applied and extended in Feng~\cite{Feng}, 
Feng et al.~\cite{FFK}, Kraaij~\cite{Kraaij19,Kraaij22} and Popovic~\cite{Popovic}.\
Here, the resolvent equation $(\id-\lambda A)u=f$ for the limit operator $Au:=\lim_{n\to\infty}A_n u$ 
is solved in the viscosity sense and needs to satisfy a comparison principle which does 
not hold a priori but has to be verified on a case-by-case basis.\ Since the existence of a solution is 
always guaranteed, the limit semigroup can then be constructed via the Euler formula 
$S(t)f:=\lim_{n\to\infty}R(\nicefrac{t}{n})^n f$, where $R(\lambda)f$ denotes the unique 
viscosity solution of the resolvent equation.\ By construction, the semigroup $(S(t))_{t\geq 0}$ 
is generated by the operator $A$ and one can show that $S(t)f=\lim_{n\to\infty}S_n(t)f$. 

In contrast to the previously mentioned results, the present approach neither relies on 
the existence of nonlinear resolvents nor the theory of viscosity solutions.\ Instead, 
the key arguments are based on a recently established comparison principle for strongly 
continuous convex monotone semigroups, cf.\ Blessing et al.~\cite{BDKN} and Blessing and 
Kupper~\cite{BK22}, which uniquely determines semigroups via their $\Gamma$-generators. 
Moreover, due to an approximation procedure, it is usually sufficient to determine the 
generators for smooth functions.\ These results allow us to proceed as follows.\ First, 
the limit semigroup is defined as $S(t)f:=\lim_{l\to\infty}S_{n_l}(t)f$ for all $(t,f)$ 
in a countable dense set, where the convergence for a subsequence $(n_l)_{l\in\N}$ is
guaranteed by a relative compactness argument.\ Since we only require convergence with
respect to (w.r.t.)\ the mixed topology rather than convergence w.r.t.\ the supremum norm, 
the latter can be verified by means of Arzel\`a-Ascoli's theorem.\ After an extension to 
arbitrary $(t,f)$, we then show that the generator of $(S(t))_{t\geq 0}$ is given by 
$Af=\lim_{n\to\infty}A_n f$ for smooth functions~$f$.\ At this point we would like to 
emphasize that the semigroups $(S_n)_{n\in\N}$ satisfy the comparison principle 
from~\cite{BDKN,BK22} which transfers to $(S(t))_{t\geq 0}$.\ In particular, the limit 
semigroup does not depend on the choice of the convergent subsequence $(n_l)_{l\in\N}$ 
and therefore satisfies $S(t)f=\lim_{n\to\infty}S_n(t)f$.\ This stability result for 
convex monotone semigroups is stated in Theorem~\ref{thm:Sn} and Theorem~\ref{thm:unique}.\ 
Our results also cover a variety of approximation schemes of the form $S(t)f=\lim_{n\to\infty}I_n^{k_n}f$,
where $(I_n)_{n\in\N}$ is a family of one-step operators describing the dynamics on
a discrete time scale of size $h_n>0$ with $h_n\to 0$ and $k_nh_n\to t$, see Theorem~\ref{thm:cher}.\
This extension of the classical Chernoff approximation, cf.\ Chernoff~\cite{chernoff68,chernoff74}, 
includes finite-difference methods for HJB equations, cf.\ Barles and Jakobsen~\cite{BJ07},
Bonnans and Zidani~\cite{BZ03} and Krylov~\cite{Krylov98,Krylov05}, and Markov chain 
approximations for stochastic control problems, cf.\ Dupuis and Kushner~\cite{DK01} 
and Fleming and Soner~\cite{FS}.\ Moreover, Theorem~\ref{thm:cher} allows to construct 
nonlinear semigroups without relying on the existence of nonlinear resolvents, e.g.,
in the context of upper envelopes of families of linear semigroups, cf.\ Denk et al.~\cite{DKN20},
Nendel and R\"ockner~\cite{NR21} and Nisio~\cite{Nisio76}.\ Finally, we remark that our 
notion of a strongly continuous convex monotone semigroup coincides with the one in Goldys
et al.~\cite{GNR}, where it is shown that $u(t):=S(t)f$ is a viscosity solution of the 
abstract Cauchy problem $\partial_t u=Au$ with initial condition $u(0)=f$.\ Our convergence 
result for semigroups can therefore be seen as an analogue of the statement that limits of
viscosity solutions are again viscosity solutions and Chernoff-type approximations are
the analogue of monotone schemes.\ We refer to Fleming and Soner~\cite[Chapter~II.3]{FS} 
for a broad discussion on the relation between semigroups and viscosity solutions and 
to Yong and Zhou~\cite[Chapter~4]{MR1696772} for an illustration of the interplay between 
dynamic programming and viscosity solutions in a stochastic optimal control setting.\ 
Convex monotone semigroups also appear in the context of stochastic processes under
model uncertainty, cf.\ Coquet et al.~\cite{MR1906435}, Criens and Niemann~\cite{Criens}, 
Fadina et al.~\cite{MR3955321}, Hu and Peng~\cite{MR4251961}, Krak et al.~\cite{Krak2017}, 
K\"uhn~\cite{kuhn2018viscosity}, Neufeld and Nutz~\cite{neufeld2017nonlinear} and Peng~\cite{Peng19}.\
In many situations, these semigroups admit a stochastic representation via backward stochastic
differential equations (BSDEs) and second order backward stochastic differential equations 
(2BSDEs), cf.\ Cheridito et al.~\cite{MR2319056}, El Karoui et al.~\cite{MR1434407}, 
Kazi-Tani et al.~\cite{MR3361253} and Soner et al.~\cite{Soner2012}. For stability 
and approximation results for BSDEs, we refer to Briand et al.~\cite{Briand2002},
Hu and Peng~\cite{Hu1997}, Geiss et al.~\cite{Geiss2020} and Papapantoleon et al.~\cite{Papapantoleon2023}. 

The abstract results are illustrated in a variety of applications which we briefly discuss here.\
In Subsection~\ref{sec:yosida}, we derive an implicit Euler formula and a Yosida approximation
for upper envelopes of families of linear semigroups.\ Here, we would like to emphasize that 
difficulties in defining the resolvent of the supremum of linear operators are avoided by 
considering the supremum over linear resolvents rather than the resolvent of a nonlinear operator.\
We refer to Budde and Farkas~\cite{bff},  Cerrai~\cite{MR1293091}, K\"{u}hnemund~\cite{kuhnemund} 
and Pazy~\cite{pazy} for the corresponding results on linear semigroups in different settings.\
In Subsection~\ref{sec:VanVis}, we study the stability of convex HJB equations, cf.\ Barron and 
Jensen~\cite{MR1080619,MR1076572}, Crandall and Lions~\cite{CL83}, Frankowska~\cite{MR1200233} 
and Kraaij~\cite{Kraaij22}.\ In Subsection~\ref{sec:LargeDev}, we consider explicit Euler schemes 
for Lipschitz ordinary differential equations (ODEs) with additive noise, where the distribution 
of the noise might be uncertain.\ Depending on the scaling of the noise this can either be seen 
as a robustness result which states that, regardless of possible numerical errors, the Euler 
scheme still converges to the solution of the ODE or as an Euler--Maruyama scheme  for stochastic 
differential equations (SDEs) driven by G-Brownian motions, cf.\ Geng et al.~\cite[Section 3]{Geng2014}, 
Hu et al.~\cite{HJL21} and Peng~\cite{Peng08,Peng19}.\ The analysis of randomized Euler schemes 
is continued by focusing on asymptotic convergence rates by means of a large deviations approach, 
cf.\ Dupuis and Ellis~\cite{DE}, Dembo and Zeitouni~\cite{DZ2010}, Feng and Kurtz~\cite{FK} 
and Varadhan~\cite{Varadhan1984}, leading to Freidlin--Wentzell-type results, cf.\ Feng and
Kurtz~\cite{FK} and Freidlin and Wentzell~\cite{FW}.\ Finally, in Subsection~\ref{sec:control},
we provide Markov chain approximations for a class of stochastic control problems, cf.\ Dupuis 
and Kushner~\cite{DK01}, Fleming and Soner~\cite{FS} and Krylov~\cite{Krylov99}.

\section{Setup and main results}
\label{sec:main}

\subsection{Basic notation and the mixed topology}

Let $\kappa\colon\Rd\to (0,\infty)$ be a bounded continuous function satisfying
\begin{equation} \label{eq:kappa}
 c_\kappa:=\sup_{x\in\Rd}\sup_{|y|\leq 1}\frac{\kappa(x)}{\kappa(x-y)}<\infty.
\end{equation} 
For every $X\subset\Rd$, let $\Ck(X)$ be the space of all continuous 
functions $f\colon X\to\R$ with $\|f\|_{\kappa,X}:=\sup_{x\in X}|f(x)|\kappa(x)<\infty$.\
If $X=\Rd$, we write $\|\cdot\|_\kappa:=\|\cdot\|_{\kappa,X}$.\ The weight function~$\kappa$ 
allows to define semigroups on spaces of unbounded functions such that, for instance, 
in the context of stochastic optimal control or financial applications, unbounded cost 
or payoff functions can be treated, respectively.

An operator $S\colon\Ck(X)\to\Ck(X)$ is monotone if $Sf\leq Sg$ for all $f,g\in\Ck(X)$ with 
$f\leq g$ and convex if $S(\lambda f+(1-\lambda)g)\leq\lambda Sf+(1-\lambda)Sg$ for all
$f,g\in\Ck(X)$ and $\lambda\in [0,1]$, where $f\leq g$ means that $f(x)\leq g(x)$ for all 
$x\in X$.\ For an operator $S\colon\Ck(X)\to\Ck(X)$ and $f\in \Ck(\R^d)$, we define $Sf:=S(f|_X)$.\
Moreover, we write $f_n\downarrow f$ if a sequence of functions $f_n\colon X\to\R$ decreases 
pointwise to another function $f\colon X\to\R$. Denoting by $|\cdot|$ the Euclidean distance,
we define
\[ B_{\Ck(X)}(r):=\{f\in\Ck(X)\colon\|f\|_\kappa\leq r\} \quad\mbox{and}\quad
B_{X}(r):=\{x\in X\colon |x|\leq r\}. \]

Subsequently, we endow $\Ck(X)$ with the mixed topology between $\|\cdot\|_{\kappa,X}$ and the 
topology of uniform convergence on compact sets, i.e., the strongest locally convex topology 
on $\Ck(X)$ that coincides on $\|\cdot\|_{\kappa,X}$-bounded sets with the topology of uniform 
convergence on compact subsets, see~\cite{wiweger} for the definition and basic properties of 
the mixed topology.\ The mixed topology plays a crucial role in our analysis for several reasons.\
First, since equation~\eqref{eq:Sn} below is derived from Arzel\`a--Ascoli's theorem, we can 
only expect uniform convergence on compact subsets but not $\|\cdot\|_\kappa$-convergence.\ 
Second, the results in Section 2.3 rely on a comparison principle in~\cite{BDKN} which uses 
the mixed topology.\ Third, the dual space of $\Ck(\Rd)$ endowed with the mixed topology 
coincides with the space of all Borel measures with finite variation for which $1/\kappa$ 
is integrable, see~\cite{GNR} and the references therein.\ Fourth, while the mixed topology
is not metrizable, for convex monotone operators and closed subsets $X\subset\Rd$, continuity
w.r.t.\ the mixed topology is equivalent to continuity from above and sequential continuity,
see~\cite{delbaen2,nendel} for equivalent descriptions of continuity properties for monotone 
functionals and operators in the mixed topology.\ In addition, convergence of sequences in 
the mixed topology can be characterized as follows:\ a sequence $(f_n)_{n\in\N}\subset\Ck(X)$ 
converges to $f\in\Ck(X)$ w.r.t. the mixed topology if and only if 
\begin{equation} \label{eq:seqcov}
 \sup_{n\in\N}\|f_n\|_{\kappa,X}<\infty \quad\mbox{and}\quad \lim_{n\to\infty}\|f-f_n\|_{\infty,K}=0
\end{equation}
for all compact subsets $K\Subset X$, where $\|f\|_{\infty,K}:=\sup_{x\in K}|f(x)|$. 
For a proof, we refer to~\cite[Proposition~A.4]{GNR}. Subsequently, if not stated otherwise, 
all limits in $\Ck(X)$ are taken w.r.t.\ the mixed topology and compact subsets are denoted 
by $K\Subset X$.\ Finally, the mixed topology belongs to the class of strict topologies, 
see~\cite{Haydon,Kunze,Sentilles}.

\subsection{Stability of convex monotone semigroups}

Let $\R_+:=\{x\in\R\colon x\geq 0\}$.

\begin{definition}
 Let $\T\subset\R_+$ satisfying $0\in\T$ and $s+t\in\T$ for all $s,t\in\T$ and $X\subset\Rd$.\
 A convex monotone semigroup on $\Ck(X)$ is a family $(S(t))_{t\in\T}$ of operators 
 $S(t)\colon\Ck(X)\to\Ck(X)$ such that
 \begin{enumerate}
  \item[(i)] $S(t)$ is convex and monotone with $S(t)f_n\downarrow 0$ for all $t\in \T$ and $f_n\downarrow 0$,
  \item[(ii)] $S(0)f=f$ and $S(s+t)f=S(s)S(t)f$ for all $s,t\in\T$ and $f\in\Ck(X)$,
  \item[(iii)] $\sup_{t\in [0,T]\cap\T}\|S(t)\frac{r}{\kappa}\|_{\kappa,X}<\infty$ for all $r,T\geq 0$.
 \end{enumerate}
 Moreover, the semigroup is called strongly continuous if $\T=\R_+$ and
 $f=\lim_{t\downarrow 0}S(t)f$ for all $f\in\Ck(X)$. In this case, 
 the generator of the semigroup is defined by
 \[ A\colon D(A)\to\Ck(X),\; f\mapsto\lim_{h\downarrow 0}\frac{S(h)f-f}{h}, \]
 where the domain $D(A)$ consists of all $f\in\Ck(X)$ such that the previous limit exists.
\end{definition}

Due to the convexity, the continuity from above at zero extends to continuity from above at
all points, i.e., it holds $S(t)f_n\downarrow S(t)f$ for all $(f_n)_{n\in\N}\subset\Ck(X)$
and $f\in\Ck(X)$ with $f_n\downarrow f$.\ In particular, we obtain $S(t)0=0$ for all $t\geq 0$.\ In addition, for every $\epsilon>0$, 
$r,T\geq 0$ and $K\Subset X$, there exist $c\geq 0$ and $K'\Subset X$ with 
\[ \|S(t)f-S(t)g\|_{\infty,K}\leq c\|f-g\|_{\infty,K'}+\epsilon \]
for all $t\in [0,T]\cap\T$ and $f,g\in B_{\Ck(X)}(r)$.\ This follows from the continuity from above,
Dini's theorem and~\cite[Lemma~C.2]{BDKN}.
Moreover, if $(S(t))_{t\geq 0}$ is strongly continuous, it follows from~\cite[Lemma~3.2]{BDKN}
that $S(t)f=\lim_{s\to t}S(s)f$ for all $t\geq 0$ and $f\in\Ck(X)$. We point out that our notion of a strongly continuous convex monotone semigroup coincides with~\cite[Definition 3.1]{GNR}. 
In particular, we refer to~\cite[Section 4]{GNR} for a series of examples in the linear case.

In order to provide a framework that covers discretizations in time and space,
we introduce the following sets and notation.\
Let $(\T_n)_{n\in\N}$ and $(X_n)_{n\in\N}$ be sequences of subsets $\T_n\subset\R_+$ 
and $X_n\subset\Rd$ such that, for every $t\geq 0$ and $x\in\Rd$, there exist 
$t_n\in\T_n$ and $x_n\in X_n$ with $t_n\to t$ and $x_n\to x$.\ Let $s\pm t\in\T_n$ for all $n\in\N$ and $s,t\in\T_n$ with $s\geq t$, which implies that $0\in \T_n$ for all $n\in \N$.\
If $h_n:=\inf(\T_n\setminus \{0\})>0$, it follows that $\{kh_n\colon k\in \N_0\}\subset\T_n$. In addition, the sets~$X_n$ 
are supposed to be closed and to satisfy $x+y\in X_n$ for all $n\in\N$ and $x,y\in X_n$.\
Typical examples for the sets~$\T_n$ and~$X_n$ are grids with mesh size tending to zero.
For a sequence $(f_n)_{n\in\N}$ with $f_n\in\Ck(X_n)$ and $f\in\Ck(\Rd)$, we define 
$f:=\lim_{n\to\infty}f_n$ if and only if 
\[ \sup_{n\in\N}\|f_n\|_{\kappa,X_n}<\infty \quad\mbox{and}\quad
\lim_{n\to\infty}\|f-f_n\|_{\infty, K_n}=0 \]
for all $K\Subset\Rd$ with $K\cap X_n\neq\emptyset$ for all $n\in\N$, where $K_n:=K\cap X_n$.\
The continuity of~$f$ guarantees that the previous limit is unique. 
In case that $X_n=\Rd$, this definition is consistent with the convergence of sequences in equation~\eqref{eq:seqcov}.\ A sequence $(f_n)_{n\in\N}$ with $f_n\in\Ck(X_n)$
is called uniformly equicontinuous, if for every $\epsilon>0$, there exists 
$\delta>0$ with $|f_n(x)-f_n(y)|<\epsilon$ for all $n\in\N$ and $x,y\in X_n$ with $|x-y|<\delta$.

\begin{definition} \label{def:Dinfty}
 Let $(S_n)_{n\in\N}$ be a sequence such that $(S_n(t))_{t\in\T_n}$ is a convex monotone 
 semigroup on $\Ck(X_n)$  for all $n\in\N$.\
 The asymptotic domain $\D_\infty$ consists of all $f\in\Ck(\Rd)$ such that there exist
 $g\in\Ck(\Rd)$ and $f_n\in\Ck(X_n)$ with $f_n\to f$ satisfying the following conditions:
 \begin{enumerate}
  \item[(i)] There exist $t_n\in\T_n\backslash\{0\}$ with $\sup_{n\in\N}t_n<\infty$ and
   \begin{equation} \label{eq:Linfty}
    \sup_{n\in\N}\sup_{t\in (0,t_n]\cap\T_n}\left\|\frac{S_n(t)f_n-f_n}{t}\right\|_{\kappa,X_n}< \infty. 
   \end{equation}
  \item[(ii)] For every $K\Subset\Rd$, there exist $h_n\in (0,t_n]\cap\T_n$ with $h_n\to 0$ and
   \begin{equation} \label{eq:Dinfty}
    \lim_{n\to\infty}\left\|\frac{S_n(h_n)f_n-f_n}{h_n}-g\right\|_{\infty,K_n}=0. 
   \end{equation}
 \end{enumerate}
\end{definition}

\begin{assumption} \label{ass:Sn}
 Let $(S_n)_{n\in\N}$ be a sequence such that $(S_n(t))_{t\in\T_n}$ is a convex monotone 
 semigroup on $\Ck(X_n)$ for all $n\in\N$ satisfying the following conditions:
 \begin{enumerate}
  \item[(i)] For every $r,T\geq 0$, there exists $c\geq 0$ with
   \[ \|S_n(t)f-S_n(t)g\|_{\kappa,X_n}\leq c\|f-g\|_{\kappa, X_n} \]
   for all $n\in\N$, $t\in [0,T]\cap\T_n$ and $f,g\in B_{\Ck(X_n)}(r)$. 
  \item[(ii)] For every $\epsilon>0$, $r,T\geq 0$ and $K\Subset\Rd$, there exist $c\geq 0$
   and $K'\Subset\Rd$ with
   \[ \|S_n(t)f-S_n(t)g\|_{\infty,K_n}\leq c\|f-g\|_{\infty,K'_n}+\epsilon \]
   for all $n\in\N$, $t\in [0,T]\cap\T_n$ and $f,g\in B_{\Ck(X_n)}(r)$. 
  \item[(iii)] There exist countable dense sets $\D_0\subset\Ck(\Rd)$ and $\T\subset\R_+$ 
   such that, for every $f\in\D_0$, there exist $\bar{f}_n\in\Ck(X_n)$ with $\bar{f}_n\to f$ 
   and $\|\bar{f}_n\|_\kappa\leq\|f\|_\kappa$ for all $n\in\N$
   satisfying inequality~\eqref{eq:Linfty} and, for every $t\in\T$, there exist $\bar{t}_n\in\T_n$ with $\bar{t}_n\to t$ such that 
   the sequence $(S_n(\bar{t}_n)\bar{f}_n)_{n\in\N}$ is uniformly equicontinuous.\ 
   In addition, for every $f\in\Ck(\Rd)$, there exists $(f_n)_{n\in\N}\subset\D_0$
   with $f_n\to f$ and $\|f_n\|_\kappa\leq\|f\|_\kappa$ for all $n\in\N$.
 \end{enumerate}
\end{assumption}

\begin{remark}\label{rem.cond.iii}
 Let $\Cci(\Rd)$ be the space of all infinitely differentiable functions $f\colon\Rd\to\R$ 
 with compact support.\ We point out that there exists a countable set $\D_0\subset \Cci(\R^d)$ 
 such that, for every $f\in\Ck(\Rd)$, there exists $(f_n)_{n\in\N}\subset\D_0$ with $f_n\to f$ 
 and $\|f_n\|_\kappa\leq\|f\|_\kappa$ for all $n\in\N$.\ Indeed, for every $n\in\N$, there exists a 
 countable set $\D_n\subset\Cci(\Rd)$ such that, for every continuous function $f\colon\Rd\to\R$ with
 $\supp(f)\subset B_{\Rd}(n)$, there exists $(f_k)_{k\in\N}\subset\D_n$ with $\|f-f_k\|_\infty\to 0$.\
 Moreover, there exist $(\phi_n)_{n\in\N}\subset\Cci(\Rd)$ with $\phi_n\equiv 1$ on $B_{\Rd}(n-1)$ 
 and $\supp(\phi_n)\subset B_{\Rd}(n)$ for all $n\in\N$.\ Let $f\in\Ck(\Rd)$ and define 
 $\tilde{f}_n:=(1-\frac{1}{n})f\phi_n$ for all $n\in\N$.\ For every $n\in\N$, we choose 
 $f_n\in\D_n$ with $\|f_n-\tilde{f}_n\|_\infty\leq\frac{\|f\|_\kappa}{n\|\kappa\|_\infty}$.  
 It follows that $f_n\to f$ and 
 \[ \|f_n\|_\kappa\leq\|\kappa\|_\infty\|f_n-\tilde{f}_n\|_\infty+\|\tilde{f}_n\|_\kappa
    \leq\|f\|_\kappa \quad\mbox{for all } n\in\N, \]
 which shows that $\D_0:=\bigcup_{n\in\N}\D_n\subset \Cci(\R^d)$ satisfies the second part of 
 condition~(iii).\ Hence, Assumption~\ref{ass:Sn}(iii) is satisfied if, for every $f\in\Cci(\R^d)$ 
 and $t\geq 0$, there exist $\bar{t}_n\in\T_n$ with $\bar{t}_n\to t$ and $\bar{f}_n\in\Ck(X_n)$ 
 with $\bar{f}_n\to f$ and $\|\bar{f}_n\|_\kappa\leq\|f\|_\kappa$ for all $n\in\N$ such that 
 inequality~\eqref{eq:Linfty} is valid and the sequence 
 $(S_n(\bar{t}_n)\bar{f}_n)_{n\in\N}$ is uniformly equicontinuous.
\end{remark}

\begin{theorem} \label{thm:Sn}
 Let $(S_n)_{n\in\N}$ be a sequence satisfying Assumption~\ref{ass:Sn}.
 Then, there exist a strongly continuous convex monotone semigroup $(S(t))_{t\geq 0}$ on $\Ck(\Rd)$
 with generator $A\colon D(A)\to\Ck(\Rd)$ and a subsequence $(n_l)_{l\in\N}\subset\N$ with
 \begin{equation} \label{eq:Sn}
  S(t)f=\lim_{l\to\infty}S_{n_l}(t_{n_l})f_{n_l} \quad\mbox{for all } (f,t)\in\Ck(\Rd)\times\R_+,
 \end{equation}
 where $t_n\in\T_n$ with $t_n\to t$ and $f_n\in\Ck(X_n)$ with $f_n\to f$ are arbitrary.\ In addition,
 the following statements are valid:
 \begin{enumerate} 
  \item[(i)] It holds $\D_\infty\subset D(A)$. Moreover, for every $f\in\D_\infty$ 
   and $K\Subset\Rd$, 
   \begin{equation} \label{eq:thm.main.1}
    \lim_{n\to\infty}\left\|\frac{S_n(h_n)f_n-f_n}{h_n}-Af\right\|_{\infty,K_n}=0,
   \end{equation}
   where $(f_n)_{n\in\N}$ and $(h_n)_{n\in\N}$ satisfy the conditions of Definition~\ref{def:Dinfty}.
  \item[(ii)] For every $r,T\geq 0$, there exists $c\geq 0$ with
   \[ \|S(t)f-S(t)g\|_\kappa\leq c\|f-g\|_\kappa 
    \quad\mbox{for all } t\in [0,T] \mbox{ and } f,g\in B_{\Ck(\Rd)}(r). \]
  \item[(iii)] For every $\epsilon>0$, $r,T\geq 0$ and $K\Subset\Rd$, there exist $c\geq 0$ 
   and $K'\Subset\Rd$ with
   \[ \|S(t)f-S(t)g\|_{\infty,K}\leq c\|f-g\|_{\infty,K'}+\epsilon \]
   for all $t\in [0,T]$ and $f,g\in B_{\Ck(\Rd)}(r)$. 
 \end{enumerate}
\end{theorem}
\begin{proof} 
 First, we construct the operators $S(t)\colon\Ck(\Rd)\to\Ck(\Rd)$ for all $t\geq 0$.\
 By Assumption~\ref{ass:Sn}(iii), there exist countable dense sets $\D_0\subset\Ck(\Rd)$ 
 and $\T\subset\R_+$ such that, for every $f\in\D_0$ and $t\in\T$, there exist $\bar{f}_n\in\Ck(X_n)$ 
 with $\bar{f}_n\to f$ and $\|\bar{f}_n\|_\kappa\leq\|f\|_\kappa$ for all $n\in\N$
 satisfying inequality~\eqref{eq:Linfty} and $\bar{t}_n\in\T_n$ with $\bar{t}_n\to t$
 such that the sequence $(S_n(\bar{t}_n)\bar{f}_n)_{n\in\N}$ is uniformly equicontinuous.\
 Since $(S_n(\bar{t}_n)\bar{f}_n)_{n\in\N}$ is bounded due to Assumption~\ref{ass:Sn}(i), 
 Lemma~\ref{lem:AA} and a diagonalization argument imply the existence of a subsequence 
 $(n_l)_{l\in\N}\subset\N$ such that the limit
 \begin{equation} \label{eq:Snl1}
  S(t)f:=\lim_{l\to\infty}S_{n_l}(\bar t_{n_l})\bar f_{n_l}\in\Ck(\Rd)
 \end{equation}
 exists for all $(f,t)\in\D_0\times\T$.\ In order to extend $(S(t))_{t\in\T}$ 
 to arbitrary points in time, we fix $f\in\D_0$ and $T\geq 0$.
 Let $s,t\in [0,T]\cap\T$ with $s<t$ and choose the corresponding $\bar{s}_n,\bar{t}_n\in\T_n$ 
 with $\bar{s}_n\to s$ and $\bar{t}_n\to t$ so that $\bar{s}_n\leq\bar{t}_n$ for all
 sufficiently large $n\in\N$.\ By Assumption~\ref{ass:Sn}(i) and inequality~\eqref{eq:Linfty}, 
 there exist $c_1, c_2\geq 0$ and $h_n\in\T_n\backslash\{0\}$ with
 \begin{align}
  &\|S_n(\bar{t}_n)\bar{f}_n-S_n(\bar{s}_n)\bar{f}_n\|_{\kappa,X_n} 
    =\|S_n(\bar{s}_n)S_n(\bar{t}_n-\bar{s}_n)\bar{f}_n-S_n(\bar{s}_n)\bar{f}_n\|_{\kappa,X_n} \nonumber \\
  &\leq c_2\|S_n(\bar{t}_n-\bar{s}_n)\bar{f}_n-\bar{f}_n\|_{\kappa,X_n} \nonumber \\
  &\leq c_2\|S_n(k_n h_n)S_n(\bar{t}_n-\bar{s}_n-k_n h_n)\bar{f}_n-S_n(k_n h_n)\bar{f}_n\|_{\kappa,X_n} \nonumber \\
  &\quad\; +c_2\sum_{i=0}^{k_n-1}\|S_n(ih_n)S_n(h_n)\bar{f}_n-S_n(ih_n)\bar{f}_n\|_{\kappa,X_n} \nonumber \\
  &\leq c_1c_2^2(\bar{t}_n-\bar{s}_n-k_nh_n)+c_2^2k_n\|S_n(h_n)\bar{f}_n-\bar{f}_n\|_{\kappa, X_n}
    \leq c(\bar{t}_n-\bar{s}_n), \label{eq:T0}
 \end{align}
 where $k_n:=\max\{k\in\N_0\colon kh_n\leq\bar{t}_n-\bar{s}_n\}$ and $c:=2c_1c_2^2$.\ 
 Equation~\eqref{eq:Snl1} implies
 \begin{equation} \label{eq:T1}
  \|S(s)f-S(t)f\|_\kappa\leq c|s-t| \quad\mbox{for all } s,t\in [0,T]\cap\T. 
 \end{equation}
 Let $t\in [0,T]$ and $(t_n)_{n\in\N}\subset [0,T]\cap\T$ with $t_n\to t$. 
 By inequality~\eqref{eq:T1}, the limit
 \begin{equation} \label{eq:T2}
  S(t)f:=\lim_{n\to\infty}S(t_n)f\in\Ck(\Rd)
 \end{equation}
 exists and does not depend on the choice of the approximating sequence $(t_n)_{n\in\N}$.\
 In order to extend $(S(t))_{t\geq 0}$ to the whole space $\Ck(\Rd)$, we fix $t\geq 0$ and 
 $f\in\Ck(\Rd)$. Assumption~\ref{ass:Sn}(ii) and the previous arguments guarantee that, for every 
 $\epsilon>0$, $r\geq 0$ and $K\Subset\Rd$, there exist $c\geq 0$ and $K'\Subset\Rd$ 
 with
 \begin{equation} \label{eq:D}
  \|S(t)g_1-S(t)g_2\|_{\infty,K}\leq c\|g_1-g_2\|_{\infty,K'}+\epsilon
 \end{equation}
 for all $g_1, g_2\in B_{\Ck(\Rd)}(r)\cap\D_0$.\ By Assumption~\ref{ass:Sn}(iii), there exists
 $(f_n)_{n\in\N}\subset\D_0$ with $f_n\to f$ and $\|f_n\|_\kappa\leq\|f\|_\kappa$ for all $n\in\N$.\
 Since Assumption~\ref{ass:Sn}(i) and equation~\eqref{eq:T2} guarantee that the sequence 
 $(S(t)f_n)_{n\in\N}$ is bounded, inequality~\eqref{eq:D} implies that
 \[ S(t)f:=\lim_{n\to\infty}S(t)f_n\in\Ck(\Rd) \]
 exists and does not depend on the choice of the approximating sequence $(f_n)_{n\in\N}$.\
 By construction, the operators $S(t)\colon\Ck(\Rd)\to\Ck(\Rd)$ are convex and monotone with
 $S(t)0=0$ and satisfy the conditions~(ii) and~(iii).
	
 Second, we show that the family $(S(t))_{t\geq 0}$ is strongly continuous and satisfies
 \begin{equation} \label{eq:Snl2}
  S(t)f=\lim_{l\to\infty}S_{n_l}(\bar{t}_{n_l})f \quad\mbox{for all } (f,t)\in\Ck(\Rd)\times\T.
 \end{equation}
 To do so, let $f\in\Ck(\Rd)$, $t\geq 0$, $\epsilon>0$, $K\Subset\Rd$ and 
 $(f^k)_{k\in\N}\subset\D_0$ with $f^k\to f$ and $\|f^k\|_\kappa\leq\|f\|_\kappa$ for all $k\in\N$. 
 By the first part, there exist $c\geq 0$ and $K'\Subset\Rd$ with
 \[ \|S(s)g_1-S(s)g_2\|_{\infty,K}\leq c\|g_1-g_2\|_{\infty,K'}+\tfrac{\epsilon}{6} \]
 for all $s\in [0,t+1]$ and $g_1, g_2\in B_{\Ck(\Rd)}(r)$, where $r:=\|f\|_\kappa$.\ 
 Choose $k\in\N$ such that $6c\|f-f^k\|_{\infty,K'}\leq\epsilon$.\ By the inequalities~\eqref{eq:T1} 
 and~\eqref{eq:T2} there exists $\delta\in (0,1]$ with
 \begin{align*}
  \|S(s)f-S(t)f\|_{\infty,K}
  &\leq\|S(s)f-S(s)f^k\|_{\infty,K}+\|S(s)f^k-S(t)f^k\|_{\infty,K} \\
  &\quad\, +\|S(t)f^k-S(t)f\|_{\infty,K} \\
  &\leq 2c\|f-f^k\|_{\infty,K'}+\|S(s)f^k-S(t)f^k\|_{\infty,K}+\tfrac{\epsilon}{3}<\epsilon
 \end{align*}
 for all $s\geq 0$ with $|s-t|<\delta$.\ This shows the strong continuity.\ Now, let $t\in\T$.\
 Due to Assumption~\ref{ass:Sn}(ii), we can further assume that
 \begin{equation} \label{eq:Snl3} 
  \|S_n(\bar{t}_n)g_1-S_n(\bar{t}_n)g_2\|_{\infty,K_n}\leq c\|g_1-g_2\|_{\infty,K'}+\tfrac{\epsilon}{6}
 \end{equation}
 for all $n\in\N$ and $g_1, g_2\in B_{\Ck(\Rd)}(r)$.\ Choose $k\in\N$ such that 
 $12c\|f-f^k\|_{\infty, K'}\leq\epsilon$. Furthermore, by equation~\eqref{eq:Snl1}, there exists $l_0\in\N$ with 
 \begin{equation} \label{eq:Snl4}
  \|S(t)f^k-S_{n_l}(\bar{t}_{n_l})\bar{f}^k_{n_l}\|_{\infty,K_{n_l}}\leq\tfrac{\epsilon}{3}
  \quad\mbox{and}\quad
  \|f^k-\bar{f}^k_{n_l}\|_{\infty, K'_{n_l}}\leq\tfrac{\epsilon}{6c}\quad\mbox{for all } l\geq l_0. 
 \end{equation}
 We combine inequality~\eqref{eq:Snl3} and inequality~\eqref{eq:Snl4} to obtain 
 \begin{align*}
  \|S(t)f-S_{n_l}(\bar{t}_{n_l})f\|_{\infty,K_{n_l}}
  &\leq\|S(t)f-S(t)f^k\|_{\infty, K}+\|S(t)f^k-S_{n_l}(\bar{t}_{n_l})\bar{f}^k_{n_l}\|_{\infty,K_{n_l}} \\
  &\quad\; +\|S_{n_l}(\bar{t}_{n_l})\bar{f}^k_{n_l}-S_{n_l}(\bar{t}_{n_l})f\|_{\infty,K_{n_l}} \\
  &\leq 2c\|f-f^k\|_{\infty,K'}+ c\|f^k-\bar{f}^k_{n_l}\|_{\infty,K'_{n_l}}+\tfrac{2\epsilon}{3}<\epsilon
 \end{align*}
 for all $l\geq l_0$. This shows that $S(t)f=\lim_{l\to\infty}S_{n_l}(\bar{t}_{n_l})f$. 
	
 Third, we show that
 \begin{equation} \label{eq:conv}
  S(t)f=\lim_{l\to\infty}S_{n_l}(t_{n_l})f_{n_l}.
 \end{equation}
 for all $t\geq 0$, $f\in\Ck(\Rd)$, $f_n\in\Ck(X_n)$ with $f_n\to f$ and $t_n\in\T_n$
 with $t_n\to t$.
 We emphasize that the sequences $(t_n)_{n\in\N}$ and $(f_n)_{n\in\N}$ are arbitrary, 
 whereas the subsequence $(n_l)_{l\in\N}$ is as before.\
 Let $T:=\sup_{n\in\N}t_n$ and $r:=\sup_{n\in\N}c_\kappa\|f_n\|_{\kappa, X_n}$, 
 where $c_\kappa$ is given by equation~\eqref{eq:kappa}.\ Let $\epsilon>0$ and $K\Subset\Rd$.\
 By Assumption~\ref{ass:Sn}(ii) and the first part, there exist $c\geq 0$ and $K'\Subset\Rd$ with
 \begin{align} 
  \|S_n(s)g_1-S_n(s)g_2\|_{\infty, K_n} &\leq c\|g_1-g_2\|_{\infty, K_n'}+\tfrac{\epsilon}{7}, \nonumber \\
  \|S(t)g_1-S(t)g_2\|_{\infty, K} &\leq c\|g_1-g_2\|_{\infty, K'}+\tfrac{\epsilon}{7} \label{eq:conv1}
 \end{align}
 for all $n\in\N$, $s\in [0,T]\cap\T_n$ and $g_1, g_2\in B_{\Ck(X_n)}(r)$. 
 Moreover, by Assumption~\ref{ass:Sn}(iii), there exist $n_0\in\N$ and $g\in\D_0$ with $\|g\|_\kappa\leq r$ and
 \begin{equation} \label{eq:conv2}
  \max\{\|f-g\|_{\infty, K'}, \|f_n-g\|_{\infty, K_n'}\}\leq\tfrac{\epsilon}{7c}
  \quad\mbox{for all } n\geq n_0. 
 \end{equation}
 It follows from equation~\eqref{eq:T0} that there exists $\delta>0$ with
 \begin{equation} \label{eq:conv3}
  \|S_n(s_1)g-S_n(s_2)g\|_{\infty, K_n}\leq\tfrac{\epsilon}{7}
 \end{equation}  
 for all $n\in\N$ and $s_1, s_2\in [0,T]\cap\T_n$ with $|s_1-s_2|<\delta$.\
 Indeed, the corresponding proof does not rely on the particular choice of time points.\
 Furthermore, since $(S(t))_{t\geq 0}$ is strongly continuous, we can fix $s\in [0,T]\cap\T$ with
 $|s-t|<\delta$ and 
 \begin{equation} \label{eq:conv4}
  \|S(s)g-S(t)g\|_{\infty, K}\leq\tfrac{\epsilon}{7}.
 \end{equation} 
 By equation~\eqref{eq:Snl2}, there exists $l_0\in\N$ with $n_{l_0}\geq n_0$ and
 \begin{equation} \label{eq:conv5}
  \|S_{n_l}(\bar{s}_{n_l})g-S(s)g\|_{\infty, K_{n_l}}\leq\tfrac{\epsilon}{7} \quad\mbox{for all } l\geq l_0.
 \end{equation}
 For every $l\geq l_0$, we combine the inequalities~\eqref{eq:conv1}-\eqref{eq:conv5} to obtain 
 \begin{align*}
  &\|S(t)f-S_{n_l}(t_{n_l})f_{n_l}\|_{\infty,K_{n_l}} \\
  &\leq\|S(t)f-S(t)g\|_{\infty, K}+\|S(t)g-S(s)g\|_{\infty,K}+\|S(s)g-S_{n_l}(\bar{s}_{n_l})g\|_{\infty,K_{n_l}} \\
  &\quad\; +\|S_{n_l}(\bar{s}_{n_l})g-S_{n_l}(t_{n_l})g\|_{\infty,K_{n_l}}
    +\|S_{n_l}(t_{n_l})g-S_{n_l}(t_{n_l})f_{n_l}\|_{\infty,K_{n_l}} \\
  &\leq c\|f-g\|_{\infty,K}+c\|g-f_{n_l}\|_{\infty,K_{n_l}'}+\tfrac{5\epsilon}{7}\leq\epsilon.
 \end{align*}
	
 Fourth, we show that $(S(t))_{t\geq 0}$ is a semigroup.\ To do so, let $f\in\Ck(\Rd)$, $\epsilon>0$ 
 and $K\Subset\Rd$. Due to the first part, there exist $c_1\geq 0$ and $K'\Subset\Rd$ with $K\subset K'$ and 
 \[ \|S(0)f-S(0)g\|_{\infty,K}\leq c_1\|f-g\|_{\infty,K'}+\epsilon \]
 for all $g\in\Ck(\Rd)$ with $\|g\|_\kappa\leq\|f\|_\kappa$.\ By Assumption~\ref{ass:Sn}(iii),
 we can choose $g\in\D_0$ with $2(c_1+1)\|f-g\|_{\infty,K'}\leq\epsilon$ and $\|g\|_\kappa\leq\|f\|_\kappa$.\
 Let $(t^k)_{k\in\N}\subset\T$ with $t^k\to 0$. Due to inequality~\eqref{eq:T0} and inequality~\eqref{eq:T1},
  it holds
  \begin{align*}
  \|S(0)f-f\|_{\infty,K}
  &\leq\|S(0)f-S(0)g\|_{\infty,K}+\|S(0)g-S(t^k)g\|_{\infty,K} \\
  &\quad\; +\|S(t^k)g-S_{n_l}(\bar{t}^k_{n_l})\bar{g}_{n_l}\|_{\infty,K_{n_l}}
    +\|S_{n_l}(\bar{t}^k_{n_l})\bar{g}_{n_l}-\bar{g}_{n_l}\|_{\infty,K_{n_l}} \\
  &\quad\; +\|\bar{g}_{n_l}-g\|_{\infty,K}+\|f-g\|_{\infty,K} \\
  &\leq (c_1+1)\|f-g\|_{\infty,K'}+c_2(t^k+\bar{t}^k_{n_l})+\|\bar{g}_{n_l}-g\|_{\infty,K} \\
  &\quad\; +\|S(t^k)g-S_{n_l}(\bar{t}^k_{n_l})\bar{g}_{n_l}\|_{\infty,K_{n_l}}.
 \end{align*}
 Choosing first $k\in\N$ and then $l\in\N$ sufficiently large, we obtain from equation~\eqref{eq:Snl1} 
 that $\|S(0)f-f\|_{\infty,K}\leq\epsilon$.\ This shows that $S(0)f=f$.\ Now, let $s,t\geq 0$ and 
 choose $s_n, t_n\in\T_n$ with $s_n\to s$ and $t_n\to t$. It holds
	\begin{align*}
		S(s+t)f&-S(s)S(t)f =\big(S(s+t)f-S_{n_l}(s_{n_l}+t_{n_l})f\big)\\
		&+\big(S_{n_l}(s_{n_l})S_{n_l}(t_{n_l})f-S_{n_l}(s_{n_l})S(t)f\big) +\big(S_{n_l}(s_{n_l})S(t)f-S(s)S(t)f\big)
	\end{align*}
    for all $l\in\N$. Equation~\eqref{eq:conv} implies that the first and third term on the right-hand side 
	converge to zero as $l\to\infty$.\ Moreover, for every $\epsilon>0$ and $K\Subset\Rd$, 
	it follows from Assumption~\ref{ass:Sn}(i) and~(ii) that there exist $c\geq 0$ and 
	$K'\Subset\Rd$ with
	\[ \|S_{n_l}(s_{n_l})S_{n_l}(t_{n_l})f-S_{n_l}(s_{n_l})S(t)f\|_{\infty,K_{n_l}}
	\leq c\|S_{n_l}(t_{n_l})f-S(t)f\|_{\infty,K'_{n_l}}+\epsilon \]
	for all $l\in\N$. Hence, we can use equation~\eqref{eq:conv} again to obtain
	\[ \lim_{l\to\infty}\big(S_{n_l}(s_{n_l})S_{n_l}(t_{n_l})f-S_{n_l}(s_{n_l})S(t)f\big)=0 \] 
    This shows that $S(s+t)f=S(s)S(t)f$.
	
	Fifth, we show that $\D_\infty\subset D(A)$. Furthermore, we show that 
	\begin{equation} \label{eq:gen}
		\lim_{n\to\infty}\left\|\frac{S_n(h_n)f_n-f_n}{h_n}-Af\right\|_{\infty,K_n}=0
        \quad\mbox{for all } f\in\D_\infty \mbox{ and } K\Subset\Rd,
	\end{equation}
	where $(h_n)_{n\in\N}$ and $(f_n)_{n\in\N}$ are as in Definition~\ref{def:Dinfty}.\ 
    Let $f\in\D_\infty$ and $K\Subset\Rd$.\ Choose $(f_n)_{n\in\N}$ and 
    $(t_n)_{n\in\N}$ such that the conditions of Definition~\ref{def:Dinfty} are satisfied.\
    Let $\epsilon>0$ and $T:=\sup_{n\in\N}t_n$. By Assumption~\ref{ass:Sn}(i), there exists 
    $\lambda_0\in (0,1]$ with
	\begin{equation} \label{eq:gen1}
		\sup_{n\in\N}\sup_{t\in [0,T]\cap\T_n}\lambda_0\|S_n(t)f_n\|_{\kappa, X_n}
		\leq\inf_{x\in K}\frac{\kappa(x)\epsilon}{3}. 
	\end{equation}
	Let $\lambda\in (0,\lambda_0]$ with $\lambda T\leq\lambda_0$. Inequality~\eqref{eq:Linfty} implies
	\begin{equation} \label{eq:gen2}
		r:=\sup_{n\in\N}\sup_{h\in (0,t_n]\cap\T_n}\left\|\frac{S_n(h)f_n-f_n}{\lambda h}+f_n\right\|_{\kappa, X_n}<\infty.
	\end{equation}
	By Assumption~\ref{ass:Sn}(ii), there exist $c\geq 0$ and $K'\Subset\Rd$ with
	\begin{equation} \label{eq:gen3}
		\|S_n(t)g_1-S_n(t)g_2\|_{\infty, K_n}\leq c\|g_1-g_2\|_{\infty, K'_n}+\tfrac{\epsilon}{3}
	\end{equation}
	for all $n\in\N$, $t\in [0,T]\cap\T_n$ and $g_1, g_2\in B_{\Ck(X_n)}(r)$.\ Choose $g\in\Ck(\Rd)$
	as in Definition~\ref{def:Dinfty} and $h_n\in (0,t_n]\cap\T_n$ with $h_n\to 0$ such that 
    equation~\eqref{eq:Dinfty} is valid for the compact set $K'$.
	In particular, there exists $n_0\in\N$ with 
	\begin{equation} \label{eq:gen4}
		\left\|\frac{S_n(h_n)f_n-f_n}{h_n}-g\right\|_{\infty, K'_n}\leq\frac{\epsilon}{3c}
		\quad\mbox{for all } n\geq n_0. 
	\end{equation}
	In the sequel, all functions are evaluated at a fixed point $x\in K$.\ For every $n\geq n_0$
	and $k\in\N$ with $kh_n\leq T$, we use Lemma~\ref{lem:lambda} and inequality~\eqref{eq:gen1} to estimate
	\begin{align*}
		&S_n(kh_n)f_n-f_n 
		=\sum_{i=1}^k \big(S_n((i-1)h_n)S_n(h_n)f_n-S_n((i-1)h_n)f_n\big) \\
		&\leq\lambda h_n\sum_{i=1}^k \left(S_n((i-1)h_n)\left(\frac{S_n(h_n)f_n-f_n}{\lambda h_n}+f_n\right)
		-S_n((i-1)h_n)f_n\right) \\
		&\leq\lambda h_n\sum_{i=1}^k S_n((i-1)h_n)\left(\frac{S_n(h_n)f_n-f_n}{\lambda h_n}+f_n\right)
		+\frac{kh_n\epsilon}{3}. 
	\end{align*}
	Furthermore, it follows from equation~\eqref{eq:gen2} and inequality~\eqref{eq:gen3} that
	\begin{align*}
		&\lambda h_n\sum_{i=1}^k S_n((i-1)h_n)\left(\frac{S_n(h_n)f_n-f_n}{\lambda h_n}+f_n\right)
		+\frac{kh_n\epsilon}{3} \\
		&=\lambda h_n\sum_{i=1}^k S_n((i-1)h_n)\left(\frac{1}{\lambda}g+f_n\right)+\frac{kh_n\epsilon}{3} \\
		&\quad\; +\lambda h_n\sum_{i=1}^k \left(S_n((i-1)h_n)\left(\frac{S_n(h_n)f_n-f_n}{\lambda h_n}+f_n\right)
		-S_n((i-1)h_n)\left(\frac{1}{\lambda}g+f_n\right)\right) \\
		&\leq\lambda h_n\sum_{i=1}^k S_n((i-1)h_n)\left(\frac{1}{\lambda}g+f_n\right)
		+ckh_n\left\|\frac{S_n(h_n)f_n-f_n}{h_n}-g\right\|_{\infty, K'_n}+\frac{2kh_n\epsilon}{3}.
	\end{align*}
	Combining the previous estimates with inequality~\eqref{eq:gen4} yields 
	\begin{equation} \label{eq:gen5}
		S_n(kh_n)f_n-f_n \leq\lambda h_n\sum_{i=1}^k S_n((i-1)h_n)\left(\frac{1}{\lambda}g+f_n\right)
		+kh_n\epsilon.
	\end{equation}
	Let $t\in [0,T]$, $k_n:=\max\{k\in\N_0\colon kh_n\leq t\}$ and 
	$i_n^s:=\max\{i\in\N\colon (i-1)h_n\leq s\}$ for all  $s\in [0,t]$ and $n\in\N$.\ Since 
	$k_nh_n\to t$ and $i_n^s\to s$ for all $s\in [0,t]$, it follows from 
    equation~\eqref{eq:conv}, inequality~\eqref{eq:gen5} and Assumption~\ref{ass:Sn}(i) that
	\begin{align}
		&S(t)f-f =\lim_{l\to\infty}\big(S_{n_l}(k_{n_l}h_{n_l})f_{n_l}-f_{n_l}\big) \nonumber \\
		&\leq\lim_{l\to\infty}\lambda\int_0^t\sum_{i=1}^{k_{n_l}}
		S_{n_l}((i-1)h_{n_l})\left(\frac{1}{\lambda}g+f_{n_l}\right)\one_{[(i-1)h_{n_l}, ih_{n_l})}(s)\,\d s 
		+\epsilon t \nonumber \\
		&=\lambda\int_0^t \lim_{l\to\infty}S_{n_l}((i_{n_l}^s-1)h_{n_l})\left(\frac{1}{\lambda}g+f_{n_l}\right)\d s
		+\epsilon t=\lambda\int_0^t S(s)\left(\frac{1}{\lambda}g+f\right)\d s+\epsilon t. \label{eq:gen6}
	\end{align}
	The previous estimate holds uniformly for all $\lambda\in (0,\lambda_0]$ with $\lambda T\leq\lambda_0$ and $t\in [0,T]$.\ Regarding the lower bound, we observe that
	\begin{align*}
		S_n(kh_n)f_n&-f_n = -\sum_{i=1}^k \big(S_n((i-1)h_n)f_n-S_n((i-1)h_n)S_n(h_n)f_n\big) \\
		&\geq -\lambda h_n\sum_{i=1}^k \left(S_n((i-1)h_n)
		\left(S_n(h_n)f_n-\frac{S_n(h_n)f_n-f_n}{\lambda h_n}\right)-S_n(ih_n)f_n\right)
	\end{align*}
	for all $k,n\in\N$ and $\lambda>0$ with $\lambda h_n\leq 1$. Hence, there exists $\lambda_1\in (0,1]$ with 
	\begin{equation} \label{eq:gen7}
		S(t)f-f\geq -\lambda\int_0^t S(s)\left(-\frac{1}{\lambda}g+f\right)\d s-\epsilon t
	\end{equation}
	for all $\lambda\in (0,\lambda_1]$ with $\lambda T\leq\lambda_1$ and $t\in [0,T]$.\ Since inequality~\eqref{eq:gen6} and inequality~\eqref{eq:gen7} hold uniformly for all $x\in K$, 
    the strong continuity of $(S(t))_{t\geq 0}$ implies
	\[ \lim_{h\downarrow 0}\left\|\frac{S(h)f-f}{h}-g\right\|_{\infty, K}=0. \]
	We conclude the proof by observing that the function $g\in\Ck(\Rd)$ neither depends on 
    the choice of $K'\Subset\Rd$ nor on $\epsilon>0$. 
\end{proof}

Denote by $\UCK(X)$ the $\|\cdot\|_\kappa$-closure of $\Lipb(X)$ in $\Ck(X)$
for all $X\subset\Rd$, where $\Lipb(X)$ consists of all bounded Lipschitz functions $f\colon X\to\R$.

\begin{corollary} \label{cor:Sn}
 Let $(S_n)_{n\in\N}$ be a sequence that satisfies Assumption~\ref{ass:Sn} and $(S(t))_{t\geq 0}$ be a strongly continuous convex monotone semigroup on $\Ck(\Rd)$
 with generator $A\colon D(A)\to\Ck(\Rd)$ for a subsequence $(n_l)_{l\in\N}\subset\N$ as in Theorem \ref{thm:Sn}.\ Moreover,
 for every $f\in\Cbi(\Rd)$, there exist $f_n\in\Ck(X_n)$ and $t_n\in\T_n\backslash\{0\}$ with 
 $f_n\to f$ and
  \[ \sup_{n\in\N}\sup_{t\in (0,t_n]\cap\T_n}\left\|\frac{S_n(t)f_n-f_n}{t}\right\|_{\kappa,X_n}< \infty. \] 
 Then, for every $f\in\UCK(\Rd)$ such that there exist $f_n\in\Ck(X_n)$, $h_n\in \T_n\setminus\{0\}$ and $g\in\UCK(\Rd)$
 with $f_n\to f$, $h_n\to 0$ and 
 \[ \lim_{n\to\infty}\left\|\frac{S_n(h_n)f_n-f_n}{h_n}-g\right\|_{\kappa,X_n}=0, \]
 it follows that $f\in D(A)$ with
 \[ \lim_{h\downarrow0}\left\|\frac{S(h)f-f}{h}-Af\right\|_\kappa=0. \]
\end{corollary}
\begin{proof}
 Since $\UCK(\Rd)$ is the $\|\cdot\|_\kappa$-closure of $\Cbi(\Rd)$, one can argue
 similar to the proof of Theorem~\ref{thm:Sn} to show that
 \[ \lim_{t\to 0}\|S(t)f-f\|_\kappa=0 \quad\text{for all } f\in\UCK(\Rd). \]
 In particular, we obtain
 \[ \lim_{t\downarrow 0}\left\|\frac{1}{t}\int_0^t S(s)f\,\d s-f\right\|_\kappa=0 
    \quad\text{for all } f\in\UCK(\Rd). \]
 Hence, the claim follows similarly to the proof of Theorem~\ref{thm:Sn}.
\end{proof}

In case that $\T_n=\R_+$ for all $n\in\N$, the rather complicated Definition~\ref{def:Dinfty}
is not necessary.\ Subsequently, the generator of $(S_n(t))_{t\geq 0}$ is denoted by $A_n$ for all $n\in\N$.

\begin{assumption} \label{ass:Sn2}
 Let $(S_n)_{n\in\N}$ be a sequence of strongly continuous convex monotone semigroups 
 $(S_n(t))_{t\geq 0}$ on $\Ck(X_n)$ satisfying the following conditions:
 \begin{enumerate}
  \item[(i)] For every $r,T\geq 0$, there exists $c\geq 0$ with
   \[ \|S_n(t)f-S_n(t)g\|_{\kappa,X_n}\leq c\|f-g\|_{\kappa,X_n} \]
   for all $n\in\N$, $t\in [0,T]$ and $f,g\in B_{\Ck(X_n)}(r)$. 
  \item[(ii)] For every $\epsilon>0$, $r,T\geq 0$ and $K\Subset\Rd$, there exist $c\geq 0$
   and $K'\Subset\Rd$ with
   \[ \|S_n(t)f-S_n(t)g\|_{\infty,K_n}\leq c\|f-g\|_{\infty,K'_n}+\epsilon \]
   for all $n\in\N$, $t\in [0,T]$ and $f,g\in B_{\Ck(X_n)}(r)$. 
  \item[(iii)] There exists a countable set $\D_0\subset\Ck(\Rd)$ such that, 
   for every $f\in\D_0$, there exist $\bar{f}_n\in D(A_n)$ with $\bar{f}_n\to f$ 
   and $\|\bar{f}_n\|_\kappa\leq\|f\|_\kappa$ for all $n\in\N$
   such that $(A_n\bar{f}_n)_{n\in\N}$ converges to a limit in $\Ck(\Rd)$ and,
   for every $t\geq 0$, there exist $\bar{t}_n\in\T_n$ with $\bar{t}_n\to t$ such that
 $(S_n(\bar{t}_n)\bar{f}_n)_{n\in\N}$ is uniformly equicontinuous.\
   Moreover, for every $f\in\Ck(\Rd)$, there exists $(f_n)_{n\in\N}\subset\D_0$ 
   with $f_n\to f$ and $\|f_n\|_\kappa\leq\|f\|_\kappa$ for all $n\in\N$.
 \end{enumerate}
\end{assumption}

\begin{theorem} \label{thm:Sn2}
 Let $(S_n)_{n\in\N}$ be a sequence satisfying Assumption~\ref{ass:Sn2}.\ Then, there exist
 a strongly continuous convex monotone semigroup $(S(t))_{t\geq 0}$ on $\Ck(\Rd)$ 
 with generator $A\colon D(A)\to\Ck(\Rd)$ and a subsequence $(n_l)_{l\in\N}\subset\N$ with
 \[ S(t)f=\lim_{l\to\infty}S_{n_l}(t_{n_l})f_{n_l} \quad\mbox{for all } (f,t)\in\Ck(\Rd)\times\R_+, \]
 where $t_n\in\R_+$ with $t_n\to t$ and $f_n\in\Ck(X_n)$ with $f_n\to f$ are arbitrary.\
 In addition, the following statements are valid:
 \begin{enumerate}
  \item[(i)] For every $f\in\Ck(\Rd)$ and $f_n\in D(A_n)$ with $f_n\to f$ such that the sequence 
   $(A_n f_n)_{n\in\N}$ converges, it holds $f\in D(A)$ and $Af=\lim_{n\to\infty}A_n f_n$.\
  \item[(ii)] For every $r,T\geq 0$, there exists $c\geq 0$ with
   \[ \|S(t)f-S(t)g\|_\kappa\leq c\|f-g\|_\kappa 
    \quad\mbox{for all } t\in [0,T] \mbox{ and } f,g\in B_{\Ck(\Rd)}(r). \]
  \item[(iii)] For every $\epsilon>0$, $r,T\geq 0$ and $K\Subset\Rd$, there exist $c\geq 0$ 
   and $K'\Subset\Rd$ with
   \[ \|S(t)f-S(t)g\|_{\infty,K}\leq c\|f-g\|_{\infty,K'}+\epsilon \]
   for all $t\in [0,T]$ and $f,g\in B_{\Ck(\Rd)}(r)$. 
\end{enumerate}
\end{theorem}
\begin{proof}
 For every $f\in\Ck(\Rd)$ and $f_n\in D(A_n)$ with $f_n\to f$ such that
 $(A_n f_n)_{n\in\N}$ converges, we show that $f\in\D_\infty$.\
 For every $n\in\N$ and $t\geq 0$, it holds
 \begin{equation} \label{eq:int1}
  S_n(t)f_n-f_n\leq\int_0^t \big(S_n(s)(f_n+A_n f_n)-S_n(s)f_n\big)\,\d s
 \end{equation}
 since the proof of~\cite[Lemma~2.3]{BK22} does not change when $\Rd$ is replaced by $X_n$.
	Furthermore, it follows from $f_n\in D(A_n)$ and Assumption~\ref{ass:Sn2}(i) that
    the mapping $\R_+\to\R,\; t\mapsto (S_n(t)f_n)(x)$ is Lipschitz continuous and therefore satisfies
	\[ (S_n(t)f_n-f_n)(x)=\int_0^t \frac{\d}{\d s}(S_n(s)f_n)(x)\,\d s 
	\quad\mbox{for all } t\geq 0 \mbox{ and } x\in X_n. \]
	For every $n\in\N$, $t\geq 0$ and $h\in (0,1]$, Lemma~\ref{lem:lambda} implies
	\begin{align*}
		&\frac{S_n(t)f_n-S_n(t)S_n(h)f_n}{h}-S_n(t)(f_n-A_n f_n)-S_n(t)f_n \\
		& \leq S_n(t)\left(\frac{f_n-S_n(h)f_n}{h}+S_n(h)f_n\right)-S_n(t)S_n(h)f_n+S_n(t)f_n-S_n(t)(f_n-A_n f_n) \\
		&\leq\frac{1}{2}S_n(t)\left(2\left(S_n(h)f_n-f_n+A_n f_n-\frac{S_n(h)f_n-f_n}{h}\right)+f_n-A_n f_n\right) \\
		&\quad\; -\frac{1}{2}S_n(t)(f_n-A_n f_n)-S_n(t)S_n(h)f_n+S_n(t)f_n. 
	\end{align*}
	The strong continuity of $(S_n(t))_{t\geq 0}$, Assumption~\ref{ass:Sn2}(ii) and $f_n\in D(A_n)$ 
	imply that the right-hand side converges to zero as $h\to 0$. Hence, 
	\begin{equation} \label{eq:int2}
		S_n(t)f_n-f_n\geq -\int_0^t \big(S_n(s)(f_n-A_n f_n)-S_n(s)f_n\big)\,\d s
	\end{equation}
	for all $n\in\N$ and $t\geq 0$.\ Since $(A_n f_n)_{n\in\N}$ is bounded, it follows 
	from inequality~\eqref{eq:int1}, inequality~\eqref{eq:int2} and Assumption~\ref{ass:Sn2}(i) 
	that there exists $c\geq 0$ with 
	\[ \|S_n(t)f_n-f_n\|_\kappa\leq ct \quad\mbox{for all } n\in\N \mbox{ and } t\in [0,1]. \]
	Furthermore, the limit $g:=\lim_{n\to\infty}A_n f_n\in\Ck(\Rd)$ exists 
	by assumption.\ Hence, for every $n\in\N$ and $K\Subset\Rd$, we can choose $h_n\in (0,1]$ with 
	\begin{align*}
		\left\|\frac{S_n(h_n)f_n-f_n}{h_n}-g\right\|_{\infty,K_n}
		&\leq\left\|\frac{S_n(h_n)f_n-f_n}{h_n}-A_n f_n\right\|_{\infty,K_n}+\|A_n f_n-g\|_{\infty,K_n} \\
		&\leq\|A_n f_n-g\|_{\infty,K_n}+\tfrac{1}{n}\to 0.
	\end{align*}
	Now, the claim follows from Theorem~\ref{thm:Sn}.
\end{proof}

\begin{corollary} \label{cor:Sn2}
 Let $(S_n)_{n\in\N}$ be a sequence that satisfies Assumption~\ref{ass:Sn2} and $(S(t))_{t\geq 0}$ be a strongly continuous convex monotone semigroup on $\Ck(\Rd)$ 
 with generator $A\colon D(A)\to\Ck(\Rd)$ for a subsequence $(n_l)_{l\in\N}\subset\N$ as in Theorem \ref{thm:Sn2}.\ Moreover,
 for every $f\in\Cbi(\Rd)$, there exist $f_n\in\Ck(X_n)$ and $t_n>0$ with 
 $f_n\to f$ and
  \[ \sup_{n\in\N}\sup_{t\in (0,t_n]}\left\|\frac{S_n(t)f_n-f_n}{t}\right\|_{\kappa,X_n}< \infty. \] 
 Then, for every $f\in\UCK(\Rd)$ such that there exist $f_n\in D(A_n)$, $h_n>0$ and $g\in\UCK(\Rd)$ 
 with $f_n\to f$, $h_n\to 0$, $\|A_n f_n-g\|_{\kappa,X_n}\to 0$ and
 \[ \lim_{n\to \infty}\left\|\frac{S_n(h_{n})f_n-f_n}{h_n}-A_n f_n\right\|_{\kappa,X_n}=0, \]
 it follows that $f\in D(A)$ with
 \[ \lim_{h\downarrow 0}\left\|\frac{S(h)f-f}{h}-Af\right\|_\kappa=0. \]
\end{corollary}
\begin{proof}
 Similar to the proof of Theorem~\ref{thm:Sn2}, one can show that the conditions of 
 Corollary~\ref{cor:Sn} are satisfied.
\end{proof}

Another important case are so-called Chernoff-type approximations which provide a very
flexible and powerful tool to construct nonlinear semigroups, see~\cite{BDKN, BK22, BK22+} 
and the references therein.\
Let $\T_n:=\{kh_n\colon k\in\N_0\}$ for all $n\in\N$, where $(h_n)_{n\in \N}\subset (0,\infty)$
is a fixed sequence with $h_n\to 0$. Then, every semigroup $(S(t))_{t\in\T_n}$ has the form 
\[ S_n(kh_n)f=I_n^k f=(I_n\circ\ldots\circ I_n)f \quad\mbox{with}\quad I_n f:=S_n(h_n)f \]
In many applications, the one-step operators $I_n$ have an explicit representation which
allows to derive properties that transfer to the limit.\ If $X_n=\Rd$ and $I_n:=I(h_n)$ 
for a family $(I(t))_{t\geq 0}$ of operators, Chernoff-type approximations have already been
studied in the previously mentioned references.\
Now, we can show that any stable and consistent approximation scheme converges.\
For every $X\subset\Rd$ and $r\geq 0$, let $\Lipb(X,r)$ be the set of all $r$-Lipschitz functions 
$f\colon X\to\R$ with $\sup_{x\in X}|f(x)|\leq r$. 

\begin{assumption} \label{ass:cher}
 Let $(I_n)_{n\in\N}$ be a sequence of operators $I_n\colon\Ck(X_n)\to\Ck(X_n)$
 satisfying the following conditions:
 \begin{enumerate}
  \item[(i)] $I_n$ is convex and monotone with $I_n0=0$ for all $n\in\N$. 
  \item[(ii)] There exists $\omega\geq 0$ with
   \[ \|I_n f-I_n g\|_{\kappa, X_n}\leq e^{\omega h_n}\|f-g\|_{\kappa,X_n}
    \quad\mbox{for all } n\in\N \mbox{ and } f,g\in\Ck(X_n). \]
  \item[(iii)] For every $\epsilon>0$, $r,T\geq 0$ and $K\Subset\Rd$, there exist $c\geq 0$ and 
   $K'\Subset \Rd$ with
   \[ \|I_n^k f-I_n^k g\|_{\infty,K_n}\leq c\|f-g\|_{\infty, K'_n}+\epsilon \]
   for all $k,n\in\N$ with $kh_n\leq T$ and $f,g\in B_{\Ck(X_n)}(r)$.
  \item[(iv)] It holds $I_n\colon\Lipb(X_n,r)\to\Lipb(X_n,e^{\omega h_n}r)$ for all $r\geq 0$
   and $n\in\N$.
  \item[(v)] There exists a countable set $\D_0\subset\Lipb(\Rd)$ such that, for every $f\in\D_0$,
   there exist $r\geq 0$ and $\bar{f}_n\in\Lipb(X_n,r)$ with $\bar{f}_n\to f$ 
   and $\|\bar{f}_n\|_\kappa\leq\|f\|_\kappa$ for all $n\in\N$ such that the limit
   \[ g:=\lim_{n\to\infty}\frac{I_n\bar{f}_n-\bar{f}_n}{h_n}\in\Ck(\Rd) \quad\mbox{exists.} \]
   Moreover, for every $f\in\Ck(\Rd)$, there exists $(f_n)_{n\in\N}\subset\D_0$ with $f_n\to f$
   and $\|f_n\|_\kappa\leq\|f\|_\kappa$ for all $n\in\N$.
 \end{enumerate}
\end{assumption}

In Subsection~\ref{sec:cont}, we discuss sufficient conditions on the one-step operators~$I_n$ 
which guarantee that the iterated operators $I_n^k$ satisfy condition~(iii).

\begin{theorem} \label{thm:cher}
 Let $(I_n)_{n\in\N}$ be a sequence satisfying Assumption~\ref{ass:cher}.\ Then, there exist a 
 strongly continuous convex monotone semigroup $(S(t))_{t\geq 0}$ on $\Ck(\Rd)$
 with generator $A\colon D(A)\to\Ck(\Rd)$ and a subsequence $(n_l)_{l\in\N}\subset\N$ with
 \[ S(t)f=\lim_{l\to\infty}I_{n_l}^{k_{n_l}^t}f_{n_l} \quad\mbox{for all } (f,t)\in\Ck(\Rd)\times\R_+, \]
 where $k_n^t\in\N$ with $k_n^t h_n\to t$ and $f_n\in\Ck(X_n)$ with $f_n\to f$ are arbitrary.\
 Moreover, the following statements are valid:
 \begin{enumerate}
  \item[(i)] For every $f\in\Ck(\Rd)$ and $f_n\in\Ck(X_n)$ such that $f_n\to f$ and the limit
   \[ g:=\lim_{n\to\infty}\frac{I_n f_n-f_n}{h_n} \quad\mbox{exists}, \]
   it holds $f\in D(A)$ and $Af=g$.
  \item[(ii)] It holds $\|S(t)f-S(t)g\|_\kappa\leq e^{\omega t}\|f-g\|_\kappa$ for all $t\geq 0$ 
   and $ f,g\in\Ck(\Rd)$.
  \item[(iii)] For every $\epsilon>0$, $r,T\geq 0$ and $K\Subset\Rd$, there exist $c\geq 0$ 
   and $K'\Subset\Rd$ with
   \[ \|S(t)f-S(t)g\|_{\infty,K}\leq c\|f-g\|_{\infty,K'}+\epsilon \]
   for all $t\in [0,T]$ and $f,g\in B_{\Ck(\Rd)}(r)$. 
  \item[(iv)] It holds $S(t)\colon\Lipb(\Rd,r)\to\Lipb(\Rd,e^{\omega t}r)$ for all $r,t\geq 0$.\
 \end{enumerate}
\end{theorem} 
\begin{proof}
	We just have to verify Assumption~\ref{ass:Sn} for $S_n(kh_n):=I_n^k$.\
	The sequence $(S_n)_{n\in\N}$ consists of convex monotone semigroups 
	on $\Ck(X_n)$ and Assumption~\ref{ass:cher}(iii) guarantees that
	Assumption~\ref{ass:Sn}(ii) is satisfied. For every $k,n\in\N$ and $f,g\in\Ck(X_n)$,
	Assumption~\ref{ass:cher}(ii) implies
	\begin{equation} \label{eq:cher.norm}
		\|I_n^k f-I_n^k g\|_{\kappa, X_n}\leq e^{\omega kh_n}\|f-g\|_{\kappa, X_n} 
	\end{equation}
	showing that Assumption~\ref{ass:Sn}(i) is valid.\ Finally, it follows from Assumption~\ref{ass:cher}(iv) 
    and~(v) that Assumption~\ref{ass:Sn}(iii) is valid. 
    Now, Theorem~\ref{thm:Sn} yields the claim.
\end{proof}

\begin{corollary} \label{cor:cher}
 Let $(I_n)_{n\in\N}$ be a sequence satisfying Assumption~\ref{ass:cher} and $(S(t))_{t\geq 0}$ be a 
 strongly continuous convex monotone semigroup on $\Ck(\Rd)$
 with generator $A\colon D(A)\to\Ck(\Rd)$ for a subsequence $(n_l)_{l\in\N}\subset\N$ as in Theorem \ref{thm:cher}.\ Moreover,
 for every $f\in\Cbi(\Rd)$, there exist $f_n\in\Ck(X_n)$ with $f_n\to f$ and 
 \[ \sup_{n\in\N}\left\|\frac{I_n f_n-f_n}{h_n}\right\|_{\kappa,X_n}< \infty. \] 
 Then, for every $f\in\UCK(\Rd)$ such that there exist $f_n\in\Ck(X_n)$ and $g\in\UCK(\Rd)$ 
 with $f_n\to f$ and
 \[ \lim_{n\to\infty}\left\|\frac{I_n f_n-f_n}{h_n}-g\right\|_{\kappa,X_n}=0, \]
 it follows that $f\in D(A)$ with
 \[ \lim_{h\downarrow 0}\left\|\frac{S(h)f-f}{h}-Af\right\|_\kappa=0. \]
\end{corollary}
\begin{proof}
 Similar to the proof of Theorem~\ref{thm:cher}, one can show that the conditions of 
 Corollary~\ref{cor:Sn} are satisfied.
\end{proof}

Since the convergence equation~\eqref{eq:Sn} is based on relative compactness, we a priori
only obtain convergence for a subsequence.\ To overcome this limitation, we have to show
that the semigroup $(S(t))_{t\geq 0}$ is uniquely determined by its generator.\
Then, if the set $\D_\infty$ is sufficiently large, the limit in equation~\eqref{eq:Sn}
will not depend on the choice of the subsequence and the semigroup $(S(t))_{t\geq 0}$
will be uniquely determined by $(A_n)_{n\in\N}$. In the following subsection, we provide
sufficient conditions.

\subsection{Uniqueness of the limit semigroup}
\label{sec:unique}
In order to show that $(S(t))_{t\geq 0}$ is uniquely determined by its generator, 
we rely on the recently obtained comparison principle for strongly continuous
convex monotone semigroups in~\cite{BDKN}.\ While the uniqueness of strongly continuous
linear semigroups is classical, establishing the uniqueness of nonlinear semigroups is 
far more delicate.\ The core of the problem is that, in general, the domain of the generator 
is not invariant, i.e., $S(t)\colon D(A)\not\to D(A)$.\
However, it holds $D(A)\subset\LSplus$ 
and $S(t)\colon\LSplus\to\LSplus$ for all $t\geq 0$, where $\LSplus$ denotes the so-called 
upper Lipschitz set.\ The key insight of~\cite{BDKN} is that strongly continuous convex 
monotone semigroups are uniquely determined by their so-called upper $\Gamma$-generator 
$A_\Gamma^+ f$ which is defined for all $f\in\LSplus$ and satisfies $A_\Gamma^+f=Af$ 
for all $f\in D(A)$.\ Under additional conditions, the semigroup is even uniquely 
determined by the restriction of its generator to $\Cbi(\Rd)$ consisting of all bounded 
infinitely differentiable functions $f\colon X\to\R$ such that all derivatives are bounded.\
Define $$(\tau_x f)(y):=f(x+y)$$ for all $n\in\N$, $f\in\Ck(X_n)$ and $x,y\in X_n$.

We point out that, also in a nonlinear setting, the domain can be invariant if the generator 
is m-accretive or maximal monotone, see~\cite{Barbu10,BC91,BP,Brezis71,CL71,Kato67,Kurtz73}, 
and for strongly continuous convex monotone semigroups on $L^p$-type spaces, see~\cite{DKN21,BK22}.\
However, the examples presented in this article do not fall into either of these categories, 
see \cite[Example 5.4]{DKN21+} for an explicit counterexample.\ Furthermore, since $u(t):=S(t)f$ 
defines a viscosity solution, one could also rely on comparison principles for viscosity solutions 
to obtain uniqueness of the limit semigroup.\ This link between nonlinear semigroups and viscosity 
solutions has, for instance, been used in~\cite{Criens25a,Criens25b,DKN20,kuhn2018viscosity,NR21}.

\begin{theorem} \label{thm:unique}
 Let $(S_n)_{n\in\N}$ be a sequence satisfying Assumption~\ref{ass:Sn}.\ Moreover,
 we assume that the following conditions are valid:
 \begin{enumerate}
  \item[(i)] For every $f\in\Lipb(\Rd)$ and $\epsilon>0$, there exist $\delta,t_0>0$ 
   and $f_n\in\Ck(X_n)$ with $f_n\to f$ such that 
   \[ S_n(t)(\tau_x f_n)\leq\tau_x S_n(t)f_n+\tfrac{\epsilon t}{\kappa} \]
   for all $n\in\N$, $t\in [0,t_0]\cap\T_n$ and $x\in B_{X_n}(\delta)$.
  \item[(ii)] For every $f\in\Lipb(\Rd)$, there exist $r,t_0>0$ with
   $S_n(t)f\in\Lipb(X_n,r)$ for all $n\in\N$ and $t\in [0,t_0]\cap\T_n$. 
  \item[(iii)] It holds $\Cbi\subset\D_\infty$. 
 \end{enumerate}
 Then, the semigroup $(S(t))_{t\geq 0}$ from Theorem~\ref{thm:Sn} does not depend on the 
 choice of the convergent subsequence and we obtain
 \[ S(t)f=\lim_{n\to\infty}S_n(t_n)f_n \quad\mbox{for all } (f,t)\in\Ck(\Rd)\times\R_+, \]
 where $t_n\in\T_n$ with $t_n\to t$ and $f_n\in\Ck(X_n)$ with $f_n\to f$ are arbitrary.\
 In addition, the following statements are valid: 
 \begin{enumerate} 
  \item[(a)] For every $f\in\Lipb(\Rd)$ and $\epsilon>0$, there exist $\delta,t_0>0$ with
   \[ S(t)(\tau_x f)\leq\tau_x S(t)f+\tfrac{\epsilon t}{\kappa}
    \quad\mbox{for all } t\in [0,t_0] \mbox{ and } x\in B_{\Rd}(\delta). \]
  \item[(b)] It holds $S(t)\colon\Lipb(\Rd)\to\Lipb(\Rd)$ for all $t\geq 0$.
  \item[(c)] Let $(T(t))_{t\geq 0}$ be a strongly continuous convex monotone semigroup
   on $\Ck(\Rd)$ satisfying the conditions~(a) and~(b). Let $\Cbi(\Rd)\subset D(B)$ and
   \[ Af=Bf \quad\mbox{for all } f\in\Cbi(\Rd), \]
   where $B$ denotes the generator of $(T(t))_{t\geq 0}$.\ Then, we obtain $S(t)f=T(t)f$ for all 
   $t\geq 0$ and $f\in\Ck(\Rd)$.
 \end{enumerate}
\end{theorem}
\begin{proof}
 We show that, for every $f\in\Lipb(\Rd)$ and $\epsilon>0$, there exist $\delta,t_0>0$ with
 \begin{equation} \label{eq:unique1}
  S(t)(\tau_x f)\leq\tau_x S(t)f+\tfrac{\epsilon t}{\kappa}
  \quad\text{for all } t\in [0,t_0] \text{ and } x\in B_{\Rd}(\delta). 
 \end{equation}
 By condition~(i), there exist $\delta,t_0>0$ and $f_n\in\Ck(X_n)$ with 
 $f_n\to f$ such that
 \begin{equation} \label{eq:tau}
  S_n(t)(\tau_x f_n)\leq\tau_x S_n(t)f_n+\tfrac{\epsilon t}{\kappa}
 \end{equation}
 for all $n\in\N$, $t\in [0,t_0]\cap\T_n$, $x\in B_{X_n}(\delta)$ and $y\in X_n$. 
 Let $t\in [0,t_0]$, $x\in B_{\Rd}(\delta)$, $y\in\Rd$, $t_n\in [0,t_0]\cap\T_n$ 
 and $x_n, y_n\in X_n$ with $t_n\to t$, $x_n\to x$ and $y_n\to y$.\ Then, 
 \begin{align*}
  &\big(S(t)(\tau_x f)-\tau_x S(t)f\big)(y) \leq |(S(t)(\tau_x f))(y)-(S(t)(\tau_x f))(y_{n_l})|\\
  &\quad+|S(t)(\tau_x f)-S_{n_l}(t_{n_l})(\tau_{x_{n_l}}f_{n_l})|(y_{n_l}) 
    +\big(S_{n_l}(t_{n_l})(\tau_{x_{n_l}}f_{n_l})-\tau_{x_{n_l}}S_{n_l}(t_{n_l})f_{n_l}\big)(y_{n_l})\\
  &\quad +|\tau_{x_{n_l}}S_{n_l}(t_{n_l})f_{n_l}-\tau_{x_{n_l}}S(t)f|(y_{n_l})
    +|(\tau_{x_{n_l}}S(t)f)(y_{n_l})-(\tau_x S(t)f)(y)|
 \end{align*}
 for all $l\in\N$.\
 The first and the last term on the right-hand side converge to zero since $S(t)(\tau_x f)$ 
 and $S(t)f$ are continuous while equation~\eqref{eq:Sn} yields that the 
 second and fourth term vanish.\ Moreover, due to inequality~\eqref{eq:tau}, the third term 
 can be bounded by $\frac{\epsilon t_{n_l}}{\kappa(y_{n_l})}$ which converges to $\frac{\epsilon t}{\kappa(y)}$. 
 We obtain 
 \[ (S(t)(\tau_x f))(y)\leq(\tau_x S(t)f)(y)+\tfrac{\epsilon t}{\kappa(y)}. \]
 The statements~(b) and~(c) follow from condition~(ii) and~(iii), Theorem~\ref{thm:Sn}(i)
 and~\cite[Theorem~4.7]{BDKN}.

 It remains to show that the semigroup $(S(t))_{t\geq 0}$ does not depend on the choice 
 of the convergent subsequence.\ For every subsequence $(\tilde{n}_k)_{k\in\N}\subset\N$, 
 there exist a further subsequence $(\tilde{n}_{k_l})_{l\in\N}$ and a strongly continuous 
 convex monotone semigroup $(\tilde{S}(t))_{t\geq 0}$ on $\Ck(\Rd)$ satisfying the conditions~(a)--(c).\
 In particular, it holds $Af=\tilde{A}f$ for all $f\in\Cbi(\Rd)$ and we obtain
 \[ S(t)f=\tilde{S}(t)f \quad\mbox{for all } (f,t)\in\Ck(\Rd)\times\R_+. \]
 Since every subsequence has a further subsequence which converges to a limit that is 
 independent of the choice of the subsequence, we obtain
 \[ S(t)f=\lim_{n\to\infty}S_n(t_n)f_n \quad\mbox{for all } (f,t)\in\Ck(\Rd)\times\R_+, \]
 where $t_n\in\T_n$ with $t_n\to t$ and $f_n\in\Ck(X_n)$ with $f_n\to f$ are arbitrary.
\end{proof}

\begin{corollary} \label{cor:unique1}
 Let $(S_n)_{n\in\N}$ be a sequence satisfying Assumption~\ref{ass:Sn2} and denote
 by $(S(t))_{t\geq 0}$ the corresponding limit semigroup from Theorem~\ref{thm:Sn2}.\
 Furthermore, we assume that the following conditions are valid:
 \begin{enumerate}
  \item[(i)] For every $f\in\Lipb(\Rd)$ and $\epsilon>0$, there exist $\delta,t_0>0$ 
   and $f_n\in\Ck(X_n)$ with $f_n\to f$ such that 
   \[ S_n(t)(\tau_x f_n)\leq\tau_x S_n(t)f_n+\tfrac{\epsilon t}{\kappa} \]
   for all $n\in\N$, $t\in [0,t_0]\cap\T_n$ and $x\in B_{X_n}(\delta)$.
  \item[(ii)] For every $f\in\Lipb(\Rd)$, there exist $r,t_0>0$ with
   $S_n(t)f\in\Lipb(X_n,r)$ for all $n\in\N$ and $t\in [0,t_0]\cap\T_n$. 
  \item[(iii)] For every $f\in\Cbi$, there exist $f_n\in D(A_n)$ with $f_n\to f$
   such that $(A_n f_n)_{n\in\N}$ converges to a limit in $\Ck(\Rd)$.
 \end{enumerate}
 Then, all the statements from Theorem~\ref{thm:unique} are valid.
\end{corollary}

\begin{corollary} \label{cor:unique2}
 Let $(I_n)_{n\in\N}$ be a sequence satisfying Assumption~\ref{ass:cher}.\ Furthermore,
 we assume that the following conditions are valid:
 \begin{enumerate}
  \item[(i)] For every $\epsilon>0$, there exist $\delta>0$ and $n_0\in\N$ with
   \[ I_n(\tau_x f)\leq\tau_x I_n f+\frac{r\epsilon h_n}{\kappa} \]
   for all $r\geq 0$, $n\geq n_0$, $f\in\Lipb(X_n,r)$ and $x\in B_{X_n}(\delta)$.
  \item[(ii)] For every $f\in\Cbi(\Rd)$, there exist $f_n\in\Ck(X_n)$ with $f_n\to f$ 
   such that 
   \[ g:=\lim_{n\to\infty}\frac{I_n f_n-f_n}{h_n}\in\Ck(\Rd) \quad\mbox{exists}. \] 
 \end{enumerate}
 Then, all the statements from Theorem~\ref{thm:unique} are valid.
\end{corollary}
\begin{proof}
 We show that, for every $\epsilon>0$, there exists $\delta>0$ and $n_0\in\N$ with
 \[ I_n^{k+1}(\tau_x f)\leq\tau_x I_n^{k+1}f+\frac{e^{\omega kh_n}r\epsilon (k+1)h_n}{\kappa} \]
 for all $k\in\N$, $n\geq n_0$, $r\geq 0$, $f\in\Lipb(X_n,r)$ and $x\in B_{X_n}(\delta)$.\
 For $k=0$, this follows from condition~(i).\ If the estimate is valid for some fixed $k\in\N_0$, 
 we can use the monotonicity of $I_n$, the fact that $I_n^k f\in\Lipb(X_n,e^{\omega kh_n}r)$
 and Lemma~\ref{lem:kappa} to obtain
 \begin{align*}
  I_n^{k+1}(\tau_x f) &=I_n\big(I_n^k(\tau_x f)\big)
    \leq I_n\left(\tau_x I_n^k f+\frac{e^{\omega kh_n}r\epsilon kh_n}{\kappa}\right) \\
  &\leq\tau_x I_n^{k+1}f+\frac{e^{\omega (k+1)h_n}r\epsilon (k+1)h_n}{\kappa}.
 \end{align*}
 Now, the claim follows from Theorem~\ref{thm:unique}.
\end{proof}

\subsection{Equicontinuity w.r.t.\ the mixed topology}
\label{sec:cont}

Most of the previously imposed conditions can easily be verified in applications, 
see Section~\ref{sec:examples}.\ For instance, Assumption~\ref{ass:Sn}(i) is valid if 
$S(t)\frac{r}\kappa=\frac{r}\kappa$ which can easily be checked in the context of optimal 
control if $\kappa\equiv 1$.\ However, verifying Assumption~\ref{ass:Sn}(ii) can be tricky.\
In the sequel, we discuss how this condition can be verified by means of a suitable family 
of cut-off functions.\ In the context of stochastic processes, Assumption~\ref{ass:Sn}(ii)
can also be verified by means of suitable moment estimates, see~\cite{BDKN,BK22}.\ Let $(X_i)_{i\in I}$ be a family of closed sets $X_i\subset\Rd$ and $(\T_i)_{i\in I}$ 
be a family of sets $\T_i\subset\R_+$ with $s\pm t\in\T_i$ for all $i\in I$ and $s,t\in\T_i$ with $s\geq t$.\ Moreover, let $\{S_i(t)\colon i\in I, t\in\T_i\}$
be a family of convex monotone operators $S_i(t)\colon\Ck(X_i)\to\Ck(X_i)$ with 
$S_i(t)0=0$ such that
\[
 S_i(s+t)=S_i(s)S_i(t)\quad\text{for all }i\in I\text{ and }s,t\in \T_i
\]
and, for every $r,T\geq 0$, there exists $c\geq 0$ with
\begin{equation} \label{eq:Lip}
 \|S_i(t)f-S_i(t)g\|_{\kappa,X_i}\leq c\|f-g\|_{\kappa,X_i}
\end{equation}
for all $i\in I$, $t\in [0,T]\cap\T_i$ and $f,g\in B_{\Ck(X_i)}(r)$.\ In the sequel, we use the notation $K_i:=K\cap X_i$ for all 
	$i\in I$ and $K\Subset\Rd$.

\begin{lemma} \label{cor:cutoff1}
 Assume that, for every $\epsilon>0$, $r\geq 0$ and $K\Subset\Rd$, there exists a family
 $(\zeta_x)_{x\in K}$ of continuous functions $\zeta_x\colon\Rd\to\R$ and 
 $\tilde{K}\Subset\Rd$ such that
 \begin{enumerate}
  \item[(i)] $0\leq\zeta_x\leq 1$ and $\zeta_x(x)=1$ for all $x\in K$,
  \item[(ii)] $\sup_{y\in\tilde{K}^c}\zeta_x(y)\leq\epsilon$ for all $x\in K$,
  \item[(iii)] there exist $t_i\in\T_i\backslash\{0\}$ such that, for every $i\in I$, 
   $t\in [0,t_i]\cap\T_i$ and $x\in K_i$,
   \[ \left\|(S_i(t)\big(\tfrac{r}{\kappa}(1-\zeta_x)\big)-\tfrac{r}{\kappa}(1-\zeta_x)\right\|_{\kappa,X_i}
    \leq\epsilon t. \]
   \end{enumerate}
 Then, for every $\epsilon>0$, $r,T\geq 0$ and $K\Subset\Rd$, there exist $c\geq 0$ 
 and $K'\Subset\Rd$ with
 \begin{equation} \label{eq:cutoff1}
  \|S_i(t)f-S_i(t)g\|_{\infty,K_i}\leq c\|f-g\|_{\infty,K_i'}+\epsilon
 \end{equation}
 for all $i\in I$, $t\in [0,T]\cap\T_i$ and $f,g\in B_{\Ck(X_i)}(r)$.
\end{lemma}
\begin{proof}
Let $\epsilon>0$, $r,T\geq 0$, $K\Subset\Rd$ and let $(\zeta_x)_{x\in K}$ be a family of continuous 
 functions $\zeta_x\colon\Rd\to\R$ satisfying the conditions~(i)-(iii) with $\nicefrac{\epsilon}{c}$.\
 Furthermore, let $i\in I$ and $t\in [0,T]\cap\T_i$.\ Choose $t_1, t_2\in [0,t_i]\cap\T_i$ and $k\in\N$ 
 with $t=kt_1+t_2$.\ For every $x\in K$, we use inequality~\eqref{eq:Lip} 
 and condition~(iii) to obtain
 \begin{align*}
  \|S_i(t)f_x-f_x\|_{\kappa,X_i}
  &\leq\|S_i(kt_1)S_i(t_2)f_x-S_i(kt_1)f_x\|_{\kappa,X_i} \\
  &\quad\; +\sum_{l=0}^{k-1}\|S_i(lt_1)S_i(t_1)f_x-S_i(lt_1)f_x\|_{\kappa,X_i} \\
  &\leq  c\|S_i(t_2)f_x-f_x\|_{\kappa,X_i}+ck\|S_i(t_1)f_x-f_x\|_{\kappa,X_i} 
   \leq\epsilon t,
 \end{align*}
 where $f_x:=\frac{r}{\kappa}(1-\zeta_x)$.\ Since $f_x(x)=0$, we can apply Corollary~\ref{cor:cutoff.app} 
 with $\tilde{I}:=I\times [0,T]$, $X_{(i,t)}:=X_i$ and $\Phi_{(i,t)}:=S_i(t)$.
\end{proof}

The next result is particularly useful for differential operators.\ For an infinitely differentiable function $f\colon\Rd\to\R$ and $l\in\N$, let
$\|D^l f\|_\kappa:=\sum_{|\alpha|=l}\|D^\alpha f\|_\kappa\in [0,\infty]$,
where $D^\alpha f:=\partial_{x_1}^{\alpha_1}\cdots\partial_{x_d}^{\alpha_d}f$ and 
$|\alpha|:=\sum_{j=1}^d \alpha_j$ for all $\alpha:=(\alpha_1,\ldots,\alpha_d)\in\N_0^d$.

\begin{corollary} \label{cor:cutoff2}
 Suppose that, for every $r\geq 0$, there exist $\epsilon_0>0$, $h_i\in\T_i\backslash\{0\}$,
 $N\in\N$ and a non-decreasing function $\rho\colon\R_+\to\R_+$ with $\rho(\epsilon)\to 0$ 
 as $\epsilon\to 0$ such that
 \begin{equation} \label{eq:cutoff2a}
  \left\|\frac{S_i(h)f-f}{h}\right\|_{\kappa,X_i}\leq\rho\bigg(\sum_{l=1}^N\|D^l f\|_\kappa\bigg)
 \end{equation}
 for all $i\in I$, $h\in (0,h_i]\cap\T_i$ and $f:=\frac{r}{\kappa}(1-\zeta)$ with 
 $\zeta\in\Cci(\Rd)$ such that $0\leq\zeta\leq 1$ 
 and $\|D^l f\|_\kappa\leq\epsilon_0$ for all $1\leq l\leq N$.\ Furthermore, the function 
 $\nicefrac{1}{\kappa}$ is~$N$ times continuously differentiable and, for every $\epsilon>0$, 
 there exists $K\Subset\Rd$ with $\sup_{x\in K^c}|(D^\alpha\tfrac{1}{\kappa})\kappa|(x)\leq\epsilon$ 
 for all $1\leq |\alpha|\leq N$.\ Then, for every $\epsilon>0$, $r,T\geq 0$ and $K\Subset\Rd$, 
 there exist $c\geq 0$ and $K'\Subset\Rd$ with
 \[ \|S_i(t)f-S_i(t)g\|_{\infty,K_i}\leq c\|f-g\|_{\infty,K_i'}+\epsilon \]
 for all $i\in I$, $t\in [0,T]\cap\T_i$ and $f,g\in B_{\Ck(X_i)}(r)$.
\end{corollary} 
\begin{proof}
 Let $r\geq 0$, $\epsilon>0$, $K\Subset\Rd$ and $\xi\in\Cci(\Rd)$ with
 $0\leq\xi\leq 1$ and $\xi\equiv 1$ on $B_{\Rd}(1)$. 
 Choose $\delta_0>0$ with $\delta_0 r\leq\nicefrac{\epsilon}{2}$ and 
 \begin{equation} \label{eq:cutoff2b}
  \delta_0 r\sum_{\beta<\alpha}\binom{\alpha}{\beta}\|D^\beta\tfrac{1}{\kappa}\|_\kappa
  \|D^{\alpha-\beta}\xi\|_\infty\leq\frac{\epsilon}{2}
  \quad\mbox{for all } 1\leq |\alpha|\leq N. 
 \end{equation}
 Furthermore, there exists $\tilde{K}\Subset\Rd$ with 
 $\sup_{x\in\tilde{K}^c}|(D^\alpha\frac{1}{\kappa})\kappa|(x)\leq\delta_0$ for all $1\leq |\alpha|\leq N$. 
 Let $\delta\in (0,\delta_0]$ with $\delta|y-x|\leq 1$ for all $(x,y)\in K\times\tilde{K}$ and define 
 $\zeta_x(y):=\xi(\delta(y-x))$ for all $(x,y)\in K\times\Rd$.\ The conditions~(i) and~(ii) 
 from Lemma~\ref{cor:cutoff1} are clearly valid. In addition, for every $x\in K$ and $f_x:=\frac{r}{\kappa}(1-\zeta_x)$, we use inequality~\eqref{eq:cutoff2b} 
 and the equation $(1-\zeta_x)(y)=0$ for all $(x,y)\in K\times\tilde{K}$ to obtain
 \begin{align*}
   |D^\alpha f_x|\kappa
  &\leq\sum_{\beta\leq\alpha}\binom{\alpha}{\beta}
   \big|D^\beta\big(\tfrac{r}{\kappa}\big)D^{\alpha-\beta}(1-\zeta_x)\big|\kappa 
   =r\sum_{\beta\leq\alpha}\binom{\alpha}{\beta}\big|(D^\beta\tfrac{1}{\kappa})\kappa\big|
   \delta^{|\alpha-\beta|}|D^{\alpha-\beta}\xi| \\
  &\leq\delta r\sum_{\beta<\alpha}\binom{\alpha}{\beta}\|D^\beta\tfrac{1}{\kappa}\|_\kappa
   \|D^{\alpha-\beta}\xi\|_\infty 
   +r\big|(D^\alpha\tfrac{1}{\kappa})\kappa\big|(1-\zeta_x) \\
 &\leq\sup_{y\in\tilde{K}^c}r\big|(D^\alpha\tfrac{1}{\kappa})\kappa\big|(y)+\tfrac{\epsilon}{2}
  \leq\epsilon.
 \end{align*}
 Hence, for $c:=\sum_{l=1}^N\#\{\alpha\in\N_0^d\colon |\alpha|=l\}$ and $\epsilon\leq\frac{\epsilon_0}{c}$, 
 inequality~\eqref{eq:cutoff2a} yields
 \[ \left\|\frac{S_i(h)f_x-f_x}{h}\right\|_{\kappa,X_i}\leq\rho(c\epsilon) 
		\quad\mbox{for all } i\in I \mbox{ and } h\in (0,h_i]\cap\T_i.\]
 Since $\rho(c\epsilon)\to 0$ as $\epsilon\to 0$, the claim follows from Lemma~\ref{cor:cutoff1}.
\end{proof}

For Chernoff-type approximations, the verification of condition~(iii) only requires 
knowledge of the one-step operators.\ However, in continuous time, it is important to 
replace condition~(iii) by a condition on the infinitesimal generators.\ In what follows, let \[ A_i\colon D(A_i)\to\Ck(X),\; f\mapsto\lim_{h\downarrow 0}\frac{S_i(h)f-f}{h}, \]
 where the domain $D(A_i)$ consists of all $f\in\Ck(X)$ such that the previous limit exists.

\begin{corollary} \label{cor:cutoff3}
Let $\T_i:=\R_+$ for all $i\in I$ and
 assume that, for every $\epsilon>0$, $r\geq 0$ and $K\Subset\Rd$, there exist a family
 $(\zeta_x)_{x\in K}$ of continuous functions $\zeta_x\colon\Rd\to\R$ and $\tilde{K}\Subset\Rd$ 
 such that
 \begin{enumerate}
  \item[(i)] $0\leq\zeta_x\leq 1$ and $\zeta_x(x)=1$ for all $x\in K$,
  \item[(ii)] $\sup_{y\in\tilde{K}^c}\zeta_x(y)\leq\epsilon$ for all $x\in K$,
  \item[(iii)] the function $f_x:=\frac{r}{\kappa}(1-\zeta_x)$ satisfies $f_x\in D(A_i)$,
   $\|A_i f_x\|_\kappa\leq\epsilon$ and 
   \begin{equation} \label{eq:cutoff.gen}
    \lim_{h\downarrow 0}\left\|\frac{S_i(h)f_x-f_x}{h}-A_i f_x\right\|_{\kappa,X_i}=0 
    \quad\mbox{for all } i\in I \mbox{ and } x\in K. 
   \end{equation}
 \end{enumerate}
 Then, for every $\epsilon>0$, $r,T\geq 0$ and $K\Subset\Rd$, there exist $c\geq 0$ 
 and $K'\Subset\Rd$ with
 \[ \|S_i(t)f-S_i(t)g\|_{\infty,K_i}\leq c\|f-g\|_{\infty,K'_i}+\epsilon \]
 for all $i\in I$, $t\in [0,T]$ and $f,g\in B_{\Ck(X_i)}(r)$.
\end{corollary}
\begin{proof}
 Let $\epsilon>0$, $r,T\geq 0$ and $K\Subset\Rd$.\ Choose $c\geq 0$ such that
 inequality~\eqref{eq:Lip} is valid with $r+\epsilon$ and a family $(\zeta_x)_{x\in K}$ of
 continuous functions $\zeta_x\colon\Rd\to\R$ satisfying the conditions~(i)-(iii) with
 $\nicefrac{\epsilon}{c}$.\ Let $x\in K$ and $f_x:=\frac{r}{\kappa}(1-\zeta_x)$.\ Similar
 to the proof of Theorem~\ref{thm:Sn2}, one can use inequality~\eqref{eq:Lip} and 
 equation~\eqref{eq:cutoff.gen} to show that
 \begin{align*}
  (S_i(t)f_x-f_x)\kappa &\leq\int_0^t \big(S_i(s)(f_x+A_i f_x)-S_i(s)f_x\big)\kappa\,\d s, \\
  (S_i(t)f_x-f_x)\kappa &\geq -\int_0^t \big(S_i(s)(f_x-A_i f_x)-S_i(s)f_x\big)\kappa\,\d s
 \end{align*}
 for all $i\in I$ and $t\geq 0$. Hence, inequality~\eqref{eq:Lip} and condition~(iii) imply
 \[ \|S_i(t)f_x-f_x\|_\kappa\leq c\|A_i f_x\|_\kappa t\leq\epsilon t \]
 for all $i\in I$ and $t\in [0,T]$.\ The claim follows from Lemma~\ref{cor:cutoff1}.
 \end{proof}

Note that the norm convergence in equation~\eqref{eq:cutoff.gen} is essential and 
cannot be replaced by convergence w.r.t. the mixed topology.
	
\begin{corollary} \label{cor:cutoff4}
 Let $\T_i:=\R_+$ for all $i\in I$.\ Assume that, for every $r\geq 0$, there exist $\epsilon_0>0$,
 $N\in\N$ and a non-decreasing function $\rho\colon\R_+\to\R_+$ with $\rho(\epsilon)\to 0$ 
 as $\epsilon\to 0$ such that $f:=\frac{r}{\kappa}(1-\zeta)\in D(A_i)$ with
 \[ \lim_{h\downarrow 0}\left\|\frac{S_i(h)f-f}{h}-A_i f\right\|_{\kappa,X_i}=0
    \quad\text{and}\quad 
    \|A_i f\|_{\kappa,X_i}\leq\rho\bigg(\sum_{l=1}^N\|D^l f\|_\kappa\bigg) \]
 for all $i\in I$ and $\zeta\in\Cci(\Rd)$ satisfying $0\leq\zeta\leq 1$ 
 and $\|D^l f\|_\kappa\leq\epsilon_0$ for all $1\leq l\leq N$. Furthermore, the function 
 $\nicefrac{1}{\kappa}$ is $N$-times continuously differentiable and, for every $\epsilon>0$, 
 there exists $K\Subset\Rd$ with $\sup_{x\in K^c}|(D^\alpha\tfrac{1}{\kappa})\kappa|(x)\leq\epsilon$
 for all $1\leq |\alpha|\leq N$.\ Then, for every $\epsilon>0$, $r,T\geq 0$ and $K\Subset\Rd$, 
 there exist $c\geq 0$ and $K'\Subset\Rd$ with
 \[ \|S_i(t)f-S_i(t)g\|_{\infty,K_i}\leq c\|f-g\|_{\infty,K_i'}+\epsilon \]
 for all $i\in I$, $t\in [0,T]$ and $f,g\in B_{\Ck(X_i)}(r)$.
\end{corollary} 
\begin{proof}
 Using Corollary~\ref{cor:cutoff3}, one can proceed similar as in Corollary \ref{cor:cutoff2}.
\end{proof}

We conclude with the observation that verifying the continuity of $S_i(t)f$ for all $f\in\Ck(X_i)$ 
can often be reduced to verifying this condition for all $f\in\Lipb(X_i)$.

\begin{remark} \label{rem:cutoff}
 Typically, it is easy to show that $S_i(t)\colon\Ck(X_i)\to\Fk(X_i)$ but verifying the
 continuity of $S_i(t)f$ for all $f\in\Ck(X_i)$ might be tricky.\ However, for 
 Lemma~\ref{cor:cutoff1} and the Corollaries~\ref{cor:cutoff2}-\ref{cor:cutoff4} 
 to be valid, it is sufficient that
 \[ S_i(t)(\tfrac{r}{\kappa}(1-\zeta_x))\in\Ck(X_i) \quad\mbox{and}\quad 
    S_i(s+t)(\tfrac{r}{\kappa}(1-\zeta_x))=S_i(s)S_i(t)(\tfrac{r}{\kappa}(1-\zeta_x)) \]
 If $\kappa\equiv 1$, one can usually choose $\zeta_x\in\Lipb(X_i)$ 
 and show $S_i(t)\colon\Lipb(X_i)\to\Lipb(X_i)$. 
 Furthermore, if $S_i(t)\colon\Lipb(X_i)\to\Lipb(X_i)$ and inequality~\eqref{eq:cutoff1} has been
 verified for all $f,g\in\Lipb(X_i)$, it follows that $S_i(t)\colon\Ck(X_i)\to\Ck(X_i)$.\
 Indeed, for every $f\in\Ck(X_i)$, we can choose a sequence $(f_k)_{k\in\N}\subset\Lipb(X_i)$
 with $f_k\to f$ to obtain $S_i(t)f=\lim_{k\to\infty}S_i(t)f_k\in\Ck(X_i)$ for all $i\in I$ 
 and $t\in\T_i$.
\end{remark}

	\section{Examples}
	\label{sec:examples}

	\subsection{Euler formula and Yosida approximation for upper envelopes}
	\label{sec:yosida}

	In this subsection, we derive an Euler formula and a Yosida-type 
	approximation for upper envelopes of families of linear semigroups.\ The construction uses the
	supremum of linear resolvents rather than the resolvent of the supremum 
	of linear operators as it is common both in the classical theory of m-accretive 
	operators and the theory of viscosity solutions.\ We refer to~\cite{bff} and~\cite{kuhnemund} 
	for an Euler formula for bi-continuous linear semigroups as well as to~\cite{MR1293091} 
	for a Hille--Yosida theorem for weakly continuous linear semigroups on the space of 
	bounded uniformly continuous functions and to~\cite{pazy} for the classical Yosida 
	approximation. To that end, we consider a non-empty family $(A^\theta)_{\theta\in\Theta}$ 
	of linear operators $A^\theta\colon D(A^\theta)\to\Ck(\Rd)$ with domain 
	$D(A^\theta)\subset \Ck(\Rd)$ and resolvent $\varrho(A^\theta)\subset\mathbb{C}$. 
	In the sequel, we provide certain nonlinear versions of the classical Euler formula and 
	the Yosida approximation in order to construct a strongly continuous convex 
	monotone semigroup $(S(t))_{t\geq 0}$ on $\Ck(\Rd)$ with generator
	\[ Af=\sup_{\theta\in\Theta}A^\theta f \quad\mbox{for all } f\in\Cbi(\Rd). \]

	\begin{assumption} \label{ass:theta}
		Suppose that the conditions~(i)--(iv) or the conditions~(i),~(ii'), (iii) and (iv) 
		from the following list are satisfied: 
		\begin{itemize}
			\item[(i)] There exists $\omega\in \R$ with 
			$(\omega,\infty)\subset \bigcap_{\theta\in \Theta}\varrho(A^\theta)$ such that
			\begin{align*}
				\|(\lambda-\omega)(\lambda-A^\theta)^{-1}f\|_{\infty} &\leq \|f\|_\infty
				\quad\mbox{for all } \theta \in \Theta\text{ and }f\in\Lipb(\Rd), \\
				\|(\lambda-\omega)(\lambda-A^\theta)^{-1}f\|_{\kappa} &\leq \|f\|_\kappa
				\quad\mbox{for all } \theta \in \Theta \text{ and }f\in \Ck(\Rd),
			\end{align*}
            and $(\lambda-A^\theta)^{-1}\colon\Ck(\Rd)\to\Ck(\Rd)$
			is monotone for all $\lambda\in (\omega,\infty)$ and $\theta\in \Theta$.
			\item[(ii)] There exists a bounded continuous function $\tilde{\kappa}\colon \Rd\to (0,\infty)$ 
			with 
			\[ \lim_{|x|\to \infty}\frac{\tilde{\kappa}(x)}{\kappa(x)}=0 \quad\mbox{and}\quad
			\|(\lambda-\omega)(\lambda-A^\theta)^{-1}f\|_{\tilde{\kappa}}\leq \|f\|_{\tilde{\kappa}} \]
			for all $\lambda\in (\omega,\infty)$ and $f\in\Ck(\Rd)$. 
			\item[(ii')] For every $\epsilon>0$, $r\geq 0$ and $K\Subset\Rd$, there exist a family
			$(\zeta_x)_{x\in K}$ of continuous functions $\zeta_x\colon\Rd\to\R$  
			and $\tilde{K}\Subset\Rd$ such that
			\begin{enumerate}[label=(\alph*)]
				\item $0\leq\zeta_x\leq 1$ and $\zeta_x(x)=1$ for all $x\in K$,
				\item $\sup_{y\in\tilde{K}^c}\zeta_x(y)\leq\epsilon$ for all $x\in K$,
				\item $\frac{r}{\kappa}(1-\zeta_x)\in\bigcap_{\theta\in \Theta}D(A^\theta)$ and 
				$\sup_{\theta\in \Theta}\|A^\theta(\frac{r}{\kappa}(1-\zeta_x))\|_\kappa\leq\epsilon$ 
				for all $x\in K$.
			\end{enumerate}
			\item[(iii)] There exists $L\geq 0$ with 
			\[ (\lambda-\omega)^2\|\tau_x(\lambda-A^\theta)^{-1} f-(\lambda-A^\theta)^{-1} (\tau_x f)\|_{\infty}
			\leq Lr|x| \]
			for all $r\geq0$, $f\in \Lipb(\Rd,r)$, $x\in \Rd$, $\lambda\in (\omega,\infty)$ and $\theta\in \Theta$. 
			\item[(iv)] It holds $\Cbi(\Rd)\subset \bigcap_{\theta\in \Theta}D(A^\theta)$ and,
			for every $f\in\Cbi(\Rd)$ and $K\Subset\Rd$, 
			\[ \sup_{\theta\in \Theta}\|A^\theta f\|_\kappa<\infty
			\quad\mbox{and}\quad
			\lim_{\lambda\to\infty}\sup_{\theta\in \Theta}\|(\lambda-\omega)(\lambda-A^\theta)^{-1}A^\theta f-A^\theta f\|_{\infty,K}=0. \]
		\end{itemize}	
	\end{assumption}

	The previous conditions are, for instance, satisfied for suitably bounded families of 
	generators of L\'evy processes, Ornstein--Uhlenbeck processes and geometric 
	Brownian motions.\ In these cases, it is straightforward to verify the respective conditions 
	for the transition semigroups and use the Laplace transform to transfer them
	to the resolvents.\ For a detailed illustration how the conditions can be verified for 
	a large class of transition semigroups, we refer to~\cite[Section~6.3 and~Section 7]{NR21}.\
	At this point we want to emphasize that the nonlinear semigroups in~\cite{NR21} 
	are not constructed via the resolvents but as so-called Nisio semigroups.\ Furthermore, 
	due to Corollary~\ref{cor:cutoff.app} and Lemma~\ref{cor:cutoff1}, the conditions~(ii) and~(ii') 
	both guarantee that Assumption~\ref{ass:cher}(iii) is valid.\ While condition~(ii) is satisfied 
	for Ornstein--Uhlenbeck processes and geometric Brownian motions, it does not cover
	L\'evy processes with non-integrable jumps.\ On the other hand, condition~(ii') applies to 
	L\'evy processes with possibly non-integrable jumps but neither to Ornstein--Uhlenbeck 
	processes nor geometric Brownian motions. In the sequel, let $(\lambda_n)_{n\in\N}\subset (\omega,\infty)$ 
	be a sequence with $\lambda_n\to\infty$ and
	\[ h_n:=\frac{1}{\lambda_n-\omega} \quad\mbox{for all } n\in \N. \]
	For every $n\in\N$ and $f\in\Ck(\Rd)$, we define
	\[ I_n f:=\sup_{\theta\in\Theta}\lambda_n(\lambda_n-A^\theta)^{-1}f. \]
	Moreover, let $X_n:=\Rd$ and $\T_n:=\{kh_n\colon k\in\N_0\}$ for all $n\in\N$.  
	For every $n\in\N$ and $f\in\Ck(\Rd)$, the resolvent identity implies
	\[ I_n f=f+\sup_{\theta\in\Theta}A^\theta(\lambda_n-A^\theta)^{-1} f, \]
	and, for every $f\in\bigcap_{\theta\in\Theta}D(A^\theta)$, it follows that
	\begin{equation} \label{eq:resolvent1}
		I_n f=f+\sup_{\theta\in\Theta}(\lambda_n-A^\theta)^{-1} A^\theta f.
	\end{equation}

	\begin{theorem} \label{thm:euler}
		Suppose that Assumption~\ref{ass:theta} is satisfied.\ Then, there exists a strongly 
		continuous convex monotone semigroup $(S(t))_{t\geq 0}$ on $\Ck(\Rd)$ 
        with generator $A\colon D(A)\to\Ck(\Rd)$ 
        such that
		\[ S(t)f=\lim_{n\to\infty}I_n^{k_n^t}f \quad\mbox{for all } (f,t)\in\Ck(\Rd)\times\R_+,\]
		where $(k_n^t)_{n\in\N}\subset\N$ is an arbitrary sequence satisfying $k_n^t h_n\to t$.
		It holds
		\[ \Cbi(\Rd)\subset D(A) \quad\mbox{and}\quad
		Af=\sup_{\theta\in\Theta}A^\theta f \quad\mbox{for all } f\in\Cbi(\Rd).\]
	\end{theorem} 
	\begin{proof}
		We verify Assumption~\ref{ass:cher} and the conditions from Corollary~\ref{cor:unique2}.\ 
		First, we show that $I_n\colon\Lipb(\Rd,r)\to\Lipb(\Rd, e^{\beta h_n}r)$ for all $n\in\N$
		and $r\geq 0$, where 
		\[ \beta:=\omega+\frac{L\lambda_0}{\lambda_0-\omega} \quad\mbox{and}\quad
		\lambda_0:=\min_{n\in\N}\lambda_n>\omega. \]
		For every $n\in\N$ and $f,g\in\Lipb(\Rd)$, it follows from Assumption~\ref{ass:theta}(i) 
		that
		\[ \|I_n f-I_n g\|_\infty\leq\frac{\lambda_n}{\lambda_n-\omega}\|f-g\|_\infty
		=(1+\omega h_n)\|f-g\|_\infty\leq e^{\omega h_n}\|f-g\|_\infty. \]
		Moreover, it holds $I_n 0=0$ and Assumption~\ref{ass:theta}(iii) implies 
		\begin{equation} \label{eq:theta.tau}
			\|I_n(\tau_x f)-\tau_x I_n f\|_\infty
			\leq\frac{L\lambda_n}{\lambda_n-\omega}rh_n|x|
			\leq\frac{L\lambda_0}{\lambda_0-\omega}rh_n|x|
		\end{equation}
		for all $n\in\N$, $r\geq 0$, $f\in\Lipb(\Rd,r)$ and $x\in\Rd$. Hence, for every $x,y\in\Rd$,
		\begin{align*}
			&|(I_n f)(x+y)-(I_n f)(y)|
			\leq |\tau_x I_n f-I_n(\tau_x f)|(y)+|I_n(\tau_x f)-I_n f|(y) \\
			&\leq\frac{\lambda_0}{\lambda_0-\omega}Lrh_n|x| +(1+\omega h_n)r|x| =\left(1+\left(\omega+\frac{L\lambda_0}{\lambda_0-\omega}\right)h_n\right)r|x|
			\leq e^{\beta h_n}r|x|. 
		\end{align*}
		We obtain that $I_n f\in\Lipb(\Rd, e^{\beta h_n}r)$ and therefore Assumption~\ref{ass:cher}(iv)
		is satisfied. Furthermore, inequality~\eqref{eq:theta.tau} 
		yields condition~(i) from Corollary~\ref{cor:unique2}. 
		
		Second, we show that $I_n\colon\Ck(\Rd)\to\Ck(\Rd)$. 
		Assumption~\ref{ass:theta}(i) yields
		\[ \|I_n f-I_n g\|_\kappa\leq\frac{\lambda_n}{\lambda_n-\omega}\|f-g\|_\kappa
		=(1+\omega h_n)\|f-g\|_\kappa\leq e	^{\omega h_n}\|f-g\|_\kappa \]
		and thus Assumption~\ref{ass:cher}(ii) is valid.\ Moreover, the operators
		$I_n\colon\Ck(\Rd)\to\Fk(\Rd)$ are convex and monotone with $I_n 0=0$.\
		Let $f\in\Ck(\Rd)$ and $(f_k)_{k\in\N}\subset\Lipb(\Rd)$ 
		with $f_k\to f$.\ If Assumption~\ref{ass:theta}(ii) is valid, 
		Corollary~\ref{cor:moment.app} implies $I_n f=\lim_{k\to\infty}I_n f_k$
		and therefore $I_n\colon\Ck(\Rd)\to\Ck(\Rd)$.\ In addition, it holds
		$\|I_n^k f\|_{\tilde{\kappa}}\leq e^{\omega kh_n}\|f\|_{\tilde{\kappa}}$ 
		for all $k,n\in\N$ and $f\in\Ck(\Rd)$.\ Hence, Corollary~\ref{cor:moment.app}
		guarantees that Assumption~\ref{ass:cher}(iii) is satisfied.\ If Assumption~\ref{ass:theta}(ii')
		is valid, we can use Lemma~\ref{cor:cutoff1} and Remark~\ref{rem:cutoff} instead. 
		Indeed, equation~\eqref{eq:resolvent1} and Assumption~\ref{ass:theta}(i) imply
		\begin{align*}
			\big\|I_n\big(\tfrac{r}{\kappa}(1-\zeta_x)\big)-\tfrac{r}{\kappa}(1-\zeta_x)\big\|_\kappa 
			&\leq\sup_{\theta\in\Theta}\big\|(\lambda_n-A^\theta)^{-1}A^\theta\big(\tfrac{r}{\kappa}(1-\zeta_x)\big)\big\|_\kappa \\
			&\leq \sup_{\theta\in\Theta}\big\|A^\theta\big(\tfrac{r}{\kappa}(1-\zeta_x)\big)\big\|_\kappa h_n
			\leq\epsilon h_n.
		\end{align*}
		
		Third, we verify condition~(ii) from Corollary~\ref{cor:unique2} and thus Assumption~\ref{ass:cher}(v).
        Let $f\in \Cbi(\Rd)$ and define
		\[ A_n f:=\frac{I_n f-f}{h_n} \quad\mbox{for all } n\in\N. \]
		For every $n\in\N$, it follows from equation \eqref{eq:resolvent1} that
		\[ A_n f=\sup_{\theta\in \Theta}\big((\lambda_n-\omega)\lambda_n(\lambda_n-A^\theta)^{-1} f\big)-(\lambda_n-\omega)f
		=\sup_{\theta\in \Theta}(\lambda_n-\omega)(\lambda_n-A^\theta)^{-1}A^\theta f. \]
		Hence, for every $n\in\N$ and $K\Subset\Rd$, Assumption~\ref{ass:theta}(i) yields
		\begin{align*}
			\|A_nf\|_\kappa 
			&\leq\sup_{\theta\in \Theta}\|(\lambda_n-\omega)(\lambda_n-A^\theta)^{-1}A^\theta f\|_\kappa
			\leq \sup_{\theta\in \Theta}\|A^\theta f\|_\kappa, \\
			\|A_nf-Af\|_{\infty,K}
			&\leq\sup_{\theta\in\Theta}\|(\lambda_n-\omega)(\lambda_n-A^\theta)^{-1}A^\theta f-A^\theta f\|_{\infty,K}.
		\end{align*}
		Assumption~\ref{ass:theta}(iv) yields that condition~(ii) from Corollary~\ref{cor:unique2} is valid.\
		Now, the claim follows from Theorem~\ref{thm:cher}, where property~(iv) holds with the
		constants~$L$ and $\omega+L$ rather than $\frac{L\lambda_0}{\lambda_0-\omega}$
		and~$\beta$ since $\lambda_n\to\infty$, and from Corollary~\ref{cor:unique2}.
	\end{proof}

	For the Yosida approximation, we need norm convergence of the resolvents.

	\begin{assumption} \label{ass:theta2}
		Suppose that the conditions~(i)--(iii) or the conditions~(i),~(ii') and~(iii) from 
		Assumption~\ref{ass:theta} are satisfied. Moreover, the following statement is valid: 
		\begin{itemize}
			\item[(iv')] It holds $\Cbi(\Rd)\subset\bigcap_{\theta\in \Theta}D(A^\theta)$ and, 
			for every $f\in \Cbi(\Rd)$,
			\[ \sup_{\theta\in \Theta}\|A^\theta f\|_\kappa<\infty
			\quad\mbox{and}\quad
			\lim_{\lambda\to\infty}\sup_{\theta\in\Theta}
			\|(\lambda-\omega)(\lambda-A^\theta)^{-1}A^\theta f-A^\theta f\|_\kappa=0. \]
		\end{itemize}
	\end{assumption}
	For every $n\in\N$ and $\theta\in\Theta$, we define 
	\[ A_n^\theta:=(\lambda_n-\omega)A^\theta(\lambda_n-A^\theta)^{-1}. \]
	Since $A_n^\theta\colon\Ck(\Rd)\to\Ck(\Rd)$ is a bounded linear operator, 
	we can define 
	\[ J_n f:=\sup_{\theta\in \Theta}e^{h_n A^\theta_n}f
	\quad\mbox{for all } n\in\N \mbox{ and } f\in\Ck(\Rd), \]
	where the operator exponential is given as the power series
	$ e^{h_nA^\theta_n}:=\sum_{k=0}^\infty\frac{(h_nA_n^\theta)^k}{k!}$. 
	Since the resolvent identity 
	\begin{equation} \label{eq:resolvent2}
		h_nA_n^\theta f=\lambda_n(\lambda_n-A^\theta)^{-1}f-f
	\end{equation}
	holds for all $n\in\N$, $f\in\Ck(\Rd)$ and $\theta\in\Theta$, we obtain
	\begin{equation} \label{eq:yosidarep2}
		J_{n}f=e^{-1}\sup_{\theta\in \Theta}e^{\lambda_n(\lambda_n-A^\theta)^{-1}}f  
		\quad\mbox{for all } n\in\N \mbox{ and } f\in \Ck(\Rd).
	\end{equation}

	\begin{theorem}
		Suppose that Assumption~\ref{ass:theta2} is satisfied.\ Then, there exists a strongly 
		continuous convex monotone semigroup $(T(t))_{t\geq 0}$ on $\Ck(\Rd)$ 
        with generator $B\colon D(B)\to\Ck(\Rd)$ given by
		\begin{equation} 
			T(t)f=\lim_{n\to\infty}J_n^{k_n^t}f \quad\mbox{for all } (f,t)\in\Ck(\Rd)\times\R_+,
		\end{equation}
		where $(k_n^t)_{n\in\N}\subset\N$ is an arbitrary sequence satisfying $k_n^t h_n\to t$.
		It holds
		\[ \Cbi(\Rd)\subset D(B) \quad\mbox{and}\quad
		Bf=\sup_{\theta\in\Theta}A^\theta f \quad\mbox{for all } f\in\Cbi(\Rd). \]
		In particular, we obtain $S(t)f=T(t)f$ for all $t\geq 0$ and $f\in\Ck(\Rd)$.
	\end{theorem} 
	\begin{proof}
		First, for every $n\in \N$, $r\geq 0$, $f\in \Lipb(\Rd,r)$ and $x\in \Rd$, 
		we show that 
		\begin{equation} \label{eq:yosida.tau}
			\|\tau_x J_n f- J_n(\tau_x f)\|_\infty
			\leq\frac{L\lambda_0 e^{\beta h_0}}{\lambda_0-\omega}rh_n|x|, 
		\end{equation}
		where $\beta:=\omega+\frac{L\lambda_0}{\lambda_0-\omega}$,
		$\lambda_0:=\min_{n\in\N}\lambda_n>\omega$ and $h_0:=\frac{1}{\lambda_0-\omega}=\sup_{n\in\N}h_n$.\
	 We define $B_{n,\theta}:=\lambda_n(\lambda_n-A^\theta)^{-1}$
		for all $n\in\N$ and $\theta\in\Theta$.\ Then, for every $n\in \N$, $r\geq 0$, $f\in \Lipb(\Rd,r)$,
		$x\in \Rd$ and $\theta\in\Theta$, it follows from Assumption~\ref{ass:theta}(iii) that
		\begin{equation} \label{eq:yosida.tau1}
			\|\tau_x B_{n,\theta}f-B_{n,\theta}(\tau_x f)\|_\infty
			\leq\frac{L\lambda_0}{\lambda_0-\omega}rh_n|x|.
		\end{equation}
		Furthermore, by induction, Assumption~\ref{ass:theta}(i) implies
		\begin{equation} \label{eq:yosida.tau2}
			\|B_{n,\theta}^kf-B_{n,\theta}^kg\|_\infty\leq (1+\omega h_n)^k\|f-g\|_\infty
		\end{equation}
		for all $k,n\in\N$, $f,g\in\Lipb(\Rd)$ and $\theta\in\Theta$.\ Hence, similar to the proof
		of Theorem~\ref{thm:euler}, one can show that
		\begin{equation} \label{eq:yosida.tau3}
			B_{n,\theta}^k\colon\Lipb(\Rd,r)\to\Lipb(\Rd,(1+\beta h_n)^k r)
		\end{equation}
		for all $k,n\in\N$, $r\geq 0$ and $\theta\in\Theta$. Combining the 
		inequalities~\eqref{eq:yosida.tau1}--\eqref{eq:yosida.tau3} yields
		\begin{align*}
			\|\tau_x B_{n,\theta}^k f-B_{n,\theta}^k(\tau_x f)\|_\infty
			&\leq\sum_{l=1}^k \|B_{n,\theta}^{k-l}(\tau_x B_{n,\theta}^l f)
			-B_{n,\theta}^{k-l}B_{n,\theta}(\tau_x B_{n,\theta}^{l-1}f)\|_\infty \\
			&\leq\sum_{l=1}^k (1+\omega h_n)^{k-l}
			\|\tau_x B_{n,\theta}B_{n,\theta}^{l-1}f-B_{n,\theta}(\tau_x B_{n,\theta}^{l-1}f)\|_\infty \\
			&\leq\sum_{l=1}^k (1+\omega h_n)^{k-l}\frac{L\lambda_0}{\lambda_0-\omega}
			(1+\beta h_n)^{l-1}rh_n|x| \\
			&\leq k(1+\beta h_n)^{k-1}\frac{L\lambda_0}{\lambda_0-\omega}rh_n|x|
		\end{align*}
		for all $n\in\N$, $r\geq 0$, $f\in\Lipb(\Rd,r)$, $x\in\Rd$ and $\theta\in\Theta$. 
		Equation~\eqref{eq:yosidarep2} implies
		\begin{align*}
			\|\tau_x J_n f-J_n(\tau_x f)\|_\infty
			&\leq e^{-1}\sup_{\theta\in\Theta}\|\tau_x e^{B_{n,\theta}}-e^{B_{n,\theta}}(\tau_x f)\|_\infty \\
			&\leq e^{-1}\sup_{\theta\in\Theta}\sum_{k=0}^\infty
			\frac{\|\tau_x B_{n,\theta}^k f-B_{n,\theta}^k(\tau_x f)\|_\infty}{k!} \\
			&\leq e^{-1}\sum_{k=1}^\infty \frac{(1+\beta h_n)^{k-1}}{(k-1)!}
			\frac{L\lambda_0}{\lambda_0-\omega}rh_n|x| \\
			&=\frac{L\lambda_0}{\lambda_0-\omega}e^{\beta h_n}rh_n|x|
			\leq\frac{L\lambda_0 e^{\beta h_0}}{\lambda_0-\omega}rh_n|x|
		\end{align*}
		for all $n\in\N$, $r\geq 0$, $f\in\Lipb(\Rd,r)$ and $x\in\Rd$. 
		
		Second, we verify Assumption~\ref{ass:cher} and the conditions from Corollary~\ref{cor:unique2}.\
        We have already shown that condition~(i) from Corollary~\ref{cor:unique2} is satisfied.\
        Moreover, it follows from equation~\eqref{eq:yosidarep2} and equation~\eqref{eq:yosida.tau3} that 
		\[ J_n\colon\Lipb(\Rd,r)\to\Lipb(\Rd, e^{\beta h_n}r) 
		\quad\mbox{for all } n\in\N \mbox{ and } r\geq 0. \]
		Similarly, one can derive from Assumption~\ref{ass:theta}(i) that
		\[ \|J_n f-J_n g\|_\kappa\leq e^{\omega h_n}\|f-g\|_\kappa
		\quad\mbox{for all } n\in\N \mbox{ and } f,g\in\Ck(\Rd). \]
		Clearly, the operators $J_n\colon\Ck(\Rd)\to\Fk(\Rd)$ are convex and monotone with $J_n 0=0$.
		Let $f\in\Ck(\Rd)$ and $(f_k)_{k\in\N}\subset\Lipb(\Rd)$ 
		with $f_k\to f$.\ If Assumption~\ref{ass:theta}(ii) is valid, 
		Corollary~\ref{cor:moment.app} implies $J_n f=\lim_{k\to\infty}J_n f_k$
		and thus $J_n\colon\Ck(\Rd)\to\Ck(\Rd)$.\ Moreover, it holds
		$\|J_n^k f\|_{\tilde{\kappa}}\leq e^{\omega kh_n}\|f\|_{\tilde{\kappa}}$ 
		for all $k,n\in\N$ and $f\in\Ck(\Rd)$.\ Hence, Corollary~\ref{cor:moment.app}
		guarantees that Assumption~\ref{ass:cher}(iii) is valid.\ If Assumption~\ref{ass:theta}(ii')
		is satisfied, we can use Lemma~\ref{cor:cutoff1} and Remark~\ref{rem:cutoff} instead. 
		Indeed, Assumption~\ref{ass:theta}(i) implies 
		\begin{align*}
			\big\|J_n\big(\tfrac{r}{\kappa}(1-\zeta_x)\big)-\tfrac{r}{\kappa}(1-\zeta_x)\big\|_\kappa
			&\leq \sup_{\theta\in\Theta}\int_0^1\big\|e^{sh_n A_n^\theta}h_n A_n^\theta
			\big(\tfrac r\kappa (1-\zeta_x)\big)\big\|_\kappa\, \d s \\
			&\leq \sup_{\theta\in\Theta}e^{\omega h_n}h_n\big\|A^\theta\big(\tfrac r\kappa (1-\zeta_x)\big)\big\|_\kappa
			\leq e^{\omega h_0}\epsilon h_n. 
		\end{align*}
		It remains to verify condition~(ii) from Corollary~\ref{cor:unique2} and thus Assumption~\ref{ass:cher}(v).
		We use Assumption~\ref{ass:theta}(i) and equation~\eqref{eq:resolvent2} to estimate
		\begin{align*}
			\Big\|\frac{e^{h_n A_n^\theta}f-f}{h_n}-A_n^\theta f\Big\|_\kappa
			&\leq\int_0^1 \|e^{sh_n A_n^\theta}A_n^\theta f-A_n^\theta f\|_\kappa\,\d s
            \leq\int_0^1 \|e^{sh_n A_n^\theta}A^\theta f-A^\theta f\|_\kappa\,\d s \\
			&\leq e^{\omega h_n}h_n\|A_n^\theta A^\theta f\|_\kappa =e^{\omega h_n}\|\lambda_n(\lambda_n-A^\theta)^{-1}A^\theta f-A^\theta f\|_\kappa \\
			&\leq e^{\omega h_n}\|(\lambda-\omega)(\lambda_n-A^\theta)^{-1}A^\theta f-A^\theta f\|_\kappa
			+\omega h_n\|A^\theta f\|_\kappa.
		\end{align*}
        for all $n\in\N$, $f\in \Cbi(\Rd)$ and $\theta\in\Theta$.
		Assumption~\ref{ass:theta2}(iv') and equation~\eqref{eq:yosidarep2} imply
		\[ \lim_{n\to\infty}\left\|\frac{J_n f-f}{h_n}-\sup_{\theta\in\Theta}A^\theta f\right\|_\kappa=0. \]
		Now, the claim follows from Theorem~\ref{thm:cher} and Corollary~\ref{cor:unique2}.
	\end{proof}

	\subsection{Vanishing viscosity and stability of convex HJB equations}
	\label{sec:VanVis}

	We provide a stability result for convex HJB equations which covers, in particular, the vanishing 
    viscosity method.\ Let $\kappa\equiv 1$ and let $\Cb(\Rd)$ be the space
	of all bounded continuous functions $f\colon\Rd\to\R$.\
	Let $(\varphi_n)_{n\in\N}$ be a sequence of functions $\varphi_n\colon\S^d_+\times\Rd\to [0,\infty]$,
	where $\S^d_+$ consists of all symmetric positive semi-definite $d\times d$-matrices,
	and define
	\[ H_n(a,b):=\sup_{(\tilde{a},\tilde{b})\in\S^d_+\times\Rd}
        \Big(\frac{1}{2}\tr(a\tilde{a})+b^T\tilde{b}-\varphi_n(\tilde{a},\tilde{b})\Big)
	   \quad\mbox{for all } (a,b)\in\R^{d\times d}\times\Rd. \]
	Here, we denote by $\tr(a\tilde{a})$ the matrix trace and by $b^T\tilde{b}$ the Euclidean inner product.
	Under the condition $H_n\to H$, we will show that a sequence of semigroups
	$(S_n(t))_{t\geq 0}$, whose generators are given by 
	\[ (A_n f)(x)=H_n(D^2 f(x), Df(x)) \quad\mbox{for all } f\in\Cbi(\Rd), \]
	converges to a semigroup $(S(t))_{t\geq 0}$ with generator $Af=H(D^2 f, Df)$.\ Here and throughout, $Df$ denotes the gradient of $f$ and $D^2f$ its Hessian matrix.

	\begin{assumption} \label{ass:Hn}
		Suppose that the following conditions are satisfied:
		\begin{enumerate}
			\item[(i)] For every $n\in\N$, there exists $(a,b)\in\S^d_+\times\Rd$ with $\varphi_n(a,b)=0$. 
			\item[(ii)] The penalization functions grow superlinearly, i.e., 
			\[ \lim_{|a|+|b|\to\infty}\inf_{n\in\N}\frac{\varphi_n(a,b)}{|a|+|b|}=\infty. \]
			\item[(iii)] There exists a function $H\colon\R^{d\times d}\times\Rd\to\R$ with
			\[ \lim_{n\to\infty}\sup_{(a,b)\in K}|H_n(a,b)-H(a,b)|=0 
			\quad\mbox{for all } K\Subset\R^{d\times d}\times\Rd. \]
		\end{enumerate}
	\end{assumption}

	Condition~(ii) guarantees that $(H_n)_{n\in\N}$ is a sequence of 
	real-valued functions.

	\begin{theorem}
		Suppose that Assumption~\ref{ass:Hn} is valid.\ Then, there exists a sequence
		$(S_n)_{n\in\N}$ of strongly continuous convex monotone semigroups $(S_n(t))_{t\geq 0}$
		on $\Cb(\Rd)$ with generators $A_n\colon D(A_n)\to\Cb(\Rd)$ such that
		\[ \Cbi(\Rd)\subset D(A_n) \quad\mbox{and}\quad A_n f=H_n(D^2 f, Df) \]
		for all $n\in\N$ and $f\in\Cbi(\Rd)$ which satisfy Assumption~\ref{ass:Sn2}
        and the conditions from Corollary~\ref{cor:unique1}.\ Hence, 
		there exists another strongly continuous convex monotone semigroup $(S(t))_{t\geq 0}$ 
		on $\Cb(\Rd)$ with generator $A\colon D(A)\to\Cb(\Rd)$ such that
		\[ S(t)f=\lim_{n\to\infty}S_n(t)f \quad\mbox{for all } (f,t)\in\Cb(\Rd)\times\R_+ \]
		satisfying $\Cbi(\Rd)\subset D(A)$ and $Af=H(D^2 f, Df)$ for all $f\in\Cbi(\Rd)$. 
	\end{theorem}
	\begin{proof}
		We construct the semigroups $(S_n(t))_{t\geq 0}$ by using Theorem~\ref{thm:cher}
        and Corollary~\ref{cor:unique2}.
		To do so, for every $n\in\N$, $t\geq 0$, $f\in\Cb(\Rd)$ and $x\in\Rd$, we define
		\[ (I_n(t)f)(x):=\sup_{(a,b)\in\S^d_+\times\Rd}\big(\E[f(x+\sqrt{a}W_t+bt)]-\varphi_n(a,b)t\big), \]
		where $\sqrt{a}\in\S^d_+$ satisfies $\sqrt{a}\sqrt{a}=a$ and $(W_t)_{t\geq 0}$ is a $d$-dimensional Brownian motion on a filtered probability space $(\Omega,\F,(\F_t)_{t\geq 0},\P)$.\
		It is straightforward to show that the operators $I_n(t)\colon\Cb(\Rd)\to\Cb(\Rd)$ 
		are well-defined and satisfy
		\begin{itemize}
			\item $I_n(t)$ is convex and monotone with $I_n(t)0=0$ for all $n\in\N$ and $t\geq 0$, 
			\item $\|I_n(t)f-I_n(t)g\|_\infty\leq\|f-g\|_\infty$ for all $n\in\N$, $t\geq 0$ and $f,g\in\Cb(\Rd)$, 
			\item $I_n(t)(\tau_x f)=\tau_x I_n(t)f$ for all $n\in\N$, $t\geq 0$, $f\in\Cb(\Rd)$ and $x\in\Rd$, 
			\item $I_n(t)\colon\Lipb(\Rd,r)\to\Lipb(\Rd,r)$ for all $n\in\N$ and $r,t\geq 0$. 
		\end{itemize}
		Let $n\in\N$ and $f\in\Cbi(\Rd)$. Due to Assumption~\ref{ass:Hn}(ii), there exists $r\geq 0$ with
		\begin{align*}
			H_n(D^2 f(x),Df(x)) &=\sup_{|a|+|b|\leq r}\Big(\frac{1}{2}\tr(aD^2 f(x))+b^T Df(x)-\varphi_n(a,b)\Big), \\
			(I_n(t)f)(x) &=\sup_{|a|+|b|\leq r}\big(\E[f(x+\sqrt{a}W_t+bt)]-\varphi_n(a,b)t\big)
		\end{align*}
		for all $t\geq 0$ and $x\in\Rd$. Hence, for every $h>0$, It\^o's formula implies
		\begin{align*}
		&\frac{I_n(h)f-f}{h}\\
		&=\sup_{|a|+|b|\leq r}\bigg(\E\bigg[\frac{1}{h}\int_0^h 
		\Big(\frac{1}{2}\tr(aD^2 f+b^T Df\Big)(\,\cdot\,+\sqrt{a}W_s+bs)\,\d s\bigg]-\varphi_n(a,b)\bigg) 
		\end{align*}
		and the right-hand side converges uniformly to $H_n(D^2 f,Df)$ as $h\to 0$. Furthermore, 
		\[ \left\|\frac{I_n(h)f-f}{h}\right\|_\infty
		\leq\sup_{(a,b)\in\S^d_+\times\Rd}\Big(\frac{1}{2}|a|\cdot\|D^2f\|_\infty+|b|\cdot\|Df\|_\infty
		-\varphi_n(a,b)\Big). \]
		Let $n\in\N$ and $(h_k)_{k\in\N}\subset (0,\infty)$ be a sequence with $h_k\to 0$.\
        Define $I_{n, k}f:=I_n(h_k)f$ for all $f\in\Cb(\Rd)$.\ Due to the previous arguments
		and Corollary~\ref{cor:cutoff2}, the sequence $(I_{n,k})_{k\in\N}$ satisfies Assumption~\ref{ass:cher}
        and the conditions from Corollary~\ref{cor:unique2}.\ Hence, by Theorem~\ref{thm:cher} 
        and Corollary~\ref{cor:unique2}, there exists a strongly continuous convex monotone semigroup 
        $(S_n(t))_{t\geq 0}$ on $\Cb(\Rd)$ with 
		\[ S_n(t)f=\lim_{k\to\infty}(I_{n,k})^{m_k^t}f \quad\mbox{for all } (f,t)\in\Cb(\Rd)\times\R_+,\]
		where $(m_k^t)_{k\in\N}\subset\N$ is an arbitrary sequence satisfying $m_k^t h_k\to t$ as $k\to\infty$.\ 
		Moreover, the following statements are valid:
		\begin{enumerate}
			\item[(i)] $\Cbi(\Rd)\subset D(A_n)$ and $A_n f=H_n(D^2 f, Df)$ for all $f\in\Cbi(\Rd)$. 
			\item[(ii)] $\|S_n(t)f-S_n(t)g\|_\infty\leq\|f-g\|_\infty$ for all $n\in\N$, $t\geq 0$ and $f,g\in\Cb(\Rd)$.
			\item[(iii)] For every $\epsilon>0$, $r,T\geq 0$ and $K\Subset\Rd$, there exist $c\geq 0$ 
			and $K'\Subset\Rd$ with
			\[ \|S_n(t)f-S_n(t)g\|_{\infty,K}\leq c\|f-g\|_{\infty,K'}+\epsilon \]
			for all $t\in [0,T]$ and $f,g\in B_{\Cb(\Rd)}(r)$. 
			\item[(iv)] $S_n(t)(\tau_x f)=\tau_x S_n(t)f$ for all $n\in\N$, $t\geq 0$, $f\in\Cb(\Rd)$ and $x\in\Rd$.
			\item[(v)] $S_n(t)\colon\Lipb(\Rd,r)\to\Lipb(\Rd,r)$ for all $n\in\N$ and $r,t\geq 0$. 
		\end{enumerate}
		In particular, Theorem~\ref{thm:unique}(c) guarantees that the semigroup $(S_n(t))_{t\geq 0}$ 
        does not depend on the choice of the sequence $(h_k)_{k\in\N}$.\
		Finally, we note that the Assumption~\ref{ass:Sn2}(i) and~(iii) and the conditions from
        Corollary~\ref{cor:unique1} are clearly satisfied.\ 
        Moreover, Corollary~\ref{cor:cher} and Corollary~\ref{cor:cutoff4} 
        guarantee that Assumption~\ref{ass:Sn2}(ii) is valid.\
		Hence, the second part of the claim follows from Theorem~\ref{thm:Sn2} and Corollary~\ref{cor:unique1}.
	\end{proof}

\subsection{Large deviations for randomized Euler schemes}
	\label{sec:LargeDev}
    We begin by studying randomized Euler schemes for ODEs with Lipschitz coefficients.\
    Let $\psi\colon\Rd\to\Rd$ be a bounded Lipschitz continuous function.\
	The aim is to show that the unique solution of the parameter
	dependent ODE 
	\begin{equation} \label{eq:ODE}
		\begin{cases}
			\partial_t u(t,x)=\psi(u(t,x)),  & t\geq 0, \\
			u(0,x)=x,  & x\in\Rd, 
		\end{cases}
	\end{equation}
	can be approximated by a randomized Euler scheme under model uncertainty.\
	For that purpose, let $(h_n)_{n\in\N}\subset (0,\infty)$ be a sequence with $h_n\to 0$ 
	and let $(\xi_n)_{n\in\N}$ be a sequence of i.i.d.~random variables on a sublinear 
	expectation space $(\Omega,\H, \e)$ with $\e[|\xi_1]^3]<\infty$ and 
	$\e[a^T\xi_1]=0$ for all $a\in\Rd$.\ In case that $\e[\,\cdot\,]=\E[\,\cdot\,]$ is a linear 
	expectation, there exists a unique probability measure $\mu$ on $\B(\Rd)$ with
	\[ \E[f(\xi_n)]=\int_{\Rd}f(z)\,\mu(\d z) \quad\mbox{for all } n\in\N \mbox{ and } f\in\Cb(\Rd). \]
	Furthermore, if the distribution $\mu$ of the random variables $(\xi_n)_{n\in\N}$ is uncertain, 
	a worst case approach consists in taking the supremum over a set of measures, i.e.,
	\[ \e[f(\xi_n)]:=\sup_{\mu\in M}\int_{\Rd}f(z)\,\mu(\d z). \]
	For more details, we refer to~\cite{Peng19} and~\cite[Appendix~B]{BK22}.\ For every 
	$n\in\N$, $\delta>0$, $t\geq 0$ and $x\in\Rd$, we define a randomized 
	Euler scheme under model uncertainty by
	\[ X^{n,\delta,x}_0:=x \quad\mbox{and}\quad
	X^{n,\delta,x}_{(k+1) h_n} 
	:=X^{n,\delta,x}_{kh_n}+h_n\psi\left(X^{n,\delta,x}_{k h_n}\right)+\delta \sqrt{h_n}\xi_{k+1} \]
	and show that $X^{n,\delta_n,x}_{k_n^t h_n}\to u(t,x)$ for all $(t,x)\in\R_+\times\Rd$,
	where $(\delta_n)_{n\in\N}\subset (0,\infty)$ and $(k_n^t)_{n\in\N}\subset\N$ are 
	sequences with $\delta_n\to 0$ and $k_n^t h_n\to t$. To be precise, we show that
	\[ \e\left[f\left(X^{n,\delta_n,x}_{k_n^t h_n}\right)\right]\to f(u(t,x)) 
	\quad\mbox{for all } f\in\Cb(\Rd). \]
	This means that, regardless of possible numerical errors, the Euler scheme converges 
	to the solution of~\eqref{eq:ODE}.\ Moreover, if $\delta_n\to\delta>0$, 
	the Euler scheme converges to the solution of a stochastic differential equation 
	driven by a G-Brownian motion, cf.~\cite{Peng19}. Define
	\[ (I_{n,\delta}f)(x):=\e\left[f\left(x+h_n\psi(x)+\delta \sqrt{h_n}\xi_1\right)\right] \]
	for all $n\in\N$, $\delta>0$, $f\in\Cb(\Rd)$ and $x\in\Rd$. Let $L$ be the Lipschitz constant
	of $\psi$.

	\begin{theorem} \label{thm:randomEuler} 
		Let $(\delta_n)_{n\in\N}\subset(0,\infty)$ with $\delta_n\to\delta\in\R_+$.\
		Then, there exists a strongly continuous convex monotone semigroup $(S_\delta(t))_{t\geq 0}$ 
		on $\Cb(\Rd)$ with generator $A_\delta\colon D(A_\delta)\to\Cb(\Rd)$ such that
		\[ S_\delta(t)f=\lim_{n\to\infty}(I_{n,\delta_n})^{k_n^t}f 
		=\lim_{n\to\infty}\e\left[f\left(X^{n,\delta_n,\,\cdot}_{k_n^t h_n}\right)\right]
		\quad\mbox{for all } (f,t)\in\Cb(\Rd)\times\R_+, \]
		where $(k_n^t)_{n\in\N}\subset\N$ satisfies $k_n^t h_n\to t$.\
		Moreover, the following statements are valid:
		\begin{enumerate}
			\item[(i)] It holds $\Cbi(\Rd)\subset D(A_\delta)$ and 
			\[ (A_\delta f)(x)=\frac{1}{2}\delta^2\e\big[\xi_1^T D^2 f(x)\xi_1\big]+Df(x)^T\psi(x)
			\quad\mbox{for all } f\in\Cbi(\Rd). \]
			\item[(ii)] It holds $\|S_\delta(t)f-S_\delta(t)g\|_\infty\leq\|f-g\|_\infty$ for all $t\geq 0$ 
			and $ f,g\in\Cb(\Rd)$.
			\item[(iii)] For every $\epsilon>0$, $r,T\geq 0$ and $K\Subset\Rd$, there exist $c\geq 0$ 
			and $K'\Subset\Rd$ with
			\[ \|S_\delta(t)f-S_\delta(t)g\|_{\infty,K}\leq c\|f-g\|_{\infty,K'}+\epsilon \]
			for all $t\in [0,T]$ and $f,g\in B_{\Cb(\Rd)}(r)$. 
			\item[(iv)] For every $r, t\geq 0$, $f\in\Lipb(\Rd,r)$ and $x\in\Rd$,
			\[ \|S_\delta(t)(\tau_x f)-\tau_x S_\delta(t)f\|_\infty\leq Le^{Lt}rt|x|. \]
			Furthermore, it holds $S_\delta(t)\colon\Lipb(\Rd,r)\to\Lipb(\Rd,e^{Lt}r)$ for all $r,t\geq 0$. 
		\end{enumerate}
		In particular, it holds $(S_\delta(t)f)(x)=\e[f(X_t^{\delta,x})]$ for all $t\geq 0$, $f\in\Cb(\Rd)$ 
		and $x\in\Rd$, where $X_t^{\delta,x}$ denotes the unique solution of the parameter
		dependent ordinary (stochastic) differential equation
		\begin{equation} \label{eq:SDE}
			\begin{cases}
				\d X^{\delta,x}_t=\psi(X^{\delta,x}_t)\d t+\delta\d W_t, & t\geq 0, \\
				X^{\delta,x}_0=x, & x\in\Rd, 
			\end{cases}
		\end{equation}
		which is driven by a G-Brownian motion $(W_t)_{t\geq 0}$ with 
		$\e[W_1W_1^T]=\e[\xi_1\xi_1^T]$. 
	\end{theorem}
	\begin{proof}
		We verify Assumption~\ref{ass:cher} and the conditions from Corollary~\ref{cor:unique2}.\
		It is straightforward to show that the operators $I_{n,\delta_n}\colon\Cb(\Rd)\to\Cb(\Rd)$ 
		are well-defined and that the following statements are valid:
		\begin{itemize}
			\item $I_{n,\delta_n}$ is convex and monotone with $I_{n,\delta_n}0=0$ for all $n\in\N$, 
			\item $\|I_{n,\delta_n}f-I_{n,\delta_n}g\|_\infty\leq\|f-g\|_\infty$ for all $n\in\N$ and $f,g\in\Cb(\Rd)$, 
			\item $\|\tau_x I_{n,\delta_n}f-I_{n,\delta_n}(\tau_x f)\|_\infty\leq Lrh_n|x|$ for all $n\in\N$, 
			$f\in\Lipb(\Rd,r)$ and $x\in\Rd$, 
			\item $I_{n,\delta_n}\colon\Lipb(\Rd,r)\to\Lipb(\Rd,(1+Lh_n)r)$ for all $n\in\N$ and $r\geq 0$. 
		\end{itemize}
		Moreover, for every $n\in\N$, $f\in\Cbi(\Rd)$ and $x\in\Rd$, Taylor's formula yields
		\begin{align*}
			&f(x+h_n\psi(x)+\delta_n\sqrt{h_n}\xi_1) =f(x)+Df(x)^T(h_n\psi(x)+\delta_n\sqrt{h_n}\xi_1) \\
			&\qquad\qquad\; +\frac{1}{2}(h_n\psi(x)+\delta_n\sqrt{h_n}\xi_1)^T D^2f(x)(h_n\psi(x)+\sqrt{h_n}\delta_n\xi_1)
			+R_n(f, x,\xi_1),
		\end{align*}
		where the remainder term can be estimated by
		\[ |R_n(f,x,\xi_1)|\leq\frac{|h_n\psi(x)+\delta_n\sqrt{h_n}\xi_1|^3}{6}\|D^3 f\|_\infty. \]
		Since $\e[a^T\xi_1]=0$ for all $a\in\Rd$, we obtain 
		\begin{align*}
			\frac{(I_{n,\delta_n}f-f)(x)}{h_n}
			&=\e\left[\frac{f(x+h_n\psi(x)+\delta_n\sqrt{h_n}\xi_1)-f(x)}{h_n}\right] \\
			&=Df(x)^T\psi(x)+\frac{\delta_n^2}{2}\e\left[\xi^T D^2f(x)\xi\right]+(\tilde{R}_n(f,x,\xi_1),
		\end{align*}
		where the remainder term can be estimated by
		\[ |\tilde{R}_n(f,x,\xi_1)|\leq\frac{1}{2}h_n|\psi(x)^T D^2 f(x)\psi(x)|+\frac{\e[|R_n(f,x,\xi_1)|]}{h_n}. \]
		Since $\e[|\xi_1|^3]<\infty$ and $\|\psi\|_\infty<\infty$, we obtain
		\[ \frac{I_{n,\delta_n}f-f}{h_n}\to Df(\cdot)^T\psi+\frac{\delta^2}{2}\e\left[\xi^T D^2 f(\cdot)\xi\right]
		\quad\mbox{for all } f\in\Cbi(\Rd). \]
		The previous estimates and Corollary~\ref{cor:cutoff2} guarantee that Assumption \ref{ass:cher}(iii) 
        is satisfied.\ Furthermore, since the random variables $(\xi_n)_{n\in\N}$ are i.i.d., it follows from the independence argument outlined in~\cite[Appendix B]{BK22} that
		\[ (I_{n,\delta_n})^k f_n=\e\left[f\left(X^{n,\delta_n,x}_{kh_n}\right)\right] \]
		for all $k,n\in\N$, $f\in\Cb(\Rd)$ and $x\in\Rd$.\ Finally, it follows from the strong Markov
		property for stochastic differential equations driven by G-Brownian motions in~\cite{HJL21}
		that $(T_\delta(t)f)(x):=\e[f(X_t^{\delta,x})]$ defines a strongly continuous
		monotone semigroup on $\Cb(\Rd)$ satisfying the conditions~(a)--(c) from Theorem~\ref{thm:unique}.\
		Now, the claim follows from Theorem~\ref{thm:cher}, Theorem~\ref{thm:unique} and Corollary~\ref{cor:unique2}. 
	\end{proof}

    Subsequently, we focus on the convergence rate of the scheme by means of a large deviations approach, 
	see~\cite{Varadhan1984, DE, DZ2010, FK}.\ In doing so, we derive a large deviations result for stochastic
	differential equations in the spirit of Freidlin--Wentzell~\cite{FW}.\ 
    While the choice of the parameters $\delta_n$ and $h_n$ is still arbitrary,
    the sequence $(\alpha_n)_{n\in\N}\subset (0,\infty)$ will depend on $\delta_n$ and $h_n$.
For every $n\in\N$ and $f\in\Cb(\Rd)$, we define
	\[ J_n f:=\frac{1}{\alpha_n}\log\big(I_{n,\delta_n} e^{\alpha_n f}\big), \]
	 where $I_{n,\delta}$ is defined as before. Recall also the randomized Euler scheme $(X^{n,\delta,x}_{k h_n})_{k\in\mathbb{N}}$.

	\begin{theorem} \label{thm:LargeDev}
		Assume that there exists $\delta,\gamma\geq 0$ with $\delta_n\to\delta$ and
		$\alpha_n\delta_n^2\to\gamma$. Let $h_n\to 0$ and $\alpha_n h_n\to 0$. Assume in addition that $\xi_1$ is bounded.
		Then, there exists a strongly continuous convex monotone semigroup 
		$(T_{\gamma,\delta}(t))_{t\geq 0}$ on $\Cb(\Rd)$
        with generator $B_{\gamma,\delta}\colon D(B_{\gamma,\delta})\to\Cb(\Rd)$ such that
		\[ T_{\gamma,\delta}(t)f
		=\lim_{n\to\infty}J_n^{k_n^t}f
		=\lim_{n\to\infty}\frac{1}{\alpha_n}\log\left(\e\left[\exp\left(
		\alpha_n f\left(X^{n,\delta_n,\,\cdot}_{k_n^t h_n}\right)\right)\right]\right) \]
		for all $t\geq 0$, $f\in\Cb(\Rd)$ and sequences $(k_n^t)_{n\in\N}\subset\N$
		with $k_n^t h_n\to t$.\ Moreover, the following statements are valid:
		\begin{enumerate}
			\item[(i)] It holds $\Cbi(\Rd)\subset D(B_{\gamma,\delta})$ and 
			\[ (B_{\gamma,\delta}f)(x)
			=\frac{1}{2}\e\big[\delta^2\xi_1^T D^2 f(x)\xi_1+\gamma |Df(x)^T\xi_1|^2\big]+Df(x)^T\psi(x) \]
			for all $f\in\Cbi(\Rd)$ and $x\in\Rd$. 
			\item[(ii)] It holds $\|T_{\gamma,\delta}(t)f-T_{\gamma,\delta}(t)g\|_\infty\leq\|f-g\|_\infty$ 
			for all $t\geq 0$ and $ f,g\in\Cb(\Rd)$.
			\item[(iii)] For every $\epsilon>0$, $r,T\geq 0$ and $K\Subset\Rd$, there exist $c\geq 0$ 
			and $K'\Subset\Rd$ with
			\[ \|T_{\gamma,\delta}(t)f-T_{\gamma,\delta}(t)g\|_{\infty,K}\leq c\|f-g\|_{\infty,K'}+\epsilon \]
			for all $t\in [0,T]$ and $f,g\in B_{\Cb(\Rd)}(r)$. 
			\item[(iv)] For every $r, t\geq 0$, $f\in\Lipb(\R^d,r)$ and $x\in\Rd$,
			\[ \|T_{\gamma,\delta}(t)(\tau_x f)-\tau_x T_{\gamma,\delta}(t)f\|_\infty\leq Le^{Lt}rt|x|. \]
			Furthermore, it holds $T_{\gamma,\delta}(t)\colon\Lipb(r)\to\Lipb(e^{Lt}r)$ for all $r,t\geq 0$. 
		\end{enumerate}
	\end{theorem}
	\begin{proof}
		By relying on the estimates in the proof of Theorem~\ref{thm:randomEuler}, 
		it is straightforward to show that the operators $J_n\colon\Cb(\Rd)\to\Cb(\Rd)$ 
		are well-defined and satisfy
		\begin{itemize}
			\item $J_n$ is convex and monotone with $J_n0=0$ for all $n\in\N$, 
			\item $\|J_n f-J_n g\|_\infty\leq\|f-g\|_\infty$ for all $n\in\N$ and $f,g\in\Cb(\Rd)$, 
			\item $\|\tau_x J_n f-J_n(\tau_x f)\|_\infty\leq Lrh_n|x|$ for all $n\in\N$, $f\in\Lipb(\Rd,r)$ 
			and $x\in\Rd$, 
			\item $J_n\colon\Lipb(\Rd,r)\to\Lipb(\Rd,(1+Lh_n)r)$ for all $n\in\N$ and $r\geq 0$. 
		\end{itemize}
		Indeed, one can show that $J_n(f+c)=J_n f+c$ and $J_n(\lambda f+(1-\lambda)g)\leq\max\{J_n f, J_n g\}$ 
		for all $n\in\N$, $f,g\in\Cb(\Rd)$, $c\in\R$ and $\lambda\in [0,1]$. Hence, 
		\begin{align*}
			&\big(J_n(\lambda f+(1-\lambda)g)\big)(x)-\lambda (J_n f)(x)-(1-\lambda)(J_n g)(x) \\
			&=\Big(J_n\big(\lambda(f-(J_n f)(x))+(1-\lambda)(g-(J_n g)(x))\big)\Big)(x) \\
			&\leq\max\big\{\big(J_n(f-(J_n f)(x))\big)(x), \big(J_n(g-(J_n g)(x))\big)(x)\big\}=0
		\end{align*}
		for all $n\in\N$, $f,g\in\Cb(\Rd)$ and $\lambda\in [0,1]$ which shows that $J_n$
		is convex.\ Moreover, it holds $J_n f\leq J_n\big(g+\|f-g\|_\infty\big)=J_n g+\|f-g\|_\infty$
		and reversing the roles of $f$ and~$g$ yields $\|J_n f-J_n g\|_\infty\leq\|f-g\|_\infty$. 
		The other properties can be derived similarly. 
		
		For every $n\in\N$, $f\in\Cbi(\Rd)$ and $x\in\Rd$, Taylor's formula implies
		\begin{align*}
			&\exp\left(\alpha_n\big(f\big(x+h_n\psi(x)+\delta_n\sqrt{h_n}\xi\big)-f(x)\big)\right) \\
			&=1+\alpha_n Df(x)^T (h_n\psi(x)+\delta_n\sqrt{h_n}\xi_1)
			+\frac{\alpha_n^2}{2}|Df(x)^T(h_n\psi(x)+\delta_n\sqrt{h_n}\xi_1)|^2 \\
			&\quad\; +\frac{\alpha_n}{2}(h_n\psi(x)+\delta_n\sqrt{h_n}\xi_1)^T D^2f(x)(h_n\psi(x)+\delta_n\sqrt{h_n}\xi_1)
			+R_n(f,x,\xi_1),
		\end{align*} 
		where the remainder term can be estimated by
		\[ |R_n(f,x,\xi_1)|\leq\frac{\alpha_n|h_n\psi(x)+\delta\sqrt{h_n}\xi_1|^3}{6}
		\rho\bigg(\sum_{i=1}^3\|D^i f\|_\infty\bigg) \]
		for a suitable function $\rho\colon\R_+\to\R_+$ which satisfies $\rho(\epsilon)\to 0$ 
		as $\epsilon\to 0$.\ We use the fact that $\e[a^T\xi_1]=0$ for all $a\in\Rd$ to conclude
		\begin{align*}
			&\e\left[\exp\left(\alpha_n\big(f\big(x+h_n\psi(x)+\delta_n\sqrt{h_n}\xi_1\big)-f(x)\big)\right)\right] \\
			&=1+\alpha_n h_n Df(x)^T\psi(x)+\frac{\alpha_n^2 h_n^2}{2}|Df(x)^T\psi(x)|^2
			+\frac{\alpha_n h_n^2}{2}\psi(x)^T D^2f(x)\psi(x) \\
			&\quad\; +\frac{\alpha_n h_n}{2}\e\big[\alpha_n\delta_n^2|Df(x)^T\xi_1|^2
			+\delta_n^2\xi_1^T D^2 f(x)\xi_1\big]+\tilde{R}_n(f,x,\xi_1),
		\end{align*}
		where the remainder term can be estimated by $|\tilde{R_n}(f,x,\xi_1)|\leq\e[|R_n(f,x,\xi_1)|]$.\
		Moreover, by assumption, it holds $h_n\to 0$ and $\alpha_n h_n\to 0$.
		Hence, Taylor's formula implies
		\begin{align*}
			\frac{J_n f-f}{h_n}
			&=\frac{1}{\alpha_n h_n}\log\left(
			\e\left[\exp\left(\alpha_n\big(f\big(x+h_n\psi(x)+\delta_n\sqrt{h_n}\xi_1\big)-f(x)\big)\right)\right]\right) \\
			&=Df(x)^T\psi(x)+\frac{\alpha_n h_n}{2}|Df(x)^T\psi(x)|^2+\frac{h_n}{2}\psi(x)^T D^2 f(x)\psi(x) \\
			&\quad\; +\frac{1}{2}\e\big[\alpha_n\delta_n^2 |Df(x)^T\xi_1|^2+\delta_n^2\xi_1^T D^2f(x)\xi_1\big]
			+\hat{R}_n(f,x,\xi_1),
		\end{align*}
		where the remainder term can be estimated by $ |\hat{R}_n(f,x,\xi_1)|\leq\hat{\rho}\big(\sum_{i=1}^3\|D^i f\|_\infty\big)r_n $
		for a suitable function $\hat{\rho}\colon\R_+\to\R_+$ with $\hat{\rho}(\epsilon)\to 0$ 
		as $\epsilon\to 0$ and a sequence $(r_n)_{n\in\N}\subset\R_+$ with $r_n\to 0$. 
		Since $\delta_n\to\delta$, $\alpha_n\delta_n^2\to\gamma$, $h_n\to 0$ and $\alpha_n h_n\to 0$,
		we obtain
		\[ \lim_{n\to\infty}\left\|\frac{J_n f-f}{h_n}
		-\frac{1}{2}\e\big[\delta^2\xi_1^T D^2 f(\cdot)\xi_1+\gamma |Df(\cdot)^T\xi_1|^2\big]
		-Df(\cdot)^T\psi(\cdot)\right\|_\infty=0 \]
		for all $f\in\Cbi(\Rd)$.\ Moreover, the previous estimates and Corollary~\ref{cor:cutoff2}
		imply that Assumption~\ref{ass:cher}(iii) is valid. Since the random variables 
		$(\xi_n)_{n\in\N}$ are i.i.d., we obtain
		\[ (J_n^k f)(x)=\frac{1}{\alpha_n}\log\left(\e\left[\exp\left(
		\alpha_n f\left(X^{n,\delta_n,x}_{kh_n}\right)\right)\right]\right) \]
		for all $k,n\in\N$, $f\in\Cb(\Rd)$ and $x\in\Rd$.\ Now, the claim follows from
		Theorem~\ref{thm:cher} and Corollary~\ref{cor:unique2}.
	\end{proof}

	If the limit is a Hamilton--Jacobi semigroup, we obtain a Laplace principle.

	\begin{corollary} \label{cor:LargeDev}
		Let $(\delta_n)_{n\in\N}\subset (0,\infty)$ be a sequence with $\delta_n\to 0$
		and $\delta_n^2 n\to\infty$. Then, for every $f\in\Cb(\Rd)$,
		\[ \delta_n^2\log\e\left[\exp\left(\frac{1}{\delta_n^2}
		f\left(\frac{\delta_n}{\sqrt{n}}\sum_{i=1}^n \xi_i\right)\right)\right]
		\to\sup_{y\in\Rd}\big(f(y)-\varphi(y)\big) \]
	  with rate function $\varphi\colon\Rd\to [0,\infty]$  given by
		$\varphi(y):=\sup_{z\in\Rd}\big(y^Tz-\tfrac{1}{2}\e\big[|z^T\xi_1|^2\big]\big)$.
	\end{corollary}
	\begin{proof}
		Applying Theorem~\ref{thm:LargeDev} with $h_n:=\frac{1}{n}$, $\alpha_n:=\nicefrac{1}{\delta_n^2}$,
		and $\psi\equiv 0$ yields 
		\begin{align*}
			T_{1,0}(1)f &=\lim_{n\to\infty}\delta_n^2\log\e\bigg[\exp\bigg(\frac{1}{\delta_n^2}
			f\Big(\,\cdot\,+\frac{\delta_n}{\sqrt{n}}\sum_{i=1}^n \xi_i\Big)\bigg)\bigg] 
			\quad\mbox{for all } f\in\Cb(\Rd), \\
			B_{1,0}f &=\frac{1}{2}\e[|Df(x)^T\xi_1|^2] \quad\mbox{for all } f\in\Cbi(\Rd). 
		\end{align*}
		Similar to the proof of~\cite[Theorem~3.4]{BK22}, one can verify the Hopf--Lax formula
		\[ (T_{1,0}(t)f)(x)=\sup_{y\in\Rd}\big(f(x+ty)-\varphi(y)t\big) 
		  \quad\mbox{for all } t\geq 0,\,  f\in\Cb(\Rd) \mbox{ and }x\in \Rd. \] 
	\end{proof}

	We also obtain a logarithmic small-noise limit in the spirit of the Freidlin--Wentzell large 
	deviations principle~\cite{FW} characterizing the convergence rate of the 
	solution $X^{\delta,x}_t$ of the SDE~\eqref{eq:SDE} driven by a standard Brownian
	motion to the solution $X^{0,x}_t=u(t,x)$ of the ODE~\eqref{eq:ODE}.\ In the sublinear 
	case, where the solution $X^{\delta,x}_t$ of the SDE~\eqref{eq:SDE} is driven by a 
	G-Brownian motion, we refer to~\cite{CX,GJ} for large deviations principles.

	\begin{corollary}
		Let $(\delta_n)_{n\in\N}\subset (0,\infty)$ be a sequence with $\delta_n\to 0$.
		Then, for every $t \ge 0$, $x\in \Rd$ and $f\in\Cb(\Rd)$,
		\[ \delta^2_n\log\bigg(\e\bigg[\exp\Big(\frac{1}{\delta_n^2}f\big(X^{\delta_n,x}_t\big)\Big)\bigg]\bigg)
		\to (T_{1,0}(t)f)(x), \]
		where $(T_{1,0}(t))_{t\ge 0}$ is the semigroup from Theorem~\ref{thm:LargeDev} for $\gamma=1$ and $\delta=0$. 
	\end{corollary}
	\begin{proof}
		For every $n\in\N$, let $(S_{\delta_n}(t))_{t\geq 0}$ be the semigroup 
		from Theorem~\ref{thm:randomEuler}.\ Fix $n\in\N$ and choose 
		$\alpha_k:=\frac{1}{\delta_n^2}$ for all $k\in\N$ and 
		$(h_k)_{k\in\N}\subset (0,\infty)$ with $h_k\to0$. Define
		\[ (I_{k,\delta_n}f)(x):=\e\left[f\left(x+h_k\psi(x)+\delta_n\sqrt{h_k}\xi_1\right)\right]
		\quad\mbox{and}\quad
		J_{k,\delta_n} f:=\frac{1}{\alpha_k}\log\big(I_{k,\delta_n}e^{\alpha_k f}\big) \]
		for all $k,n\in\N$, $f\in\Cb(\Rd)$ and $x\in\Rd$. Theorem~\ref{thm:randomEuler}
		and Theorem~\ref{thm:LargeDev} imply
		\begin{align*}
			&\delta_n^2\log\left(\e\left[\exp\left(\frac{1}{\delta_n^2}f\left(X^{\delta_n,\,\cdot}_t\right)\right)\right]\right)
			=\delta_n^2\log\left(S_{\delta_n}(t)e^{\nicefrac{f}{\delta_n^2}}\right) \\
			&=\lim_{k\to\infty}\delta_n^2\log\left((I_{k,\delta_n})^{m_k^t}e^{\nicefrac{f}{\delta_n^2}}\right) =\lim_{k\to\infty}(J_{k,\delta_n})^{m_k^t}f=T_{1,\delta_n}(t)f
		\end{align*}
		for all $t\geq 0$, $f\in\Cb(\Rd)$ and $(m_k^t)_{k\in\N}\subset\N$ with
		$m_k^t h_k\to t$.\ Moreover, Theorem~\ref{thm:LargeDev} and Corollary~\ref{cor:cutoff4} 
		yield that Assumption~\ref{ass:Sn2} and the conditions from Corollary~\ref{cor:unique2}
        are valid.\
        Indeed, it follows from the proof of Theorem~\ref{thm:LargeDev} and~\cite[Theorem~4.3]{BK22+} that
		\[ \lim_{h\downarrow 0}\left\|\frac{T_{1,0}(h)f-f}{h}-B_{1,0}f\right\|_\infty=0 
		\quad\mbox{for all } f\in\Cbi(\Rd). \]
		Hence, the claim follows from Theorem~\ref{thm:Sn2}, Corollary~\ref{cor:unique1} and Corollary~\ref{cor:LargeDev}.
	\end{proof}

	\subsection{Discretization of stochastic optimal control problems}
	\label{sec:control}
	
	{We fix a Borel measurable function $\varphi\colon\S^d_+\times \Rd\times \R^{d\times d}\to [0,\infty]$
	and consider the value function of a dynamic stochastic control problem with finite
	time horizon given by
	\begin{align*} 
		(T(t)f)(x):=\sup_{(a,b,c)\in\A} \bigg(\E\bigg[f\bigg(x+\int_0^t\sqrt{a_s}\,\d W_s+\int_0^t b_s\,\d s&+{\int_0^t c_s\, \d J_s}\bigg)\bigg]\\
        &
		-\E\bigg[\int_0^t \varphi(a_s,b_s,c_s)\,\d s\bigg]\bigg),
	\end{align*}
	where $(W_t)_{t\geq 0}$ is a $d$-dimensional Brownian motion on a complete 
	filtered probability space $(\Omega,\F,(\F_t)_{t\geq 0},\P)$ satisfying the usual conditions, $(J_t)_{t\geq0}$ is an adapted compound Poisson process with jump measure $\mu$ and intensity $\lambda>0$
	and $\A$ consists of all predictable processes $(a,b,c)\colon\Omega\times\R_+\to\S^d_+\times\Rd\times \R^{d\times d}$ 
	with 
    \[ \E\left[\int_0^t |a_s|+|b_s|+|c_s|\,\d s\right]<\infty \quad\text{for all } t\geq 0. \] 
	The aim of this subsection is to approximate $T(t)f$ by iterating a sequence 
	of discrete static control problems given by
	\[ (I_n f)(x):=\sup_{(a,b,c)\in\A_n}
	\bigg(\E\bigg[f\bigg(x+\sqrt{ah_n}\xi_n+bh_n+c\sum_{k=1}^{N_n} Z_n^k\bigg)\bigg]-\varphi_n(a,b,c)h_n\bigg), \]
	where $\A_n\subset\S^d_+\times\Rd\times \R^{d\times d}$, $x\in X_n\subset\Rd$ and 
	$\varphi_n\colon\A_n\to [0,\infty]$. We will show that
	\[ T(t)f=\lim_{n\to\infty}I_n^{k_n^t}f \quad\mbox{for all } (f,t)\in\Cb(\Rd)\times\R_+. \]
	
	Subsequently, we formalize the definition of the operators $(I_n)_{n\in\N}$ and 
	impose sufficient conditions to guarantee the convergence.\ Let $(h_n)_{n\in\N}\subset (0,\infty)$ 
	be a sequence with $h_n\to 0$ and define $\T_n:=\{kh_n\colon k\in\N_0\}$ for all $n\in\N$.\
	Furthermore, let $(X_n)_{n\in\N}$ be a sequence of closed sets $X_n\subset\Rd$ with
	$x+y\in X_n$ for all $x,y\in X_n$ such that, for every $x\in\Rd$, there exists
	$x_n\in X_n$ with $x_n\to x$.\ For every $n\in \N$, we choose independent random vectors $\xi_n\colon\Omega\to\Rd$ 
	with $\E[\xi_n]=0$, $\E[\xi_n\xi_n^T]=I_d\in\R^{d\times d}$ and $\sup_{n\in\N}\E[|\xi_n|^3]<\infty$, Poisson distributed $N_n$ with intensity $\lambda_n h_n$ and an i.i.d.\ sequence $(Z_n^k)_{k\in \N}$ with $Z_n^k\sim \mu_n$ for all $k\in \N$.\ Moreover, let $\A_n\subset\S^d_+\times\Rd\times \R^{d\times d}$ and $x+\sqrt{ah_n}\xi_n(\omega)+bh_n+ c\sum_{k=1}^{N_n(\omega)} Z_n^k(\omega)\in X_n$ for all $n\in\N$, $x\in X_n$, $(a,b,c)\in\A_n$ 
    and $\omega\in\Omega$.\ Finally, let $(\varphi_n)_{n\in\N}$ be a sequence of functions 
    $\varphi_n\colon\A_n\to [0,\infty]$ and 
    \[
    \big(A^{a,b,c}f\big)(x):=\frac{1}{2}\tr\big(aD^2f(x)\big)+b^TD f(x)+\lambda\int_{\R^d} f(x+cz)-f(x)\, \mu(\d z)
    \]
    for all $f\in \Cbi(\R^d)$ and $x\in \R^d$.
	
	\begin{assumption} \label{ass:control}
		Suppose that the following conditions are satisfied:
		\begin{enumerate}
			\item[(i)] There exists $(a^*,b^*,c^*)\in\S^d_+\times\Rd\times \R^{d\times d}$ with $\varphi(a^*,b^*,c^*)=0$. 
			\item[(ii)] It holds $\lim_{|a|+|b|+|c|\to\infty}\frac{\varphi(a,b,c)}{|a|+|b|}=\infty$.
			\item[(iii)] For every $n\in\N$, there exists $(a_n^*,b_n^*,c_n^*)\in\A_n$ with $\varphi_n(a_n^*,b_n^*,c_n^*)=0$.\
			Furthermore, it holds $\sup_{n\in\N}(|a_n^*|+|b_n^*|+|c_n^*|)<\infty$. 
			\item[(iv)] It holds $\lim_{|a|+|b|+|c|\to\infty}\inf_{n\in\N}\frac{\varphi_n(a,b,c)}{|a|+|b|}=\infty$. 
			\item[(v)] For every $f\in \Cbi(\R^d)$, it holds uniformly w.r.t. $x\in \R^d$ that
            \begin{align*}
             \lim_{n\to \infty}\sup_{(a,b,c)\in\A_n}&
        \Big(\big(A^{a,b,c}f\big)(x)-\varphi_n(a,b,c)\Big)\\
        &\qquad=\sup_{(a,b,c)\in\S^d_+\times\Rd\times \R^{d\times d}}
        \Big(\big(A^{a,b,c}f\big)(x)-\varphi(a,b,c)\Big).            \end{align*}
            
            \item[(vi)] It holds 
            $0\leq \lambda_n\to \lambda\in (0,\infty)$ and $\mu_n\to\mu$ in distribution.
		\end{enumerate}
	\end{assumption}

	\begin{theorem} \label{thm:control}
		Suppose that Assumption~\ref{ass:control} is satisfied.\ Then, there exists a strongly continuous 
		convex monotone semigroup $(S(t))_{t\geq 0}$ on $\Cb(\Rd)$ with generator $A\colon D(A)\to\Cb(\Rd)$ such that
		\[ S(t)f=\lim_{n\to\infty}I_n^{k_n^t}f \quad\mbox{for all } (f,t)\in\Cb(\Rd)\times\R_+,\]
		where $(k_n^t)_{n\in\N}\subset\N$ is arbitrary with $k_n^t h_n\to t$.
		It holds $\Cbi(\Rd)\subset D(A) $ and 
		\[ Af=\sup_{(a,b,c)\in\S^d_+\times\Rd\times \R^{d\times d}}\big(A^{a,b,c}f-\varphi(a,b,c)\big) \quad \text{for all }f\in \Cbi(\Rd).\]
        Furthermore, it holds
		$S(t)f=T(t)f$ for all $t\geq 0$ and $f\in\Cb(\Rd)$. 
	\end{theorem}
	\begin{proof}
		We verify Assumption~\ref{ass:cher} and the conditions from Corollary~\ref{cor:unique2}.\
		It is straightforward to show that the operators $I_n\colon\Cb(X_n)\to\Cb(X_n)$ 
		are well-defined and 
		\begin{itemize}
			\item $I_n$ is convex and monotone with $I_n0=0$ for all $n\in\N$, 
			\item $\|I_n f-I_n g\|_{\infty, X_n}\leq\|f-g\|_{\infty,X_n}$ for all $n\in\N$ and $f,g\in\Cb(X_n)$, 
			\item $\tau_x I_n f=I_n(\tau_x f)$ for all $n\in\N$, $f\in\Cb(X_n)$ and $x\in X_n$,  
			\item $I_n\colon\Lipb(X_n,r)\to\Lipb(X_n,r)$ for all $n\in\N$ and $r\geq 0$. 
		\end{itemize}
		Let $f\in\Cbi(\Rd)$ and $J_n:=\sum_{k=1}^{N_n}Z_n^k$. For every $n\in\N$ and $x\in X_n$, 
		\begin{align*}
			&f(x+cJ_n+\sqrt{ah_n}\xi_n+bh_n)-f(x+cJ_n) \\
            &=\int_0^1 Df(x+cJ_n+ t(\sqrt{ah_n}\xi_n+bh_n))^T(\sqrt{ah_n}\xi_n+bh_n)\,\d t \\
			&=\int_0^1 Df(x+cJ_n+t(\sqrt{ah_n}\xi_n+bh_n))bh_n\,\d t\\
            &\quad\; +\int_0^1 Df(x+cJ_n+tbh_n))^T\sqrt{ah_n}\xi_n\,\d t \\
			&\quad\; +\int_0^1\int_0^1 \sqrt{ah_n}\xi_n^T D^2 f(x+cJ_n+st\sqrt{ah_n}\xi_n+tbh_n)\sqrt{ah_n}\xi_n t\,\d s\,\d t.
		\end{align*}
        We use $\E[\xi_n]=0$ and $\E[|\xi_n|^2]=d$ 
        to obtain
		\[ \frac{(I_n^{a,b,c}f-f)(x)}{h_n}-\varphi_n(a,b,c)
		\leq\frac{1}{2}d|a|\|D^2 f\|_\infty+|b|\|Df\|_\infty+2\lambda_n\|f\|_\infty-\varphi_n(a,b,c) \]
		for all $n\in\N$, $(a,b,c)\in\A_n$ and $x\in X_n$, where $$(I_n^{a,b,c} f)(x):=\E\big[f\big(x+cJ_n+\sqrt{ah_n}\xi_n+bh_n\big)\big].$$
		Hence, due to Assumption~\ref{ass:control}(iv), there exists $r\geq 0$ with
		\begin{equation} \label{eq:control.bound1}
			I_nf=\sup_{(a,b,c)\in\A_n^r}\big(I_n^{a,b,c}f-\varphi_n(a,b,c)h_n\big) \quad\mbox{for all } n\in\N,
		\end{equation}
		where $\A_n^r:=\{(a,b,c)\in\A_n\colon |a|+|b|+|c|\leq r\}$. By increasing $r\geq 0$,
		we can assume 
		\begin{equation} \label{eq:control.bound2.1}
			\sup_{(a,b,c)\in\S_+^d\times \Rd\times \R^{d\times d}}\Big(\big(A^{a,b,c}f\big)(x)-\varphi(a,b,c)\Big)=\sup_{|a|+|b|+|c|\leq r}\Big(\big(A^{a,b,c}f\big)(x)-\varphi(a,b,c)\Big)	
		\end{equation}
        and
        \begin{equation}\label{eq:control.bound2}
			\sup_{(a,b,c)\in \A_n}\Big(\big(A^{a,b,c}f\big)(x)-\varphi_n(a,b,c)\Big)	=\sup_{(a,b,c)\in\A_n^r}\Big(\big(A^{a,b,c}f\big)(x)-\varphi_n(a,b,c)\Big)
		\end{equation}
		for all $n\in\N$ and $x\in X_n$.\ For every $n\in\N$, $(a,b,c)\in\A_n^r$ and $x\in X_n$, it holds
		\begin{align*}
			&f(x+cJ_n+\sqrt{ah_n}\xi_n+bh_n)
			=f(x+cJ_n)+(\sqrt{ah_n}\xi_n+bh_n)^T Df(x+cJ_n) \\
			&\qquad\; +\frac{1}{2}\big(\sqrt{ah_n}\xi_n+bh_n)^T D^2 f(x+cJ_n)\big(\sqrt{ah_n}\xi_n+bh_n)
			+R^{a,b}_n(f,x+cJ_n,\xi_n)
		\end{align*}
		with $|R_n^{a,b}(f,x+cJ_n,\xi_n)|\leq\frac{|\sqrt{ah_n}\xi_n+bh_n|^3}{6}\|D^3 f\|_\infty$.\
		Since $\E[\xi_n]=0$ and $\E[\xi_n\xi_n^T]=I_d$, we obtain
		\begin{align*}
			\frac{(I_n^{a,b,c}f)(x)-f(x)}{h_n}
			&=\E\left[\frac{f(x+cJ_n+\sqrt{ah_n}\xi_n+bh_n)-f(x)}{h_n}\right] \\
			&=\frac{1}{2}\E\big[\tr\big(aD^2f(x+cJ_n)\big)\big]+b^T \E[Df(x+cJ_n)]\\
            &\quad+\lambda\int_{\R^d\setminus\{0\}} f(x+cz)-f(x)\, \mu(\d z)+\E\big[\tilde{R}_n^{a,b,c}(f,x,\xi_n)\big],
		\end{align*}
		where 
        \begin{align*}
        |\tilde{R}_n^{a,b,c}(f,x,\xi_n)|&\leq\frac{1}{2}|b|^2 h_n\|D^2 f\|_\infty+4\lambda_n^2h_n\|f\|_\infty
        +2|\lambda-\lambda_n|\|f\|_\infty \\
        &\quad+\lambda \bigg|\int_{\R^d\setminus\{0\}}f(x+cz)\,\mu(\d z)-\int_{\R^d\setminus\{0\}}f(x+cz)\,\mu_n(\d z)\bigg|\\
        &\quad+\frac1{h_n}\E[|R_n^{a,b}(f,x+cJ_n,\xi_n)|].
        \end{align*}
		Hence, equation~\eqref{eq:control.bound1}, equation~\eqref{eq:control.bound2.1}, equation~\eqref{eq:control.bound2},
		$\sup_{n\in\N}\E[|\xi_n|^3]<\infty$ and Assumption~\ref{ass:control}(v) and~(vi) imply that, for all $f\in\Cbi(\Rd)$,
		\[ \lim_{n\to\infty}\left\|\frac{I_n f-f}{h_n}-\sup_{(a,b,c)\in\S^d_+\times\Rd\times \R^{d\times d}}\big(A^{a,b,c}f-\varphi(a,b,c)\big)\right\|_\infty=0. \]
		Moreover, the previous estimates and Corollary \ref{cor:cutoff2} imply that
		Assumption \ref{ass:cher}(iii) is valid while analogous computations as in~\cite[Subsection~6.1]{BDKN}
		show that the semigroup $(T(t))_{t\geq 0}$ satisfies the conditions of Theorem~\ref{thm:unique}.\
		Hence, the claim follows from Theorem~\ref{thm:cher}, Theorem~\ref{thm:unique} and Corollary~\ref{cor:unique2}.
	\end{proof}

    \begin{remark}
     Similar techniques and estimates as presented in this section can be used to consider space-dependent controls.\ We focus on the case without diffusion and uniformly bounded and Lipschitz continuous coefficients 
     \[
     b^a\colon \R^d\to \R^d\quad\text{and}\quad c^a\colon \R^d\to \R^{d\times d}\quad \text{for all }a\in A,
     \]
     where $A$ is a nonempty action set. We assume that $A$ is a measurable space and that $(a,x)\mapsto b^a(x)$ and $(a,x)\mapsto c^a(x)$ are product measurable.\ Let $\mathcal A$ be the set of all predictable controls $\alpha\colon \Omega\times \R_+\to A$ and consider the optimal control problem
    \begin{align*} 
		(T(t)f)(x):=\sup_{\alpha\in\A} \bigg(\E\bigg[f\big(X_t^{x,\alpha}\big)-\int_0^t \varphi(\alpha_s)\,\d s\bigg]\bigg),
	\end{align*}
    where $\varphi\colon A\to [0,\infty]$ is a measurable function and 
    \[
    \d X_t^{x,\alpha}=b^{\alpha_t}(X_t^{x,\alpha})\,\d t+ c^{\alpha_t}(X_t^{x,\alpha})\,\d J_t,\quad X_0^{x,\alpha}=x,
    \]
    for all $x\in \R^d$ and $\alpha\in \mathcal A$ .\ We then consider an approximation with uniformly bounded and uniformly Lipschitz continuous coefficients
    \[
     b^a_n\colon \R^d\to \R^d\quad\text{and}\quad c_n^a\colon \R^d\to \R^{d\times d}\quad \text{for all }a\in A \text{ and }n\in \N.
     \]
     Assuming that $x+b_n^a(x)h_n+c_n^a(x)\sum_{k=1}^{N_n(\omega)}Z_n^k(\omega)\in X_n$ for all $x\in X_n$ and $n\in \N$ and $a\in A$, we consider the discrete control problem
     \[
     (I_n f)(x):=\sup_{a\in A}
	\Bigg(\E\bigg[f\bigg(x+b_n^a(x)h_n+c_n^a(x)\sum_{k=1}^{N_n} Z_n^k\bigg)\bigg]-\varphi_n(a)h_n\Bigg).
     \]
     Under Assumption \ref{ass:control} (vi) and the assumption that
     \begin{align*}
       \lim_{n\to \infty} \sup_{a\in A}& \bigg( b_n^a(x)^TDf(x)+\int_{\R^d\setminus\{0\}} f\big(x+c_n^a(x)z\big)-f(x)\,\mu(\d z)-\varphi_n(a)\bigg)\\
       &\qquad =  \sup_{a\in A} \bigg( b^a(x)^TDf(x)+\int_{\R^d\setminus\{0\}} f\big(x+c^a(x)z\big)-f(x)\,\mu(\d z)-\varphi(a)\bigg),
     \end{align*}
    uniformly in $x\in \R^d$ for all $f\in \Cbi(\R^d)$,     
     one can show that
     \[
     \lim_{n\to \infty} \frac{I_n f-f}{h_n}=\sup_{a\in A} \bigg( b^a(x)^TDf(x)+\lambda\int_{\R^d\setminus\{0\}} f\big(x+c^a(x)z\big)-f(x)\, \mu(\d z)-\varphi(a)\bigg),
     \]
     uniformly in $x\in \R^d$ for all $f\in \Cbi(\R^d)$.\ The key estimates are
          \begin{align*}
      \bigg|\frac{f\big(x+c_n^a(x)J_n+b_n^a(x)h_n\big)-f\big(x+c_n^a(x)J_n\big)}{h_n}-&b_a(x)^TDf\big(x+c_n^a(x)J_n\big)\bigg|\\
      &\qquad\qquad \leq \frac12 h_n|b_n^a(x)|^2 \|D^2f\|_\infty
     \end{align*}
     and
     \begin{align*}
      &\bigg|\frac{f\big(x+c_n^a(x)J_n\big)-f(x)}{h_n}-\lambda\int_{\R^d\setminus\{0\}} f\big(x+c_n^a(x)z\big)-f(x)\,\mu(\d z)\bigg|
      \leq 4\lambda_n^2 h_n\|f\|_\infty\\
      & +2|\lambda-\lambda_n|\|f\|_\infty + \lambda\bigg|\int_{\R^d\setminus\{0\}} f\big(x+c_n^a(x)z\big)\,\mu(\d z)-\int_{\R^d\setminus\{0\}} f\big(x+c_n^a(x)z\big)\,\mu_n(\d z)\bigg|.
     \end{align*}
     Using the uniform boundedness of the coefficients $b_n^a$ and $c_n^a$ in $a\in A$ and $n\in \N$, one can again obtain uniform bounds for the right-hand sides.
     \end{remark}}
	
We conclude this section with an alternative approximation for convex HJB equations based on finite differences.\ For simplicity, 
	we focus on the one-dimensional case with a controlled Brownian motion and refer to~\cite{Krylov05, Krylov98, BJ07, DK01, BZ03} for multi-dimensional diffusions with Lipschitz coefficients.

    \begin{remark}\label{rem.finite.differences}
    Let $(\delta_n)_{n\in\N}$,
	$(h_n)_{n\in\N}$ and $(\sigma_n)_{n\in\N}$ be sequences in $(0,\infty)$
	with $\delta_n\to 0,\; h_n\to 0$ and $\sigma_n\to\infty$
    and let 
	$\varphi\colon\R_+\to [0,\infty]$ be a function with $\lim_{\sigma\to\infty}\frac{\varphi(\sigma^2)}{\sigma^2}=\infty$ such that, for every $n\in\N$, there exists $\sigma\in [0, \sigma_n]$ with $\varphi(\sigma^2)=0$.\ Moreover, we assume that
    \begin{equation}\label{eq:finite.diff}
    \frac{\sigma_n^2 h_n}{\delta_n^2}\leq 1\quad\text{for all }n\in \N.
    \end{equation}
    and consider the finite-difference scheme 
	\[ (I_n f)(x):=f(x)+h_n\sup_{\sigma\in[0,\sigma_n]} 
	\left(\frac{\sigma^2}{2}\frac{f(x+\delta_n)-2f(x)+f(x-\delta_n)}{\delta_n^2}-\varphi(\sigma^2)\right), \]
	defined for all $n\in\N$, $f\in\Cb(\R)$ and $x\in\R$.\ We point out that the condition in \eqref{eq:finite.diff} is classical and guarantees that the finite-difference scheme is stable 
	with respect to the supremum norm.\ Employing similar arguments as in the proof of Theorem \ref{thm:control}, one can show that there exists a strongly 
		continuous convex monotone semigroup $(S(t))_{t\geq 0}$ on $\Cb(\R)$ with generator $A\colon D(A)\to \Cb(R)$, 
        given by 
		\[ S(t)f=\lim_{n\to\infty}I_n^{k_n^t}f \quad\mbox{for all } (f,t)\in\Cb(\R)\times\R_+, \]
		where $(k_n^t)_{n\in\N}\subset\N$ is an arbitrary sequence satisfying $k_n^t h_n\to t$.
		It holds
		\[ \Cbi(\R)\subset D(A) \quad\mbox{and}\quad
		Af=\sup_{\sigma\geq 0}\Big(\frac{1}{2}\sigma^2 f''-\varphi(\sigma^2)\Big) 
		\quad\mbox{for all } f\in\Cbi(\R).\]
        Choosing $J_t=0$, $\varphi(a,0)=\varphi(\sigma^2)$ with $a=\sigma^2$ and $\varphi(a,b)=\infty$ for $b\neq 0$ in the definition of $T(t)f$, Theorem \ref{thm:unique}(c) yields that $$S(t)f=T(t)f\quad \text{for all }t\geq 0\text{ and }f\in \Cb(\R).$$ 
  \end{remark}

\appendix

\section{A version of Arzel\`a Ascoli's theorem}
\label{app:A}

Let $X_n\subset\Rd$ such that, for every $x\in\Rd$, there exist $x_n\in X_n$ with $x_n\to x$.\ 
A sequence $(f_n)_{n\in\N}$ of functions $f_n\colon X_n\to\R$ is called bounded if 
$\sup_{n\in\N}\|f_n\|_{\kappa, X_n}<\infty$. 
The proof of the following lemma is similar to~\cite[Lemma~D.1]{BDKN} and 
therefore omitted.

\begin{lemma} \label{lem:AA}
 Let $(f_n)_{n\in\N}$ be a sequence of functions $f_n\colon X_n\to\R$ which is uniformly 
 equicontinuous and bounded.\ Then, there exist $f\in\Ck(\Rd)$ and a subsequence 
 $(n_l)_{l\in\N}\subset\N$ such that, for every $K\Subset\Rd$ with $K\cap X_{n_l}\neq \emptyset$ for all $l\in\N$,
 \[ \lim_{l\to\infty}\|f-f_{n_l}\|_{\infty, K_{n_l}}=0. \]
\end{lemma}

	\section{Basic convexity estimates}
	\label{app:B}

	\begin{lemma} \label{lem:lambda}
		Let $V$ be a vector space and $\phi\colon V\to\R$ be a convex functional. Then, 
		\[ \phi(v)-\phi(w)\leq\lambda\left(\phi\left(\frac{v-w}{\lambda}+w\right)-\phi(w)\right)
		\quad\mbox{for all } v,w\in V \mbox{ and } \lambda\in (0,1]. \]
	\end{lemma}
	\begin{proof}
		For every $v,w\in V$ and $\lambda\in (0,1]$, 
		\begin{align*}
			&\phi(v)-\phi(w)
			=\phi\left(\lambda\left(\frac{v-w}{\lambda}+w\right)+(1-\lambda)w\right)-\phi(w) \\
			&\leq\lambda\phi\left(\frac{v-w}{\lambda}+w\right)+(1-\lambda)\phi(w)-\phi(w) =\lambda\left(\phi\left(\frac{v-w}{\lambda}+w\right)-\phi(w)\right). 
		\end{align*}
	\end{proof}

    \begin{lemma} \label{lem:kappa}
 Let $X\subset\Rd$ and $\Phi$ be a convex operator $\Phi\colon\Ck(X)\to\Ck(X)$ such that there exists $c\geq 0$ with
 $\|\Phi(f)\|_{\kappa,X}\leq c\|f\|_{\kappa,X}$ for all $f\in\Ck(X)$. Then,
 \[ \Phi\left(f+\frac{a}{\kappa}\right)\leq\Phi(f)+\frac{c|a|}{\kappa}
 	\quad\mbox{for all } f\in\Ck \mbox{ and } a\in\R. \]
\end{lemma}
\begin{proof}
 Let $f\in\Ck(X)$ and $a\in\R$. For every $\lambda\in (0,1)$,
 \begin{align*}
  \Phi\left(f+\frac{a}{\kappa}\right)
  &\leq\lambda\Phi\left(\frac{1}{\lambda}f\right)+(1-\lambda)\Phi\left(\frac{a}{(1-\lambda)\kappa}\right) \\
  &\leq\lambda\Phi\left(\frac{1}{\lambda}f\right)+(1-\lambda)\frac{c|a|}{(1-\lambda)\kappa}
  =\lambda\Phi\left(\frac{1}{\lambda}f\right)+\frac{c|a|}{\kappa}.
 \end{align*}
 In addition, for every $x\in\Rd$, the mapping $\R\to\R,\; t\mapsto (\Phi(tf))(x)$ is convex and 
 therefore continuous. In particular, we obtain $\lambda\Phi\left(\frac{1}{\lambda}f\right)\to\Phi(f)$ 
 as $\lambda\to 1$.
\end{proof}

\section{Continuity from above}
	\label{app:cont}

	Let $\ca^+_\kappa(\Rd)$ be the set of Borel measures $\mu\colon\B(\Rd)\to [0,\infty]$ 
	with $\int_{\Rd}\frac{1}{\kappa}\,\d\mu<\infty$. Moreover, the convex conjugate of
	a functional $\phi\colon\Ck(\Rd)\to\R$ is defined by 
	\[ \phi^*\colon\ca^+_\kappa(\Rd)\to [0,\infty],\; 
	\mu\mapsto\sup_{f\in\Ck(\Rd)}\big(\mu f-\phi(f)\big),
	\quad\mbox{where}\quad \mu f:=\int_{\Rd}f\,\d\mu. \]
	In the sequel, let $(X_i)_{i\in I}$ be a family of closed sets $X_i\subset \Rd$ 
	and $(\phi_i)_{i\in I}$ be a family of convex monotone functionals $\phi_i\colon\Ck(X_i)\to\R$ 
	with $\phi_i(0)=0$ and 
	\begin{equation} \label{eq:bound.app1}
		\sup_{i\in I}\sup_{f\in B_{\Ck(X_i)}(r)}|\phi_i(f)|<\infty \quad\mbox{for all } r\geq 0. 
	\end{equation}
	We define $\phi_i(f):=\phi_i(f|_{X_i})$ and $K_i:=K\cap X_i$ for all 
	$i\in I$, $f\in\Ck(\Rd)$ and $K\Subset\Rd$.

	\begin{lemma} \label{lem:cont.app}
		The following two statements are equivalent:
		\begin{enumerate}
			\item[(i)] It holds $\sup_{i\in I}\phi_i(f_n)\downarrow 0$ for all sequences $(f_n)_{n\in\N}\subset\Ck(\Rd)$
			with $f_n\downarrow 0$. 
			\item[(ii)] For every $\epsilon>0$ and $r\geq 0$, there exist $c\geq 0$ and $K\Subset\Rd$ with 
			\[ \sup_{i\in I}|\phi_i(f)-\phi_i(g)|\leq c\|f-g\|_{\infty, K_i}+\epsilon \]
			for all $i\in I$ and $f,g\in B_{\Ck(X_i)}(r)$. 
		\end{enumerate}
	\end{lemma}
	\begin{proof}
		First, suppose that condition~(i) is satisfied.\ Let $\epsilon>0$ and $r\geq 0$.\ Since
		$\phi_i\colon\Ck(\Rd)\to\R$ is continuous from above, we can apply~\cite[Theorem~C.1]{BDKN}
		to obtain
		\begin{equation} \label{eq:dual}
			\phi_i(f)=\max_{\mu\in M_i}\big(\mu f-\phi_i^*(\mu)\big) 
			\quad\mbox{for all } i\in I \mbox{ and } f\in B_{\Ck(\Rd)}(r), 
		\end{equation}
		where 
		$M_i:=\{\mu\in\ca_\kappa^+(\Rd)\colon\phi_i^*(\mu)\leq\phi_i(\tfrac{2r}{\kappa})-2\phi_i(-\tfrac{r}{\kappa})\}$.
		Furthermore, it holds 
		\begin{equation} \label{eq:Mi}
			\mu(X_i^c)=0 \quad\mbox{for all } \mu\in M_i \mbox{ and } i\in I.
		\end{equation}
		Indeed, by contradiction, suppose that there exists $\mu\in M_i$ with $\mu(X_i^c)>0$.\ Then,
		due to Ulam's theorem, there exists $K\Subset X_i^c$ with $\mu(K)>0$.\ Moreover, by 
		Urysohn's lemma, there exists a continuous function $f\colon\Rd\to [0,1]$ with $f(x)=0$ for all $x\in X_i$ 
		and $f(x)=1$ for all $x\in K$. We use $\phi_i(\lambda f)=\phi_i(\lambda f|_{X_i})=\phi_i(0)=0$ to conclude 
		\[ \phi^\ast_i(\mu)\geq\sup_{\lambda\geq 0}\big(\mu(\lambda f)-\phi_i(\lambda f)\big)
		\geq\sup_{\lambda\geq 0}\lambda\mu(K)=\infty. \]
		This contradicts the fact that $\mu\in M_i$ guarantees $\phi_i^*(\mu)<\infty$. Next, we show that
		\begin{equation} \label{eq:dual2}
			\phi_i(f)=\max_{\mu\in M_i}\big(\mu f-\phi_i^*(\mu)\big) 
			\quad\mbox{for all } i\in I \mbox{ and } f\in B_{\Ck(X_i)}(r),
		\end{equation}
		where $\mu f=\int_{X_i}f\,\d\mu$ is well-defined by equation~\eqref{eq:Mi}.\ Let $i\in I$ and 
		$f\in B_{\Ck(X_i)}(r)$. Since $X_i\subset\Rd$ is closed and $f\kappa\in\Cb(X_i)$, by Tietze's 
		extension theorem, there exists a function $\tilde{g}\in\Cb(\Rd)$ with $(f\kappa)(x)=\tilde{g}(x)$ 
		for all $x\in X_i$ and $\|f\kappa\|_{\infty, X_i}=\|\tilde{g}\|_\infty$. Consequently, the function 
		$g:=\frac{1}{\kappa}\tilde{g}\in\Ck(\Rd)$ satisfies $f(x)=g(x)$ for all $x\in X_i$ and 
		$\|f\|_{\kappa, X_i}=\|g\|_\kappa$. It follows from equation~\eqref{eq:dual} and
		equation~\eqref{eq:Mi} that
		\[ \phi_i(f)=\phi_i(g)=\max_{\mu\in M_i}\big(\mu g-\phi_i^*(\mu)\big) 
		=\max_{\mu\in M_i}\big(\mu f-\phi_i^*(\mu)\big). \]
		Condition~(i) implies that $\phi\colon\Ck(\Rd)\to\R$, $f\mapsto\sup_{i\in I}\phi_i(f)$
		is  convex, monotone and continuous from above with $\phi(0)=0$.\ Hence, 
        condition~\eqref{eq:bound.app1} and~\cite[Theorem~2.2]{BCK} guarantee that the set
		\[ M:=\big\{\mu\in\ca_\kappa^+(\Rd)\colon\phi^*(\mu)\leq\sup_{i\in I}
		  \big(\phi_i\big(\tfrac{2r}{\kappa}\big)-2\phi_i\big(-\tfrac{r}{\kappa}\big)\big)\big\} \]
		is $\sigma(\ca_\kappa^+(\Rd), \Ck(\Rd))$-relatively compact.\ Due to Prokhorov's theorem,
		there exists $K\Subset\Rd$ with 
		$\sup_{\mu\in M}\int_{K^c}\frac{1}{\kappa}\,\d\mu\leq\frac{\epsilon}{2r}$. 
		We use equation~\eqref{eq:dual2} and 
		$M_i\subset M$ to obtain
		\begin{align*}
			&|\phi_i(f)-\phi_i(g)| 
			\leq\sup_{\mu\in M_i}|\mu f-\mu g| 
			\leq\sup_{\mu\in M_i}\left(\int_K |f-g|\,\d\mu+\int_{K^c}|f-g|\,\d\mu\right) \\
			&\leq\sup_{\mu\in M_i}\left(\mu(K)\|f-g\|_{\infty, K_i}+\int_{K^c}\frac{2r}{\kappa}\,\d\mu\right) 
			\leq c\|f-g\|_{\infty, K_i}+\epsilon
		\end{align*}
		for all $i\in I$, $f,g\in B_{\Ck(X_i)}(r)$ and $c:=\sup_{\mu\in M}\mu(K)\leq\phi(1)+\sup_{\mu\in M}\phi^*(\mu)<\infty$. 
		
		Second, suppose that condition~(ii) is valid.\ Let $(f_n)_{n\in\N}\subset\Ck(\Rd)$
		with $f_n\downarrow 0$ and $r:=\|f_1\|_\kappa$. 
		For every $\epsilon>0$, there exist $c\geq 0$ and $K\Subset\Rd$ with 
		\[ \sup_{i\in I}\phi_i(f_n)\leq c\|f_n\|_{\infty, K}+\tfrac{\epsilon}{2}
            \quad\mbox{for all } n\in\N. \]
		By Dini's theorem, there exists $n_0\in\N$ with $\sup_{i\in I}\phi_i(f_n)\leq\epsilon$ 
		for all $n\geq n_0$. 
	\end{proof}

	Now, let $(\Phi_i)_{i\in I}$ be a family of convex monotone operators $\Phi_i\colon\Ck(X_i)\to\Fk(X_i)$
	with $\Phi_i(0)=0$ and 
	\begin{equation} \label{eq:bound.app2}
		\sup_{i\in I}\sup_{f\in B_{\Ck(X_i)}(r)}\|\Phi_i(f)\|_\kappa<\infty \quad\mbox{for all } r\geq 0. 
	\end{equation}
	Here, the space $\Fk(X_i)$ consists of all functions $f\colon X_i\to\R$ with $\|f\|_\kappa<\infty$.

	\begin{corollary} \label{cor:cont.app}
		The following two statements are equivalent:
		\begin{enumerate}
			\item[(i)] It holds $\sup_{i\in I}\sup_{x\in K_i}\,(\Phi_i f_n)(x)\downarrow 0$ for all $f_n\downarrow 0$
			and $K\Subset\Rd$. 
			\item[(ii)] For every $\epsilon>0$, $r\geq 0$ and $K\Subset\Rd$, there exist $c\geq 0$ 
			and $K'\Subset\Rd$ with
			\[ \|\Phi_i f-\Phi_i g\|_{\infty,K_i}\leq c\|f-g\|_{\infty,K'_i}+\epsilon
			\quad\mbox{for all } f,g\in B_{\Ck(X_i)}(r). \]
		\end{enumerate}
	\end{corollary}
	\begin{proof}
		Set $\tilde{I}:=\{(i, x)\colon i\in I,x\in K_i\}$, $\tilde{X}_{(i,x)}:=X_i$ and $\tilde{\phi}_{(i, x)}(f):=(\Phi_i f)(x)$.
	\end{proof}

	\begin{corollary} \label{cor:moment.app}
		Let $\tilde{\kappa}\colon\Rd\to (0,\infty)$ be a bounded continuous function 
		such that, for every $\epsilon>0$, there exists $K\Subset \Rd$ with 
		$\sup_{x\in K^c}\frac{\tilde{\kappa}(x)}{\kappa(x)}\leq\epsilon$.\ Furthermore, we
		assume that $\sup_{i\in I}\Phi_i(f_n)\to 0$ for all $(f_n)_{n\in\N}\subset\Ck(\Rd)$
		with $\|f_n\|_{\tilde{\kappa}}\to 0$.\ Then, for every $\epsilon>0$, $r\geq 0$ and 
        $K\Subset\Rd$, there exist $c\geq 0$ and $K'\Subset\Rd$ with
		\[ \|\Phi_i f-\Phi_i g\|_{\infty,K_i}\leq c\|f-g\|_{\infty,K_i'}+\epsilon
		\quad\mbox{for all } i\in I \mbox{ and } f,g\in B_{\Ck(X_i)}(r). \]
	\end{corollary}
	\begin{proof}
		We verify condition~(i) from Corollary~\ref{cor:cont.app}.\ Let  
		$f_n\downarrow 0$, $\epsilon>0$ and $K\Subset\Rd$ with 
		$\sup_{x\in K^c}\frac{\tilde{\kappa}(x)}{\kappa(x)}\leq\epsilon$.\ 
		By Dini's theorem, there exists $n_0\in\N$ with $f_n(x)\leq\epsilon$ for all
		$n\geq n_0$ and $x\in K$. Hence, for every $n\geq n_0$, 
		\[ \|f_n\|_{\tilde{\kappa}}
		=\sup_{x\in K}|f_n(x)|\tilde{\kappa}(x)
		+\sup_{x\in K^c}|f_n(x)|\kappa(x)\frac{\tilde{\kappa}(x)}{\kappa(x)} 
		\leq\Big(\sup_{x\in K}\tilde{\kappa}(x)+\|f_n\|_\kappa\Big)\epsilon. \]
		We obtain $\sup_{i\in I}\|\Phi_i f_n\|_{\tilde{\kappa}}\to 0$ and 
		$\sup_{i\in I}\sup_{x\in K_i}(\Phi_i f_n)(x)\downarrow 0$ for all $K\Subset\Rd$. 
	\end{proof}

	\begin{corollary} \label{cor:cutoff.app}
		Assume that, for every $\epsilon>0$, $r\geq 0$ and $K\Subset\Rd$, there exist a family
		$(\zeta_x)_{x\in K}$ of continuous functions $\zeta_x\colon\Rd\to\R$ and 
		$\tilde{K}\Subset\Rd$ with 
		\begin{enumerate}
			\item[(i)] $0\leq\zeta_x\leq 1$ for all $x\in K$,
			\item[(ii)] $\sup_{y\in\tilde{K}^c}\zeta_x(y)\leq\epsilon$ for all $x\in K$,
			\item[(iii)] $\big(\Phi_i\big(\frac{r}{\kappa}(1-\zeta_x)\big)\big)(x)\leq\epsilon$ for all $i\in I$
			and $x\in K_i$.
		\end{enumerate}
		Then, for every $\epsilon>0$, $r\geq 0$ and $K\Subset\Rd$, there exist $c\geq 0$ and 
		$K'\Subset\Rd$ with
		\[ \|\Phi_i f-\Phi_i g\|_{\infty,K_i}\leq c\|f-g\|_{\infty,K_i'}+\epsilon
		\quad\mbox{for all } i\in I \mbox{ and } f,g\in B_{\Ck(X_i)}(r). \]
	\end{corollary}
	\begin{proof}
		We verify condition~(i) from Corollary~\ref{cor:cont.app}.\ Let 
		 $f_n\downarrow 0$, $K\Subset\Rd$, $\epsilon\in (0,1]$ and $r:=2\|f_1\|_\kappa+1$.\
		Condition~\eqref{eq:bound.app2} yields
		$c:=\sup_{i\in I}\sup_{x\in K_i}\big(\Phi_i\big(\tfrac{r}{\kappa}\big)\big)(x)+1<\infty$.
		By assumption, there exist a family $(\zeta_x)_{x\in K}$ of continuous functions 
		$\zeta_x\colon\Rd\to\R$ and $\tilde{K}\Subset\Rd$ such that the conditions~(i)-(iii) are
		satisfied with $\nicefrac{\epsilon}{c}$.\ For every $i\in I$ and $x\in K_i$, the convexity and monotonicity 
		of $\Phi_i$, condition~(i) and condition~(iii) imply
		\begin{align*}
			(\Phi_i f_n)(x) 
			&\leq\tfrac{1}{2}\big(\Phi_i(2f_n\zeta_x)\big)(x)+\tfrac{1}{2}\big(\Phi_i(2f_n(1-\zeta_x))\big)(x) \leq\tfrac{1}{2}\big(\Phi_i(2f_n\zeta_x)\big)(x)+\tfrac{\epsilon}{2}.
		\end{align*}
		By the conditions~(i) and~(ii) and Dini's theorem, there exists $n_0\in\N$ with 
		$2f_n\zeta_x\leq\frac{r\epsilon}{c\kappa}$ for all $n\geq n_0$ and $x\in K$.\
		Since $\Phi_i$ is convex and monotone with $\Phi_i(0)=0$, we obtain
		\[ \big(\Phi_i(2f_n\zeta_x)\big)(x)\leq\big(\Phi_i\big(\tfrac{r\epsilon}{c\kappa}\big)\big)(x)
		\leq\tfrac{\epsilon}{c}\big(\Phi_i\tfrac{r}{\kappa}\big)(x)\leq\epsilon \]
        for all $i\in I$, $x\in K_i$ and $n\geq n_0$
	\end{proof}

\end{document}